\documentclass[12pt,a4paper,reqno,twoside]{amsart}
\usepackage[utf8]{inputenc}
\usepackage[T1]{fontenc}
\usepackage{indentfirst}
\usepackage{textcomp}
\usepackage[english]{babel}
\usepackage{csquotes}
\usepackage{enumitem}
\usepackage{amsmath,amsfonts,amssymb,amsopn,amscd,amsthm}
\usepackage{mathrsfs,shuffle,bbm,multirow}
\usepackage{stmaryrd}
\usepackage{mathpazo}
\usepackage{graphicx}
\usepackage[dvipsnames]{xcolor}
\usepackage[colorlinks=true,citecolor=DarkOrchid,linkcolor=NavyBlue]{hyperref}
\usepackage{tikz}
\usepackage{mathtools}
\usepackage[a4paper,scale=0.81,twoside]{geometry}

\usepackage{bm}

\setlist[enumerate]{itemsep=10pt,topsep=10pt}
\setlist[itemize]{itemsep=5pt,topsep=5pt}

\newtheorem{definition}{Definition}
\newtheorem{lemma}[definition]{Lemma}
\newtheorem{theorem}[definition]{Theorem}
\newtheorem{proposition}[definition]{Proposition}

\theoremstyle{remark}
\newtheorem*{example}{Example}
\newtheorem*{remark}{Remark}
	
\newcommand{\N}{\mathbb{N}}
\newcommand{\Z}{\mathbb{Z}}
\newcommand{\R}{\mathbb{R}}
\def\C{\mathbb{C}}
\def\GG{\bm{G}}
\newcommand{\E}{\mathrm{e}}
\newcommand{\I}{\mathrm{i}}
\newcommand{\esper}{\mathbb{E}}
\newcommand{\proba}{\mathbb{P}}
\newcommand{\gauss}{\mathcal{N}_{\R}(0,1)}
\newcommand{\var}{\mathrm{var}}
\newcommand{\cov}{\mathrm{cov}}
\newcommand{\sym}{\mathfrak{S}}
\newcommand{\pym}{\mathfrak{P}}
\newcommand{\leb}{\mathrm{L}}
\newcommand{\dkol}{d_{\mathrm{Kol}}}

\newcommand{\eps}{\varepsilon}

\newcommand{\lle}{\left[\!\left[} 
\newcommand{\rre}{\right]\!\right]} 
\newcommand{\join}{\bowtie}
\renewcommand{\Re}{\mathrm{Re}}

\newcommand{\Fcal}{\mathscr{F}}
\newcommand{\Gcal}{\mathscr{G}}
\newcommand{\Ccal}{\mathscr{C}}
\newcommand{\Scal}{\mathscr{S}}
\newcommand{\Pcal}{\mathscr{P}}
\newcommand{\Mcal}{\mathscr{M}}
\newcommand{\obs}{\mathscr{O}}
\newcommand{\obsG}{\mathscr{O}_{\GG}}
\newcommand{\obsS}{\mathscr{O}_{\mathfrak{S}}}
\newcommand{\obsP}{\mathscr{O}_{\mathfrak{P}}}

\newcommand{\conf}{\mathrm{conf}}
\newcommand{\DD}[1]{\,d\hspace*{-0.3mm}{#1}}

\newcommand{\comment}[1]{}

\author{V. F\'eray \and P.-L. M\'eliot \and A. Nikeghbali}
\title[Graphons, permutons and the Thoma simplex: three mod-Gaussian moduli spaces]{Graphons, permutons and the Thoma simplex:\\ three mod-Gaussian moduli spaces}

\begin{document}
\begin{abstract}
In this paper, we show how to use the framework of mod-Gaussian convergence in order to study the fluctuations of certain models of random graphs, of random permutations and of random integer partitions. We prove that,
in these three frameworks, a generic homogeneous observable of a generic random model is mod-Gaussian under an appropriate renormalisation. This implies a central limit theorem with an extended zone of normality, a moderate deviation principle, an estimate of the speed of convergence, a local limit theorem and a concentration inequality. The universal asymptotic behavior of the observables of these models gives rise to a notion of mod-Gaussian moduli space.
\end{abstract}
\maketitle

\tableofcontents

\section{Introduction}
The purpose of this paper is to show that many random models stemming from combinatorics exhibit asymptotic fluctuations which one can study in the framework of \emph{mod-Gaussian convergence}. We start our introduction by a discussion of this notion and of the probabilistic estimates which follow from it.\medskip

\subsection{Mod-Gaussian convergence, cumulants and dependency graphs}
Let $(X_n)_{n \in \N}$ be a sequence of real random variables. The notion of mod-Gaussian convergence was introduced in \cite{JKN11},
and later developed in \cite{DKN15,MN15,FMN16,FMN17,DBMN17}. 
Mod-Gaussian convergence yields precise quantitative information about 
the convergence in distribution of an appropriate renormalisation of $X_n$ towards the standard Gaussian law:
we get normality zone estimates, precise moderate deviations, bounds on the speed of convergence, or local limit theorems. 
\medskip

Though the original definition only involved the Fourier transforms $\esper[\E^{\I\xi X_n}]$, in this paper,
we shall work with complex exponential generating functions (Laplace transforms): 

\begin{definition}[Mod-Gaussian convergence]\label{def:modgauss}
The sequence of random variables $(X_n)_{n \in \N}$ \emph{converges in the mod-Gaussian sense} with parameters $t_n \to +\infty$ and limit $\psi(z)$ if, locally uniformly on the complex plane,
$$ \psi_n(z) \coloneqq \esper[\E^{zX_n}]\,\E^{-\frac{t_n\,z^2}{2}} \to_{n \to \infty} \psi(z),$$
where $\psi(z)$ is a continuous function with $\psi(0)=1$.
\end{definition}

In \cite{FMN16}, the local uniform convergence $\psi_n \to \psi$ was only asked to occur on a band $\mathcal{S}_{(c,d)}=\{z\in \C\,|\,c < \Re(z)<d\}$), but in all instances of mod-Gaussian convergence in this paper, the convergence holds on the whole complex plane,
so we restrict to this case here.
In particular, we only consider random variables $X_n$ with an {\em entire} exponential generating function. 
\medskip

One of the ways to prove mod-Gaussian convergence
consists in estimating the coefficients of the Taylor expansion of $\log(\esper[\E^{zX_n}])$, which are called the \emph{cumulants} of $X_n$.
Formally, if  
$$\log \esper[\E^{zS_n}] = \sum_{r=1}^\infty \frac{\kappa^{(r)}(S_n)}{r!}\,z^r$$
for $|z|$ sufficiently small, then the coefficient $\kappa^{(r)}(S_n)$
is called {\em cumulant of order $r$ of $S_n$}.
In \cite{FMN16,FMN17}, we explain how estimates on cumulants can be used to 
prove mod-Gaussian convergence:
\begin{definition}[Method of cumulants]\label{def:cumulantmethod}
Let $(S_n)_{n \in \N}$ be a sequence of real random variables. 
 Let $A$ be a positive constant, and $(D_n)_{n \in \N}$ and $(N_n)_{n \in \N}$ be two positive sequences such that $\lim_{n \to \infty} \frac{D_n}{N_n}=0$. We say that the sequence $(S_n)_{n \in \N}$ satisfies the hypotheses of the method of cumulants with parameters $(D_n,N_n,A)$ and limits $(\sigma^2,L)$ if:
\begin{itemize}
	\item For any $r \geq 1$, we have the uniform bounds on cumulants (see \cite[Definition 28]{FMN17}):
\begin{equation}
|\kappa^{(r)}(S_n)|\leq N_n\,(2D_n)^{r-1}\,r^{r-2}\,A^r.\tag{MC1}\label{eq:cumulant1}
\end{equation} 
\item There exists two real numbers $\sigma^2\geq 0$ and $L$ such that:
\begin{align}
\frac{\kappa^{(2)}(S_n)}{N_n\,D_n} &= (\sigma_n)^2 = \sigma^2\left(1+o\!\left(\left(\frac{D_n}{N_n}\right)^{\!1/3}\right)\right) ;\tag{MC2}\label{eq:cumulant2}\\
\frac{\kappa^{(3)}(S_n)}{N_n\,(D_n)^2} &= L_n = L\,(1+o(1)).\tag{MC3}\label{eq:cumulant3}
\end{align}
In particular, $\sigma_n \to_{n \to \infty} \sigma$ and $L_n \to_{n \to \infty} L$.
\end{itemize}
\end{definition}

\noindent In the method of cumulants, the hypotheses \eqref{eq:cumulant2} and \eqref{eq:cumulant3} are usually proved by an \emph{ad~hoc} computation of the first moments of $S_n$, and by using the identities
\begin{align*}
\kappa^{(2)}(X) &= \var(X) = \esper[X^2] - \esper[X]^2 = \esper\!\left[\overline{X}^2\right] ;\\
\kappa^{(3)}(X) &= \esper[X^3] - 3\,\esper[X^2]\,\esper[X] + 2 (\esper[X])^3 = \esper\!\left[\overline{X}^3\right], 
\end{align*}
where $\overline{X} = X - \esper[X]$. 
On the other hand, we shall explain in a moment how to obtain uniform bounds on cumulants
from {\em sparse dependency graphs}. Let us first detail the probabilistic consequences of the method of cumulants.
Recall that the Kolmogorov distance between two random variables (or their distribution)
is $\dkol(X,Y)=\sup_{t \in \R} |F_X(t)-F_Y(t)|$,
where $F_X$ and $F_Y$ are the cumulative distribution functions of $X$ and $Y$.
\begin{theorem}[Estimates from the method of cumulants]\label{thm:cumulantestimates}
Let $(S_n)_{n \in \N}$ be a sequence of random variables that satisfies the hypotheses of the method of cumulants, with parameters $(D_n,N_n,A)$ and limits $(\sigma^2,L)$. We suppose that $\sigma^2>0$, and we set:
\begin{align*}
X_n &= \frac{S_n-\esper[S_n]}{(N_n)^{1/3}\,(D_n)^{2/3}} \qquad;\qquad
Y_n = \frac{S_n-\esper[S_n]}{\sigma_n\,\sqrt{N_n\,D_n}} = \frac{S_n-\esper[S_n]}{\sqrt{\var(S_n)}}.
\end{align*}
\begin{enumerate}[start = 0]
	\item The random variables $(X_n)_{n \in \N}$ converge in the mod-Gaussian sense, with parameters $t_n = (\sigma_n)^2(\frac{N_n}{D_n})^{1/3}$ and limit $\psi(z) = \exp(\frac{Lz^3}{6})$.

	\item Central limit theorem: we have the convergence in distribution $Y_n \rightharpoonup_{n \to \infty} \gauss$.
	\item Extended zone of normality: if $y_n= o((\frac{N_n}{D_n})^{1/6})$, then
	$$\proba[Y_n \geq y_n] = \proba[\gauss \geq y_n]\,(1+o(1)).$$
	\item Moderate deviations (see \cite[Equation (9.10)]{FMN16}): for any sequence $(y_n)_{n \in \N}$ that tends to $+\infty$ but is a $o((\frac{N_n}{D_n})^{1/4})$,
	$$ \proba[Y_n \geq y_n] = \frac{\E^{-\frac{(y_n)^2}{2}}}{y_n\sqrt{2\pi}}\,\,\exp\!\left(\frac{L}{6\sigma^3}\sqrt{\frac{D_n}{N_n}}\,(y_n)^3\right)\,(1+o(1)).$$
	By symmetry, we have similar results for the negative deviations.
	\item The Kolmogorov distance between $Y_n$ and the standard Gaussian law is bounded by
	$$\dkol(Y_n,\gauss) \leq \frac{76.36\,A^3}{(\sigma_n)^3}\,\sqrt{\frac{D_n}{N_n}},$$
	see \cite[Corollary 30]{FMN17}.

	\item We have the following local limit theorem: for any exponent $\delta \in (0,\frac{1}{2})$, any $y \in \R$ and any Jordan measurable set $B$ with Lebesgue measure $m(B)>0$,
	$$\lim_{n \to \infty}  \left(\frac{N_n}{D_n}\right)^{\!\delta}\,\proba\!\left[Y_n - y \in \left(\frac{D_n}{N_n}\right)^{\!\delta}B\right] = \frac{1}{\sqrt{2\pi}}\,\E^{-\frac{y^2}{2}}\,m(B),$$
	see \cite[Theorem 9 and Proposition 20]{DBMN17}.
\end{enumerate}
\end{theorem}

\noindent It is important for the sequel to note that we did not assume $\sigma^2>0$ in Definition \ref{def:cumulantmethod}, but that the positivity is needed for the probabilistic estimates. On the other hand, all the results but the moderate deviation estimates stay true if one assumes only the convergences $\sigma_n \to \sigma$ and $L_n \to L$ (instead of \eqref{eq:cumulant2}). 
The stronger hypothesis 
$$ (\sigma_n)^2 = \sigma^2\left(1+o\!\left(\left(\frac{D_n}{N_n}\right)^{\!1/3}\right)\right) $$
is only needed in order to get the moderate deviation estimates when $y_n = o((\frac{D_n}{N_n})^{\frac{1}{6}+\eps})$ with $\eps \in (0,\frac{1}{12})$ --- when $\eps=0$, the convergence $\sigma_n \to \sigma$ is sufficient.
\bigskip

Theorem \ref{thm:cumulantestimates} shows that the method of cumulants yields much more information than a classical central limit theorem.
In \cite[Chapters 9-11]{FMN16} and \cite[Sections 4-5]{FMN17}, 
we developed tools to prove the uniform bounds on cumulants given by Equation \eqref{eq:cumulant1}.
In this paper, we shall use the following criterion, see \cite[Theorems 9.1.7 and 9.1.8]{FMN16}.
\begin{definition}[Dependency graphs]
Let $S=\sum_{v \in V} A_v$ be a finite sum of real random variables. We say that a (undirected, simple) graph $G=(V,E)$ is a dependency graph for the family of random variables $(A_v)_{v \in V}$ if the following condition holds: given two disjoint subsets $V_1,V_2 \subset V$, if there is no edge $e = (v,w) \in E$ such that $v \in V_1$ and $w \in V_2$, then $(A_v)_{v \in V_1}$ and $(A_w)_{w \in V_2}$ are independent vectors.
\end{definition}
By abuse of language, we often say that $G$ is a dependency graph for the sum \hbox{$S=\sum_{v \in V} A_v$}
(instead of a dependency graph for the family $(A_v)_{v \in V}$).
\begin{theorem}[Uniform bounds on cumulants from a dependency graph]
Let $S=\sum_{v \in V} A_v$ be a sum of random variables with $$|A_v| \leq A\quad \text{almost surely}$$
for any $v \in V$. We assume that $S$ admits a dependency graph, and we denote $N = |V|$ and $D = 1+\max_{v\in V} (\deg v)$. Then, we have the following bound on the cumulants of $S$:
$$\forall r \geq 1,\quad |\kappa^{(r)}(S)| \leq N\,(2D)^{r-1}\,r^{r-2}\,A^r.$$
Moreover, this upper bound actually holds with a constant $N = \frac{1}{A} \sum_{v \in V}\esper[|A_v|]$, which is smaller than $|V|$.
\end{theorem}
\noindent As a consequence, if $(S_n)_{n \in \N}$ is a sequence of sums of bounded random variables, and if each sum $S_n$ admits a dependency graph $G_n$ with parameters $D_n$, $N_n$ and $A$, then the uniform bound \eqref{eq:cumulant1} holds, and the condition $\lim_{n \to \infty} \frac{D_n}{N_n}=0$ amounts to the fact that the graphs $G_n$ are sparse. 
We shall see in this article that certain natural observables of models of random graphs/permutations/partitions can be related to sums of dependent random variables with sparse dependency graphs, and therefore that these models typically exhibit mod-Gaussian convergence. This implies new asymptotic results for these models, for instance speed of convergence estimates and moderate deviations;  to the best of our knowledge, all these results are new.

\begin{remark}
  Moderate deviations and Kolmogorov distance estimates
 under bounds of cumulants were first obtained
 by Saulis and Statulevi\v{c}ius in \cite[Section 2]{LivreOrange:Cumulants},
 though under a less explicit form.
 Bounds on the Kolmogorov distance in the context of dependency graphs
 can alternatively be obtained through Stein's method,
 see \emph{e.g.}~\cite{Rin94}.
\end{remark}

\begin{remark}[Weighted dependency graphs]
In \cite[Section 5]{FMN17}, we developed a more general method in order to prove uniform bounds on cumulants, which relied on the notion of weighted dependency graphs. Here, we shall not need to use these more complex arguments.
\end{remark}
\medskip

\subsection{Random graphs, random permutations and random partitions}\label{subsec:informalpresentationmodels}
Let us now present informally the combinatorial random models that we shall study in this paper. We start with random graphs and the theory of graphons. In the following, unless stated otherwise, a \emph{graph} will always be a finite undirected simple graph, that is to say a pair $G=(V_G,E_G)$ with $V_G$ finite set of \emph{vertices}, and $E_G$ subset of the set $\mathfrak{P}_2(V_G)$ of pairs of vertices. Thus, $E_G$ is a finite set of pairs $\{v_1,v_2\}$ with $v_1,v_2 \in V_G$ and $v_1 \neq v_2$. These pairs are the \emph{edges} of the graph.

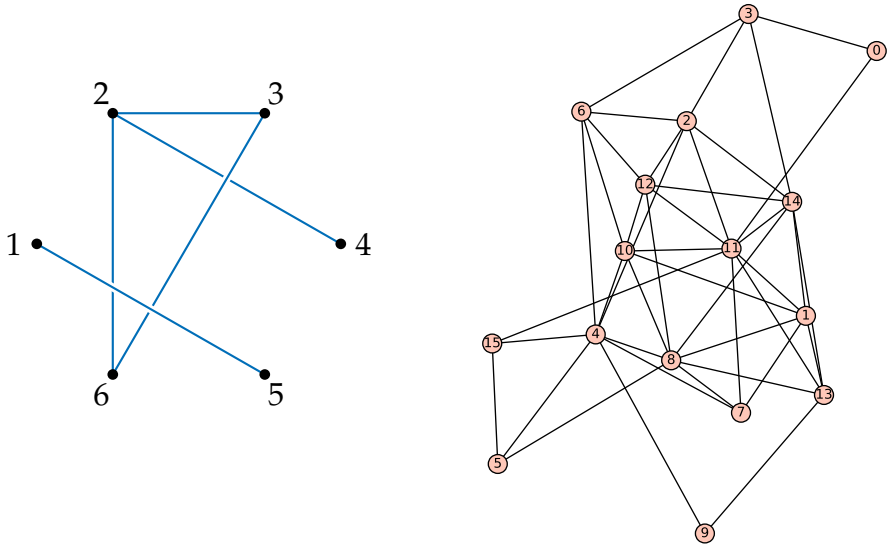
\begin{figure}[ht]
  \[
  \begin{array}{c}
\begin{tikzpicture}[scale=1]
\draw [thick,NavyBlue] (0:2cm) -- (120:2cm) -- (60:2cm);
\draw [line width=3pt,white] (60:2cm) -- (240:2cm);
\draw [thick,NavyBlue] (60:2cm) -- (240:2cm);
\draw [thick,NavyBlue] (240:2cm) -- (120:2cm);
\draw [line width=3pt,white] (180:2cm) -- (300:2cm);
\draw [thick,NavyBlue] (180:2cm) -- (300:2cm);
\foreach \x in {0,60,120,180,240,300}
\fill (\x:2cm) circle (2pt);
\draw (180:2.3cm) node {$1$};
\draw (120:2.3cm) node {$2$};
\draw (60:2.3cm) node {$3$};
\draw (0:2.3cm) node {$4$};
\draw (300:2.3cm) node {$5$};
\draw (240:2.3cm) node {$6$};
\end{tikzpicture}
\end{array}
\qquad
  \begin{array}{c}
\includegraphics[scale=0.5]{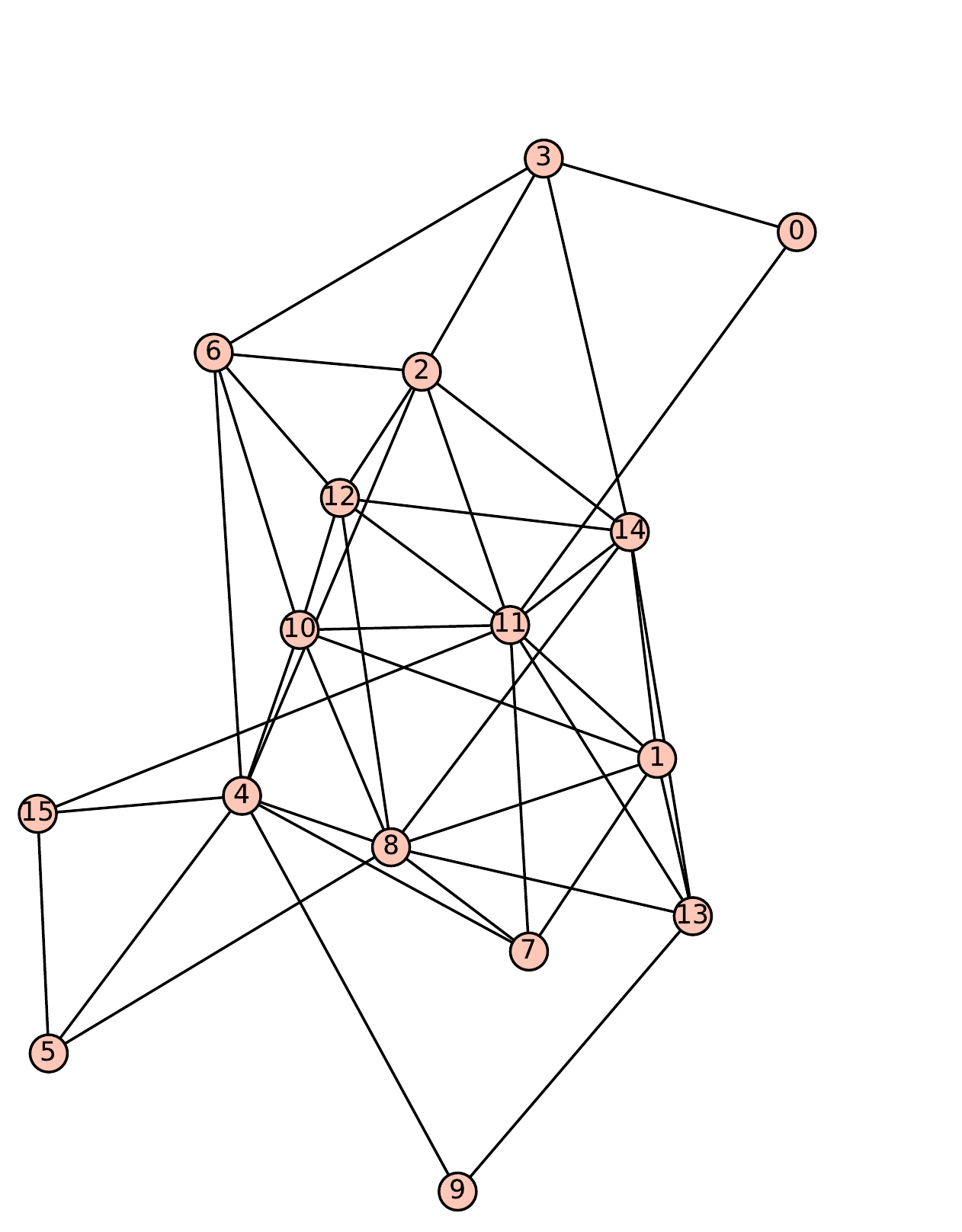}
\end{array}\]
\caption{Two graphs $G$ having each a density $\frac{1}{20}$ of triangles.\label{fig:densitytriangle}}
\end{figure}

When looking at a large (random) graph $G$, a way to speak of convergence is to look at the numbers of subgraphs that appear in $G$.
In the left picture of Figure~\ref{fig:densitytriangle},
there is only one triangle, formed by the vertices $\{2,3,6\}$.
As the number of possible triangles in a graph with $6$ vertices is $\binom{6}{3} = 20$,
the {\em density of triangles} $t_0(K_3,G)$ in $G$ is $\frac{1}{20}$.
The right picture of Figure~\ref{fig:densitytriangle}
shows a larger graph, with $16$ vertices and again a density of triangles equal to $\frac{1}{20}$
(there are 28 triangles, in this graph, among the $\binom{16}{3}=560$ possible ones).
\medskip

We will be interested in sequences $(G_n)_{n \in \N}$ of graphs,
for which the density of triangles $(t_0(K_3,G_n))_{n \in \N}$ has a limit in $[0,1]$.
 More generally, for any finite graph $F$ of size $k$, we can look at the density $t(F,G)$ of the subgraph $F$ in $G$, which is defined by the number of occurrences of $F$ in $G$, divided by the total number of possibilities for $F$ to appear as a subgraph of $G$. The precise definition of this density involves the notion of morphism of graphs, and it will be given in Section \ref{sec:observables}; one can actually define two slightly different densities $t_0(F,G)$ and $t(F,G)$, but for the moment let us forget this subtlety. A convergent sequence of graphs $(G_n)_{n \in \N}$ will then be a sequence of graphs such that for any $F$ finite graph, $(t(F,G_n))_{n\in \N}$ has a limit in $[0,1]$.
\medskip

This notion of convergence can also be {\em realized} as a convergence in some metric space.
This construction has been performed by Lov\'asz and Szegedy in \cite{LS06}, and it leads to the theory of graphons.
Namely, there exists a compact metrisable space $\Gcal$ called the \emph{space of graphons}, in which one can embed any finite graph $G$, and such that a sequence of finite graphs $(G_n)_{n \in \N}$ has convergent densities $t(F,G_n)$ for any finite graph $F$ if and only if there exists some parameter $\gamma \in \Gcal$ such that $G_n \to \gamma$ in the space of graphons. 
Moreover, graphs are dense in the space $\Gcal$; this is shown by constructing, for each graphon $\gamma$,
a model of random graphs $(G_n(\gamma))_{n \in \N}$ that converges almost surely to $\gamma$
\cite[Section 2.6]{LS06}.
\medskip

In the present paper, we prove that, for any graphon $\gamma$ and fixed graph $F$,
the observable $t(F,G_n(\gamma))$ are mod-Gaussian convergent after an appropriate renormalisation.
In the previous sentence, the word "generically" means that there exists a universal renormalisation $\alpha_{n,F}\,t(F,G_n(\gamma))$ of the density $t(F,G_n(\gamma))$ which depends only on $F$, and such that for any graphon $\gamma \in \Gcal$, we have a convergence in law 
$$ \alpha_{n,F}\,t(F,G_n(\gamma)) \rightharpoonup_{n \to \infty} \mathcal{N}_{\R}(0,\sigma^2(F,\gamma)).$$ 
When $\sigma^2(F,\gamma)>0$, we have mod-Gaussian convergence  through uniform bounds of cumulants
and all the precise probabilistic estimates listed in Theorem \ref{thm:cumulantestimates}.
Then, the only pairs of parameters $(F,\gamma)$ 
for which this asymptotic normality does not sit in the framework of mod-Gaussian convergence are those for which $\sigma^2(F,\gamma)=0$.
Notice that this singular case does not prohibit the existence of another renormalisation 
$\alpha_{n,F}'\,t(F,G_n(\gamma))$ which falls in the framework of Theorem \ref{thm:cumulantestimates}.
The theory that we shall develop will allow us:
\begin{itemize}
 	\item to describe the asymptotic fluctuations of random models $(G_n(\gamma))_{n \in \N}$ when these fluctuations have a typical size;
 	\item to identify the parameters $\gamma$ such that $\sigma^2(F,\gamma)=0$ for all graphs $F$.
      These parameters $\gamma$ will be called {\em globally singular} and they seem to exhibit some structure.
       \end{itemize} 
For instance, we shall see that the Erdös--Rényi random graphs are singular models with respect to any density of a subgraph; see the remark on singular points on page \pageref{rem:singularity}. Our theory is concerned with models of random graphs which are less symmetric and therefore more generic. It enables then a fine understanding of their asymptotic behavior.

\begin{center}
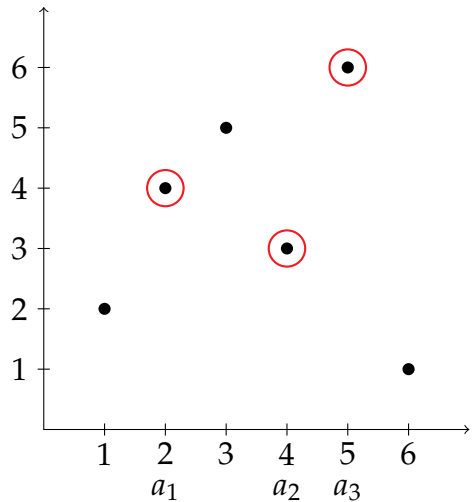
\begin{figure}[ht]
\begin{tikzpicture}[scale=0.8]
\foreach \x in {(1,2),(2,4),(3,5),(4,3),(5,6),(6,1)}
\fill \x circle (1mm);
\foreach \x in {(2,4),(4,3),(5,6)}
\draw [Red,thick] \x circle (3mm);
\draw [<->] (0,7) -- (0,0) -- (7,0);
\foreach \x in {1,2,3,4,5,6}
{\draw (\x,-0.1) -- (\x,0.1) ; \draw (-0.1,\x) -- (0.1,\x) ;};
\draw (1,-0.4) node {$1$};
\draw (2,-0.4) node {$2$};
\draw (2,-1) node {$a_1$};
\draw (3,-0.4) node {$3$};
\draw (4,-0.4) node {$4$};
\draw (4,-1) node {$a_2$};
\draw (5,-0.4) node {$5$};
\draw (5,-1) node {$a_3$};
\draw (6,-0.4) node {$6$};
\draw (-0.4,1) node {$1$};
\draw (-0.4,2) node {$2$};
\draw (-0.4,3) node {$3$};
\draw (-0.4,4) node {$4$};
\draw (-0.4,5) node {$5$};
\draw (-0.4,6) node {$6$};
\end{tikzpicture}
\caption{The permutation $213$ is a pattern in $\sigma=245361$.\label{fig:pattern}}
\end{figure}
\end{center}


An approach analogous to the theory of graphons can be used in order to deal with models of random permutations. Recall that a \emph{permutation} of size $n$ is a bijection $\sigma : \lle 1,n\rre \to \lle 1,n\rre$. The set of all permutations of size $n$ is the \emph{symmetric group} of order $n$, denoted $\sym(n)$, and with cardinality $n!$. If $\tau \in \sym(k)$ and $\sigma \in \sym(n)$ with $k \leq n$, we say that $\tau$ is a \emph{pattern} in $\sigma$ if there exists a subset $\{a_1 < a_2<\cdots <a_k\}\subset \lle 1,n\rre$ such that $\sigma(a_i)<\sigma(a_j)$ if and only if $\tau(i)<\tau(j)$. This definition is better understood on a picture: if one draws the diagram of $\sigma$ (graph of $\sigma$ viewed as a function), then one can isolate points with abscissa $a_1<a_2<\cdots<a_k$ such that the restriction of the diagram of $\sigma$ to these points is the diagram of the permutation $\tau$; see Figure \ref{fig:pattern}. In the previous example, there are $2$ sets $\{a_1<a_2<a_3\}$ that make appear $\tau=213$ as a pattern in $\sigma=245361$, namely, $\{2,4,5\}$ and $\{3,4,5\}$. Therefore, it is natural to say that $\tau = 213$ has a pattern density $t(\tau,\sigma)$ equal to $\frac{2}{\binom{6}{3}}=\frac{1}{10}$ in the permutation $\sigma$. \medskip

We can extend the definition to any pattern $\tau$, and exactly as for graphs, we then say that a sequence $(\sigma_n)_{n\in \N}$ of permutations converges if, for any finite permutation $\tau$, the sequence of densities $(t(\tau,\sigma_n))_{n \in \N}$ has a limit in $[0,1]$. Again, this notion of convergence leads to the construction of a compact metrisable set $\Scal$ called the \emph{space of permutons}, in which one can embed any finite permutation $\sigma$, and such that a sequence of permutations $(\sigma_n)_{n\in \N}$ has convergent densities $t(\tau,\sigma_n) \to t(\tau)$ for any pattern $\tau$ if and only if there exists some parameter $\pi \in \Scal$ such that $\sigma_n \to \pi$ in the space of permutons. The space of permutons was introduced by Hoppen \emph{et al.} in \cite{HKMS11,HKMRS13}.
Again, these authors prove that permutations are dense in $\Scal$ by contructing,
for each permuton $\pi$, a model of random permutations $(\sigma_n(\pi))_{n \in \N}$ such that
$\sigma_n(\pi)$ tends to $\pi$ almost surely
\cite[Section 4]{HKMRS13}.
\medskip

\noindent In this paper, we shall prove that  for any pattern $\tau$ and permuton $\pi \in \Scal$, the observable $t(\tau,\sigma_n(\pi))$ are generically mod-Gaussian convergent after an appropriate renormalisation.
Again, this statement implies a central limit theorem for the densities of patterns $t(\tau,\sigma_n(\pi))$, and additional results such as a bound on the speed of convergence or a concentration inequality (see next section).
\bigskip

A third class of models consists in random integer partitions which stem from random permutations obtained by shuffle of cards, and which are related to the representation theory of the infinite symmetric group $\sym(\infty)$ (see \cite{Mel12} and \cite[Chapter 12]{Mel17}).
The mod-Gaussian convergence of some observables of these models was established in \cite[Chapter 11]{FMN16},
but we take here a different approach to emphasize similarities with graphons and permutons.
\medskip

\noindent Recall that an \emph{integer partition} of size $n$ is a non-increasing sequence $\lambda = (\lambda_1 \geq \cdots \geq \lambda_r)$ of positive integers with $\lambda_1+\lambda_2+\cdots + \lambda_r=n$. The set of all integer partitions of size $n$ is denoted $\pym(n)$, and an integer partition $\lambda$ is usually represented by its Young diagram, which is the array of boxes with $\lambda_1$ cells on its first bottom row, $\lambda_2$ cells on its second row, \emph{etc.} 
\begin{center}		
\begin{figure}[ht]
\begin{tikzpicture}[scale=1]
\draw (0,0) -- (5,0) -- (5,1) -- (4,1) -- (4,2) -- (2,2) -- (2,3) -- (0,3) -- (0,0);
\draw (1,0) -- (1,3);
\draw (0,2) -- (2,2) -- (2,1) -- (4,1) -- (4,0);
\draw (0,1) -- (2,1) -- (2,0);
\draw (3,0) -- (3,2);
\draw [dashed] (0,0) -- (2,2);
\draw [very thick, NavyBlue, <->] (0.6,0.5) -- (5,0.5);
\draw [very thick, NavyBlue, <->] (1.6,1.5) -- (4,1.5);
\draw [very thick, Red, <->] (0.5,0.6) -- (0.5,3);
\draw [very thick, Red, <->] (1.5,1.6) -- (1.5,3);
\draw [NavyBlue] (2.8,0.2) node {$a_1$};
\draw [NavyBlue] (2.8,1.2) node {$a_2$};
\draw [Red] (0.25,2.5) node {$b_1$};
\draw [Red] (1.25,2.5) node {$b_2$};
\end{tikzpicture}
\caption{The Young diagram and the Frobenius coordinates of the integer partition $(5,4,2)$ of size $n=11$.\label{fig:partition}}
\end{figure}
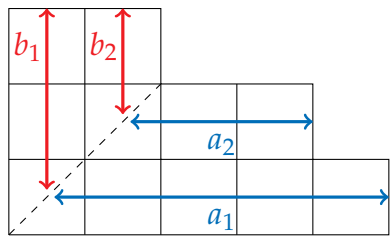
\end{center}

The \emph{Frobenius coordinates} of $\lambda$ are the half-integers $(a_1,a_2,\ldots ,a_d;b_1,b_2,\ldots,b_d)$ defined as follows: $d$ is the size of the diagonal of the Young diagram of $\lambda$, and
$$a_i = \lambda_i - i +\frac{1}{2} \qquad;\qquad b_i = \lambda_i'-i+\frac{1}{2},$$
where $\lambda'$ is the partition obtained from $\lambda$ by symetrising the Young diagram with respect to its diagonal (see \emph{e.g.}~\cite[Section 7.2]{Mel17}). For instance, when $\lambda=(5,4,2)$, one obtains as Frobenius coordinates
$$a_1 = \frac{9}{2},\,\,a_2 = \frac{5}{2}\qquad;\qquad b_1 = \frac{5}{2},\,\,b_2 = \frac{3}{2}.$$
We refer to Figure \ref{fig:partition} for a geometric interpretation of the Frobenius coordinates. They allow one to see any integer partition as an element of the \emph{Thoma simplex}, which is the bi-infinite dimensional simplex
$$\Pcal = \left\{(\alpha,\beta) = ((\alpha_1\geq \alpha_2 \geq \cdots \geq 0),(\beta_1\geq \beta_2 \geq \cdots \geq 0))\,\,\big|\,\,\sum_{i=1}^\infty (\alpha_i+\beta_i) \leq 1\right\}.$$
Indeed, one can associate with $\lambda \in \pym(n)$ the two sequences 
\begin{align*}
\alpha(\lambda) &= \left(\frac{a_1(\lambda)}{n},\frac{a_2(\lambda)}{n},\ldots,\frac{a_d(\lambda)}{n},0,0,\ldots\right);\\
\beta(\lambda) &= \left(\frac{b_1(\lambda)}{n},\frac{b_2(\lambda)}{n},\ldots,\frac{b_d(\lambda)}{n},0,0,\ldots\right),
\end{align*}
and the pair $\omega(\lambda) = (\alpha(\lambda),\beta(\lambda))$ belongs to $\Pcal$. 
We equip the Thoma simplex $\Pcal$ with the pointwise convergence topology,
{\em \, i.e.} a sequence $(\alpha^{(n)},\beta^{(n)})$ converges if for each $i \ge 1$, 
the coordinates $\alpha^{(n)}_i$ and $\beta^{(n)}_i$ converge.
As we will see later in Proposition~\ref{prop:topology_obs_partitions}, this is equivalent to the convergence of the following observables,
to which we will refer as {\em Frobenius moments}: for $k \ge 2$, let
$$t(k,\omega=(\alpha,\beta)) \coloneqq \sum_{i=1}^\infty (\alpha_i)^k + (-1)^{k-1}\sum_{i=1}^\infty (\beta_i)^k.$$
As for graphons and permutons,
any parameter $\omega \in \Pcal$ defines, for each $n\ge 1$,
a model of random integer partitions of size $n$, denoted $(\lambda_n(\omega))_{n \in \N}$.
Namely, the distribution of $\lambda_n(\omega)$ is given by
$$\proba[\lambda_n(\omega) = \lambda] = (\dim \lambda)\,s_\lambda(\omega),$$
where $\dim \lambda$ is the number of standard tableaux with shape $\lambda$, and $s_\lambda(\omega)$ is some specialisation of the Schur function $s_\lambda$ associated with the parameter $\omega$; see Section \ref{subsec:thoma} for details on the construction of the random partitions $\lambda_n(\omega)$ and for their combinatorial interpretation.
Kerov and Vershik \cite{KV81} showed that for any $\omega = (\alpha,\beta)$ in $\Pcal$, one has convergence in probability $\lambda_n(\omega) \to \omega$ in the Thoma simplex.
We shall prove that the observables $t(k,\lambda_n(\omega))$ are generically mod-Gaussian convergent after an appropriate renormalisation. This completes a previous result of \cite[Chapter 11]{FMN16}, where the same result was proven for the random character values $\chi^{\lambda_n(\omega)}(c_k)$.
\medskip

\subsection{Concentration inequalities}
\label{ssec:concentration_intro}
The theory of graphons and permutons developed in the papers \cite{LS06,LS07,BCLSV08,HKMS11,HKMRS13} relies on certain concentration inequalities for the observables of the random graphs or permutations; see for instance \cite[Lemma 4.4]{BCLSV08} and \cite[Theorem 4.2]{HKMS11}. One of the objective of this paper is to show that these inequalities are in fact easy consequences of the mod-Gaussian structure. Indeed, under the hypotheses of the method of cumulants, we can state several concentration inequalities which show that the random variables considered are uniformly sub-Gaussian. We start with an inequality which only involves the parameters $(D_n,N_n,A)$ of the method of cumulants:
\begin{proposition}[Concentration inequality]\label{prop:concentration1}
  Let $(S_n)_{n \in N}$ be a sequence of random variables such that \eqref{eq:cumulant1} holds
  for some parameters $(D_n,N_n,A)$.
  We also assume that $|S_n| \leq N_n A$ almost surely (this is for instance the case if $S_n$ is a sum of bounded random variables with a dependency graph with parameters $(D_n,N_n,A)$). For any $x>0$,
$$\proba[|S_n - \esper[S_n]| \geq x] \leq 2\,\exp\left(-\frac{x^2}{9\,D_nN_nA^2}\right).$$
\end{proposition}
We can also write a concentration inequality which involves explicitly the variance of $S_n$, and which is more precise for small fluctuations:
\begin{proposition}[Concentration inequality involving the variance]\label{prop:concentration2}
With the exact same assumptions as in Proposition \ref{prop:concentration1}, we have for any $x>0$
$$\proba[|S_n - \esper[S_n]| \geq x] \leq 2\,\exp\left(-\frac{2x^2}{3\left(\var(S_n) + 2\E\,D_nN_nA^2 \sqrt{\frac{x}{N_nA}}\right)}\right).$$
\end{proposition}
Both inequalities are proved in Section~\ref{ssec:proofs_concentration} by using a Chernoff bound.
We shall see in Theorems~\ref{thm:modgraphon} and \ref{thm:modpermuton} below that the aforementioned concentration inequalities for graphon and permuton models follow immediately from Proposition \ref{prop:concentration1}, though with different constants than in their original form. Besides, we shall obtain new concentration inequalities for the observables of random integer partitions (Theorem \ref{thm:concentrationpartitions}).
\medskip

In the case where the bounds on cumulants are obtained through a (sparse) dependency graph,
a better bound than Proposition~\ref{prop:concentration1} has been given by Janson in \cite{Jan04},
with instead of $D_n$ the (fractional) chromatic number $\chi^*(G_n)$ of the dependency graph $G_n$ of $S_n$. 
Proposition~\ref{prop:concentration1} is therefore particularly interesting in the cases
where there is no underlying sparse dependency graph.
An example of this is Theorem~\ref{thm:concentrationpartitions} below for random partitions under central measures:
indeed, there is no standard dependency graph in this setting, but one in a noncommutative probability space,
to which Janson's result does not apply (see \cite[Chapter11]{FMN16} for the construction of this graph).
Another example without underlying (sparse) dependency graph is presented in Section~\ref{ssec:Ising}:
we give a concentration inequality
for the magnetization in the Ising model,
based on the bounds on cumulants given in \cite{Duneau2} and \cite[Section 5]{FMN17}.

\subsection{Main results and outline of the paper}
Let us now detail the content of this article. We start by stating some new results on the models that were informally introduced in Section \ref{subsec:informalpresentationmodels}. These theorems can serve as a guideline for the reader, and they give a good idea of the kind of results that one can obtain in the framework of mod-Gaussian convergence. Moreover, these asymptotic results seem to be somehow universal for random combinatorial models with a concentration property. We start with the results on random graphs:

\begin{theorem}[Speed of convergence for graphon models]\label{thm:speedofconvergencegraphon}
Let $\gamma \in \Gcal$ be a graphon, and $G_n(\gamma)$ be the random graph of size $n$ associated with it (see Section \ref{subsec:graphonspace} for details). If $F$ is a finite graph with size $k$ such that
$$\lim_{n \to \infty} n\,\var(t(F,G_n(\gamma))) > 0,$$
then the rescaled subgraph density 
$$Y_n(F,\gamma) = \frac{t(F,G_n(\gamma)) -\esper[t(F,G_n(\gamma))] }{\sqrt{\var(t(F,G_n(\gamma)))}}$$
satisfies
$$\dkol(Y_n(F,\gamma),\gauss) = O\!\left(\frac{1}{(n\,\var(t(F,G_n(\gamma))))^{3/2}}\,\frac{k^4}{\sqrt{n}}\right),$$
with a constant in the $O(\cdot)$ that is universal (it does not depend on $\gamma$ or on $F$ as long as the scaled variance $n\,\var(t(F,G_n(\gamma)))$ does not tend to $0$).
\end{theorem}
\begin{theorem}[Moderate deviations for graphon models]\label{thm:moderatedeviationgraphon}
We use the same notation as in Theorem \ref{thm:speedofconvergencegraphon}
and also assume
$$\lim_{n \to \infty} n\,\var(t(F,G_n(\gamma))) > 0. $$
Then, then there is a real number $l(F,\gamma)$ such that,
if $y_n = o(n^{1/4})$ and $y_n \to +\infty$, we have
$$\proba[Y_n(F,\gamma) \geq y_n] = \frac{\E^{-\frac{(y_n)^2}{2}}}{y_n\,\sqrt{2\pi}}\,\exp\left(l(F,\gamma)\,\frac{(y_n)^3}{\sqrt{n}}\right)\,(1+o(1)).$$
In particular, if $y_n = o(n^{1/6})$, we have $\proba[Y_n(F,\gamma) \geq y_n] = \proba[Z\ge y_n] (1+o(1))$,
where $Z$ is a standard Gaussian random variable, and $1/6$ is the maximal exponent such that this happens.
We say that
the {\em zone of normality} of $Y_n(F,\gamma)$ is $o(n^{1/6})$.
\end{theorem}
\noindent In these theorems, the condition $\lim_{n \to \infty} n\,\var(t(F,G_n(\gamma))) > 0$ will be easy to check, by evaluating a certain observable $\kappa_2(F,F)$ on the graphon $\gamma$. 
\bigskip

We have exact analogues of Theorems \ref{thm:speedofconvergencegraphon} and \ref{thm:moderatedeviationgraphon} for the models of random permutations:
\begin{theorem}[Speed of convergence for random permutations]\label{thm:speedofconvergencepermuton}
Let $\pi \in \Scal$ be a permuton, and $\sigma_n(\pi)$ be the random permutation of size $n$ associated with it (see Section \ref{subsec:permutonspace} for details). If $\tau$ is a finite permutation with size $k$ such that
$$\lim_{n \to \infty} n\,\var(t(\tau,\sigma_n(\pi))) > 0,$$
then the rescaled pattern density
$$Y_n(\tau,\pi) = \frac{t(\tau,\sigma_n(\pi)) -\esper[t(\tau,\sigma_n(\pi))] }{\sqrt{\var(t(\tau,\sigma_n(\pi)))}}$$
satisfies
$$\dkol(Y_n(\tau,\pi),\gauss) = O\!\left(\frac{1}{(n\,\var(t(\tau,\sigma_n(\pi))))^{3/2}}\,\frac{k^4}{\sqrt{n}}\right),$$
with a constant in the $O(\cdot)$ that is universal.
\end{theorem}
\begin{theorem}[Moderate deviations for random permutations]\label{thm:moderatedeviationpermuton}
We use the same notation as in Theorem \ref{thm:speedofconvergencepermuton}
and also assume
 $$\lim_{n \to \infty} n\,\var(t(\tau,\sigma_n(\pi))) > 0. $$
 Then there is a real number $l(\tau,\pi)$ which one can compute exactly, such that,
 if $y_n = o(n^{1/4})$ and $y_n \to +\infty$, we have
$$\proba[Y_n(\tau,\pi) \geq y_n] = \frac{\E^{-\frac{(y_n)^2}{2}}}{y_n\,\sqrt{2\pi}}\,\exp\left(l(\tau,\pi)\,\frac{(y_n)^3}{\sqrt{n}}\right)\,(1+o(1)).$$
In particular, if $y_n = o(n^{1/6})$, we have $\proba[Y_n(\tau,\pi) \geq y_n] = \proba[Z\ge y_n] (1+o(1))$,
where $Z$ is a standard Gaussian random variable, and $1/6$ is the maximal exponent such that this happens.
We say that
the {\em zone of normality} of $Y_n(\tau,\pi)$ is $o(n^{1/6})$.
\end{theorem}
\medskip

Last, let us state some new concentration results for the central measures on integer partitions:
\begin{theorem}[Concentration inequalities for random integer partitions]\label{thm:concentrationpartitions}
Let $\lambda_n(\omega)$ be the random integer partition with size $n$ chosen under the central measure associated with a parameter $\omega=(\alpha,\beta)$ of the Thoma simplex (see Section \ref{subsec:thoma} for details on the definition of these models). Given an integer $k$, we consider the $k$-th Frobenius moments
\begin{align*}
t(k,\lambda_n(\omega)) &= \frac{1}{n^k}\left(\sum_{i=1}^d (a_i(\lambda_n(\omega)))^k+(-1)^{k-1}\sum_{i=1}^d (b_i(\lambda_n(\omega)))^k\right) ;\\
t(k,\omega)&= \sum_{i=1}^\infty (\alpha_i)^k +(-1)^{k-1} \sum_{i=1}^\infty (\beta_i)^k.
\end{align*}
We have
$$\proba[|t(k,\lambda_n(\omega))-t(k,\omega)|\geq x ]\leq 4\,\exp\left(-\frac{nx^2}{9k^2}\right).$$
\end{theorem}
\bigskip

We do not give concentration results for graphons and permutons since those are already present in the literature.
Similarly, we chose not to state speed of convergence and moderate deviation results for random partitions;
see \cite[Section 11]{FMN16} for results of this kind, though the observables considered there are different.
\medskip

All the theorems stated previously fall in the framework of mod-Gaussian convergence, and as a consequence, their proofs are very similar. This paper is written in a way which emphasizes these similarities. In Section \ref{sec:models}, we present the three compact spaces of parameters which correspond to classes of models of random graphs, of random permutations and of random integer partitions. We explain in each case how to build a random combinatorial object $O_n(m)$ with size $n$ when $m$ is a parameter in the compact space $\Mcal$. We also explain how to see $O_n(m)$ as a random element $M_n(m)$ in $\Mcal$. The asymptotic concentration of the models $(O_n(m))_{n \in \N}$ amounts then to the convergence in probability $M_n(m) \to m$ for any $m \in \Mcal$.\medskip

In Section \ref{sec:observables}, we introduce for each space of parameters $\Mcal$ an algebra of observables $\obs_{\mathfrak{M}}$, which one can see as continuous functions on $\Mcal$, and which have the following properties:
\begin{itemize}
	\item The convergence in $\Mcal$ is equivalent to the convergence of all the observables in $\obs_{\mathfrak{M}}$.
	\item The fluctuations of the observables of the random models $O_n(m)$ are mod-Gaussian after an appropriate renormalisation, and the limiting parameters $\sigma^2$ and $L$ of these mod-Gaussian sequences can be computed by using operations in the algebra $\obs_\mathfrak{M}$.
\end{itemize}
It is quite remarkable that the topology of each class of objects can always be metrised by a nice algebraic structure, and that this structure also appears in the calculations of the fluctuations.\medskip

We prove the mod-Gaussian convergence of the observables in Section \ref{sec:dependency} (Theorems \ref{thm:modgraphon}, \ref{thm:modpermuton} and \ref{thm:modpartition}), most of the time by exhibiting a dependency graph for the random variables under study. Though the random combinatorial objects and their constructions are very different, the dependency structures as well as the formulas for the limiting variances and limiting third cumulants are almost the same for the three cases (graphs, permutations and partitions). For instance, the variance of a graph density $t(F,G_n(\gamma))$ will be estimated by using a graph observable
$$\kappa_2(F,F) = \frac{1}{|F|^2} \sum_{1\leq a,b \leq |F|} ((F\join F)(a,b) - F\sqcup F),$$
whereas the variance of a pattern density $t(\tau,\sigma_n(\pi))$ will be estimated by using a permutation observable
$$\kappa_2(\tau,\tau) = \frac{1}{|\tau|^2} \sum_{1\leq a,b \leq |\tau|} ((\tau\join \tau)(a,b) - \tau \times \tau).$$
The notations used throughout the paper will make it easy to see these analogies. In our last Section \ref{sec:moduli}, we discuss these similarities and we introduce a notion of \emph{mod-Gaussian moduli space} for a pair $(\Mcal,\obs_{\mathfrak{M}})$. Our definition yields a rigorous meaning to genericity when studying the fluctuations of random combinatorial objects in a certain class. In particular, the random models with additional symmetries are usually non-generic, and they appear as singular points of the mod-Gaussian moduli spaces. To the best of our knowledge, this geometric approach of the study of classes of random models is a new viewpoint in probability theory.
\bigskip

\section{Concentration inequalities}

This section focuses on the concentration inequalities discussed in Section~\ref{ssec:concentration_intro}.
We first prove the announced inequalities, and then give an application to the $d$-dimensional Ising model.

\subsection{Proofs of the concentration inequalities}
\label{ssec:proofs_concentration}
\begin{proof}[Proof of Proposition~\ref{prop:concentration1}]
Setting $X = S_n-\esper[S_n]$, we use Chernov's bound $\proba[X \geq x] \leq \frac{\esper[\E^{t X}]}{\E^{tx}}$ and we shall choose later the optimal parameter $t >0$. By assumption,
\begin{align*}
\log \esper[\E^{t X}] = \sum_{r=2}^\infty \frac{\kappa^{(r)}(S_n)}{r!}\,t^r 
&\leq \frac{N_n}{2D_n} \sum_{r=3}^\infty \frac{r^{r-2}}{r!} (2D_nA t)^{r}\\
&\leq\frac{N_n}{2\sqrt{2\pi}\,D_n} \sum_{r=2}^\infty \frac{1}{r^{5/2}} (2\E D_nA t)^{r}
\end{align*}
by using Stirling's bound $r! \geq (\frac{r}{\E})^r\,\sqrt{2\pi r}$. For $r \geq 2$, $r^{5/2} \geq \frac{\sqrt{27}}{2}\,r(r-1)$, so 
\begin{align*}
\log \esper[\E^{t X}] &\leq \frac{N_n}{\sqrt{54\pi}\,D_n} \sum_{r=2}^\infty \frac{\theta^r}{r(r-1)} = \frac{N_n}{\sqrt{54\pi}\,D_n}\left((1-\theta)\log(1-\theta)+\theta\right)
\end{align*}
with $\theta = 2\E D_n A t$, and assuming $\theta<1$. If one substracts $tx = \frac{\theta\,x}{2\E D_nA}$ from this upper bound and choose the optimal $\theta = 1-\E^{-\frac{\sqrt{54\pi}}{2\E} \frac{x}{N_n A}}$, one obtains:
\begin{align*}
\log \proba[X \geq x] \leq \frac{N_n}{\sqrt{54\pi}\, D_n} \left(\left(1-\E^{-\frac{\sqrt{54\pi}}{2\E}\,\frac{x}{N_nA}}\right) - \frac{\sqrt{54\pi}}{2\E}\,\frac{x}{N_nA} \right)
\end{align*}
for any positive $x$. Notice now that the function $f(u)=\frac{1-\E^{-u}-u}{u^2}$ is negative and increasing. Moreover, in the previous bound, we can assume without loss of generality that $x \leq N_n A$, because $|X| \leq N_nA$ almost surely by hypothesis. Therefore, we can replace the previous bound by
$$\log \proba[X \geq x] \leq \frac{N_n}{\sqrt{54\pi}\, D_n}\,\,f\!\left(\frac{\sqrt{54\pi}}{2\E}\right) \left(\frac{\sqrt{54\pi}}{2\E}\,\frac{x}{N_nA} \right)^{\!2} \leq \frac{C\,x^2}{D_nN_nA^2},$$
where 
$$C = \frac{\sqrt{54\pi}}{4\E^2}\,f\!\left(\frac{\sqrt{54\pi}}{2\E}\right) \leq -0.114,$$ 
which it is convenient to approximate from above by $-\frac{1}{9}$. Finally, since we can use the same bound for the variable $-S_n$ instead of $S_n$, we can bound $\proba[|X|\geq x]$ by twice the bound on $\proba[X \geq x]$.
\end{proof}

\begin{proof}[Proof of Proposition~\ref{prop:concentration2}]
With the same notations as in the proof of Proposition \ref{prop:concentration1}, we obtain by isolating the term of order $r=2$ in the log-generating function:
$$\log \esper[\E^{t X}] - tx \leq \frac{N_n}{D_n}\left(\frac{1}{\sqrt{54\pi}}\left((1-\theta)\log(1-\theta)+\theta-\frac{\theta^2}{2}\right) + \frac{\theta^2}{8\E^2}\left(\frac{\sigma_n}{A}\right)^2 - \frac{\theta}{2\E}\,\frac{x}{N_n A}\right)
$$
with $\theta = 2\E D_n A t$, and assuming $\theta\leq 1$. Then, $(1-\theta)\log(1-\theta)+\theta-\frac{\theta^2}{2} \leq \frac{\theta^3}{2}$, so we can conveniently replace this bound by the polynomial bound
\begin{align*}
\log \esper[\E^{t X}] - tx &\leq \frac{N_n}{D_n}\left(\frac{\theta^3}{6\sqrt{6\pi}} + \frac{\theta^2}{8\E^2}\left(\frac{\sigma_n}{A}\right)^2 - \frac{\theta}{2\E}\,\frac{x}{N_n A}\right)\\
&\leq \frac{N_n}{\E D_n}\left(\frac{\theta^3}{6} + \frac{\theta^2}{8\E}\left(\frac{\sigma_n}{A}\right)^2 - \frac{\theta}{2}\,\frac{x}{N_n A}\right).
\end{align*}
The optimal $\theta$ in the second line is the positive solution of 
$\theta^2+\frac{\theta}{2\E}\left(\frac{\sigma_n}{A}\right)^2 - \frac{x}{N_n A}=0,$
that is
$$\theta = \sqrt{\frac{1}{16\E^2}\left(\frac{\sigma_n}{A}\right)^4 + \frac{x}{N_n A}}-\frac{1}{4\E}\left(\frac{\sigma_n}{A}\right)^2\leq \sqrt{\frac{x}{N_nA}}\leq 1.$$
Thus, this value of $\theta$ is indeed smaller than $1$, and by using this value, we obtain:
\begin{align*}
\log \proba[X \geq x] &\leq -\frac{N_n}{\E D_n}\left(\frac{\theta^3}{3} + \frac{\theta^2}{4\E}\left(\frac{\sigma_n}{A}\right)^2 \right) = -\frac{x}{3 D_n A \left(\E \theta+\frac{1}{2}\left(\frac{\sigma_n}{A}\right)^2\right)}\left(\frac{x}{N_nA} + \frac{\theta}{4\E}\left(\frac{\sigma_n}{A}\right)^2 \right)\\
&\leq -\frac{x^2}{3 D_nN_n \left(\E A^2 \theta+\frac{1}{2}(\sigma_n)^2\right)} = -\frac{4x^2}{3 D_nN_n \left((\sigma_n)^2 + \sqrt{(\sigma_n)^4 + 16\E^2 A^4 \frac{x}{N_nA}}\right)} \\
&\leq  -\frac{2x^2}{3 D_nN_n \left((\sigma_n)^2 + 2\E A^2\sqrt{\frac{x}{N_nA}}\right)} =  -\frac{2x^2}{3  \left((\var(S_n) + 2\E D_nN_n A^2\sqrt{\frac{x}{N_nA}}\right)}. 
\end{align*}
As in the previous proof, the case of $\proba[X \leq -x] $ can be treated with the same tools.
\end{proof}
\smallskip

\subsection{Application to the \texorpdfstring{$d$}{d}-dimensional Ising model}
\label{ssec:Ising}
We consider the $d$-dimensional Ising model with inverse temperature $\beta$ and magnetic field $h$.
We place ourselves either at very high temperature (i.e. $h=0$, $\beta <\beta_0(d)$,
for some threshold $\beta_0(d)$ smaller than the critical inverse temperature $\beta_c(d)$),
or in presence of a magnetic field $h \ne 0$). 
A generic reference on the Ising model is, {\em e.g.}, \cite[Chapter 3]{Velenik}. We consider the model on the full lattice $\Z^d$ and
study the magnetization in a box $\Delta$ (which we will make grow to $\Z^d$), {\em i.e.}
$$M_\Delta \coloneqq \sum_{x \in \lle 1,n\rre^d} \sigma(x).$$ 
As explained in \cite[Section 5.3]{FMN17} (see also \cite{Duneau2}), 
$M_\Delta$ admits uniform bounds on cumulants with parameters $(C_1,C_2 \, |\Delta|, C_3)$, for some constant $C_1$,$C_2$, $C_3$
(these constants depend on the parameters $(\beta,h)$, but not on the box $|\Delta|$).
Applying Proposition \ref{prop:concentration1} gives the following concentration inequality:
for any $x > 0$, any $n \geq 1$, any $d \ge 2$ and any $\beta <\beta_0(d)$, we have
$$\proba\left[\frac{|M_{\Delta} - \esper[M_\Delta]|}{\sqrt{|\Delta|}} \geq x\right] \leq 2\,\exp(-C(d,\beta) \,x^2),$$
for some positive constant $C(d,\beta)>0$. We recover the concentration inequalities 
for Gibbs measures from \cite[Theorem 1]{Kul03} and \cite{CCKR07,CCR17}. 
Such inequalities hold in the Dobrushin uniqueness regime, for instance with $h \ne 0$ and at very high temperature. 
\medskip

\section{The three spaces of parameters}\label{sec:models}
In this section, we present three compact spaces of parameters which label models of random graphs, of random permutations and of random integer partitions. We try to use similar notations for each class of models, and we survey their theory which is relatively new for graphons and permutons.\medskip

\subsection{The space of graphons}\label{subsec:graphonspace}
In this paragraph, we follow the discussion of \cite{LS06,LS07,BCLSV08,BCLSV11}. We call \emph{graph function} a measurable function $g : [0,1]^2 \to [0,1]$ that is symmetric: $g(x,y)=g(y,x)$ for almost any $(x,y)$ (relatively to the Lebesgue measure on the square $[0,1]^2$). 
The set of graph functions will be denoted $\Fcal \subset \leb^\infty([0,1]^2,\DD{x}\DD{y})$.
Given a finite graph $G$ with vertex set $[n]$, we can associate with it a canonical graph function $g$.
Thus, $g=g(G)$ is the function on the square that takes its values in $\{0,1\}$, and is such that
$$g(x,y) = 1\, \text{ if }x \in \left(\frac{i-1}{n},\frac{i}{n}\right], \,\,y \in \left(\frac{j-1}{n},\frac{j}{n}\right]\text{ and }\{i,j\} \in E_G,$$
and $0$ otherwise. We refer to Figure \ref{fig:graphfunction} for an example.
Note that the graph function depends on the labeling of the vertices of $G$,
{\em i.e.} two isomorphic graphs $G$ and $G'$ may have different graph functions $g$ and $g'$.
The graph functions $g$ and $g'$ of isomorphic graphs are however related as follows:
there exists a measure-preserving bijection $\sigma:[0,1] \to [0,1]$ such that $g'=g^\sigma$,
where $g^\sigma(x,y)=g(\sigma(x),\sigma(y))$.
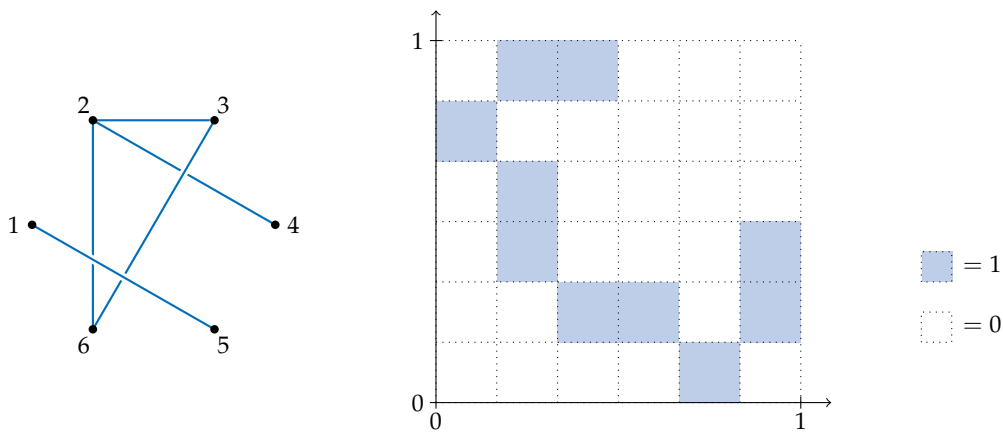
\begin{figure}[ht]
  \[
  \begin{array}{c}
\begin{tikzpicture}[scale=.8,font=\scriptsize]
\draw [thick,NavyBlue] (0:2cm) -- (120:2cm) -- (60:2cm);
\draw [line width=3pt,white] (60:2cm) -- (240:2cm);
\draw [thick,NavyBlue] (60:2cm) -- (240:2cm);
\draw [thick,NavyBlue] (240:2cm) -- (120:2cm);
\draw [line width=3pt,white] (180:2cm) -- (300:2cm);
\draw [thick,NavyBlue] (180:2cm) -- (300:2cm);
\foreach \x in {0,60,120,180,240,300}
\fill (\x:2cm) circle (2pt);
\draw (180:2.3cm) node {$1$};
\draw (120:2.3cm) node {$2$};
\draw (60:2.3cm) node {$3$};
\draw (0:2.3cm) node {$4$};
\draw (300:2.3cm) node {$5$};
\draw (240:2.3cm) node {$6$};
\end{tikzpicture}
\end{array}\qquad
\begin{array}{c}
\begin{tikzpicture}[scale=.8,font=\scriptsize]
\foreach \x in {(0.5,4.5),(4.5,0.5),(1.5,2.5),(1.5,3.5),(1.5,5.5),(2.5,1.5),(3.5,1.5),(5.5,1.5),(2.5,5.5),(5.5,2.5)}
\fill [shift=\x,NavyBlue!25!white] (-0.5,-0.5) rectangle (0.5,0.5);
\foreach \x in {0,1,2,3,4,5,6}
\draw [dotted] (\x,0) -- (\x,6);
\foreach \x in {0,1,2,3,4,5,6}
\draw [dotted] (0,\x) -- (6,\x);
\draw [->] (-0.1,0) -- (6.5,0);
\draw [->] (0,-0.1) -- (0,6.5);
\draw (6,-0.1) -- (6,0.1);
\draw (-0.1,6) -- (0.1,6);
\draw (0,-0.3) node {$0$};
\draw (6,-0.3) node {$1$};
\draw (-0.3,0) node {$0$};
\draw (-0.3,6) node {$1$};
\fill [NavyBlue!25!white] (8,2) rectangle (8.5,2.5);
\begin{scope}[xshift=-.5]
\draw [dotted] (8,2) rectangle (8.5,2.5);
\draw [dotted] (8,1) rectangle (8.5,1.5);
\draw (9,2.3) node {$=1$};
\draw (9,1.3) node {$=0$};
\end{scope}
\end{tikzpicture}
\end{array}\]
\caption{A graph and its associated graph function.\label{fig:graphfunction}}
\end{figure}\medskip

Given $g \in \leb^\infty([0,1]^2,\DD{x}\DD{y})$, we set
$$\|g\|_{\oblong} = \sup_{S,T \subset[0,1]} \left|\int_{S \times T} g(x,y)\DD{x}\DD{y}\right|,$$
where the supremum runs over pairs of measurable subsets $(S,T)$ of $[0,1]$. This is a norm on the space $\leb^\infty([0,1]^2)$ which is equivalent to the norm of operator $\|\cdot\|_{\leb^\infty([0,1]) \to \leb^1([0,1])}$. The \emph{cut-metric} on graph functions $g \in \Fcal$ (see \cite[Section 3.4]{BCLSV08}) is defined by
$$d_\oblong(g,g') = \inf_{\sigma} \|g^\sigma-g'\|_\oblong, $$
where the infimum runs over measure-preserving bijections $\sigma$ of the interval $[0,1]$.
Notice that $d_\oblong(g,g')$ is also the infimum $\inf_{\sigma,\tau} \|g^\sigma-(g')^\tau\|_\oblong$ over pairs of Lebesgue isomorphisms; as a consequence, $d_\oblong$ satisfies the triangular inequality. We define an equivalence relation on $\Fcal$ by 
$$g \sim g' \iff d_\oblong(g,g')=0.$$ 
If $\gamma$ and $\gamma'$ are the equivalence classes of the graph functions $g$ and $g'$, then the quotient space $\Gcal = \Fcal / \!\sim\,$  is endowed with the distance $\delta_\oblong(\gamma,\gamma')=d_\oblong(g,g')$. We call \emph{graphon} an equivalence class of graph functions in $\Gcal$, and the space of graphons $(\Gcal,\delta_\oblong)$ is a compact metric space: see \cite[Proposition 3.6]{BCLSV08}.\medskip

If $\gamma \in \Gcal$ and $g$ is a graph function in this equivalence class,
we can associate with $g$ a sequence of random graphs $(G_n(\gamma))_{n \in \N}$ with $|G_n(\gamma)|=n$.
The construction of $G_n(\gamma)$ ensures that its law is independent of the representative $g$ chosen in the equivalence class $\gamma$,
which justifies the notation $G_n(\gamma)$ instead of $G_n(g)$.
To construct $G_n(\gamma)$, we first draw $n$ independent random variables $X_1,\ldots,X_n$ uniformly in $[0,1]$, and then, $\binom{n}{2}$ Bernoulli random variables $B_{i,j}$ that are independent conditionally to $(X_1,\ldots,X_n)$, and such that
$$\proba[B_{i,j}=1|(X_1,\ldots,X_n)] = g(X_i,X_j)$$
for any $1\leq i<j\leq n$. The random graph $G_n(\gamma)$ is then the graph on $n$ vertices $1,2,\ldots, n$, with $i$ connected to $j$ if and only if $B_{i,j}=1$. We have drawn in Figure \ref{fig:graphon} two examples of such random graphs, when $\gamma$ is the representative of the function $g(x,y)=(x+y)/2$ or of the function $g(x,y)=xy$. 
\begin{center}
\begin{figure}[ht]
\includegraphics[scale=0.4]{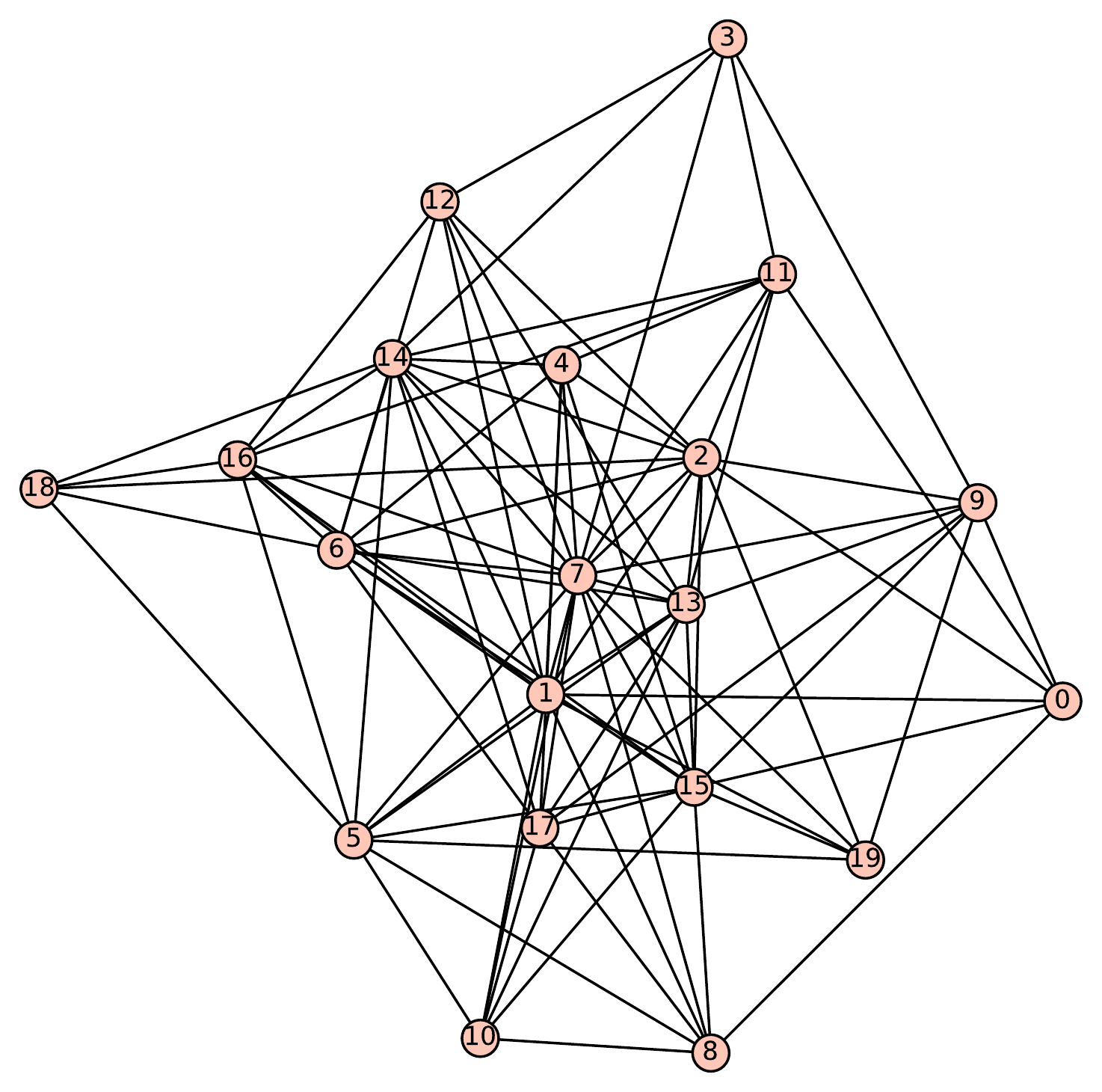} \hspace{1.5cm}
\includegraphics[scale=0.4]{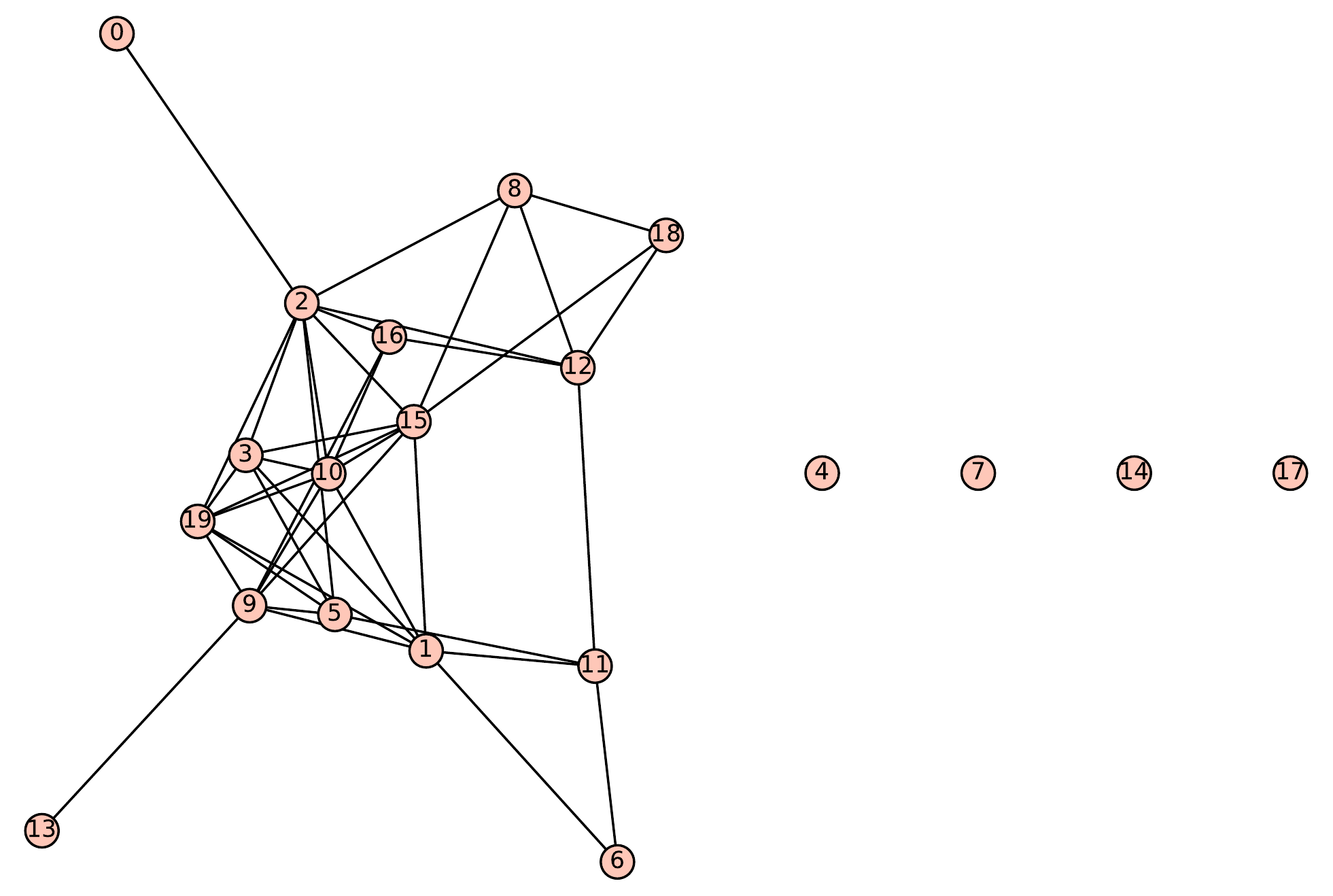}
\caption{Two random graphs of size $n=20$ associated with the graph functions $g(x,y)=\frac{x+y}{2}$ and $g'(x,y)=xy$.\label{fig:graphon}}
\end{figure}
\end{center}

We denote $\Gamma_n(\gamma)$ the equivalence class in $\Gcal$ of the graph function $g(G_n(\gamma))$ that is canonically associated with the random graph $G_n(\gamma)$. 
\begin{proposition}[Corollary 2.6 in \cite{LS06}]\label{prop:concentrationgraph}
For any graphon $\gamma \in \Gcal$, the sequence of random graphons $(\Gamma_n(\gamma))_{n \in \N}$ converges in probability towards $\gamma$.
\end{proposition}

Thus, the space of graphons $\Gcal$ parametrises certain models of random graphs which have a property of asymptotic concentration. We shall recall in Section \ref{subsec:graphdensity} that this framework allows one to deal with models that converge with respect to the notion of subgraph density, and we shall prove later that the models $(G_n(\gamma))_{n \in \N}$ have subgraph densities which are mod-Gaussian convergent.
\medskip

\subsection{The space of permutons}\label{subsec:permutonspace}
We now present an analogous construction with random permutations instead of random graphs, and which involves the notion of \emph{permutons}; we refer to \cite{HKMS11,HKMRS13,GGKK15} for the origins of this notion. 
We call {\em permuton} a Borel probability measure $\pi$ on the square $[0,1]^2$ whose marginal laws are uniform:
namely, if $p_1, p_2 : [0,1]^2 \to [0,1]$ are the two coordinate projections and if $p_{1,*},p_{2,*} : \Mcal^1([0,1]^2) \to \Mcal^1([0,1])$ are the two induced applications between the spaces of probability measures, then we require that
$$p_{1,*}(\pi) = p_{2,*}(\pi) = \text{Lebesgue measure on }[0,1].$$
We denote $\Scal$ the space of permutons. If the space of probability measures $\Mcal^1([0,1]^2)$ is endowed with the weak topology of convergence in law, then $\Mcal^{1}([0,1]^2)$ is a compact metrisable space (see for instance \cite[Chapter 1]{Bil69}), and $\Scal$ is a closed subset of it, hence a compact metrisable space itself. 
Given a finite permutation $\sigma \in \sym(n)$, one associate with it a canonical permuton $\pi=\pi(\sigma)$, which is the probability measure on $[0,1]^2$ with density
$$f(\sigma;x,y) = n\,1_{\sigma(\lceil nx\rceil)=\lceil ny \rceil}.$$
It is easily checked that this density yields uniform marginal laws;
we refer to Figure \ref{fig:permuton} for an example.\medskip

Conversely, given a permuton $\pi \in \Scal$, we associate with $\pi$ a sequence of random permutations $(\sigma_n(\pi))_{n \in \N}$ with $\sigma_n(\pi) \in \sym(n)$ for any $n$. If $(x_1,y_1),\ldots,(x_n,y_n)$ is a family of points in the square $[0,1]^2$, we say that these points are in a general configuration if all the $x_i$'s are distinct, and if all the $y_i$'s are also distinct. 
With a general family of $n$ points, we associate a unique permutation $\sigma \in \sym(n)$ with the following property: if $\psi_1 : \{x_1,\ldots,x_n\} \to \lle 1,n\rre$ and $\psi_2 : \{y_1,\ldots,y_n\} \to \lle 1,n\rre$ are increasing bijections, then,
for $i \le n$
$$\sigma(\psi_1(x_i))=\psi_2(y_i).$$
for any $i \in \lle 1,n\rre$. We then say that $\sigma$ is the \emph{configuration} of the set of points; and we denote $\sigma = \conf((x_1,y_1),\ldots,(x_n,y_n))$. Intuitively, this means that the family of points $\{(x_1,y_1),\ldots,(x_n,y_n)\}$ is the diagram of the permutation $\sigma$, up to an increasing reparametrisation of the two axes. Now, for $\pi \in \Scal$, a family of independent points $(X_1,Y_1),\ldots,(X_n,Y_n)$ under $\pi^{\otimes n}$ is in general configuration with probability $1$, since the marginal laws of the coordinates are uniform. 
We can therefore define a random permutation
$$\sigma_n(\pi) = \conf((X_1,Y_1),\ldots,(X_n,Y_n)).$$
We refer to Figure \ref{fig:randompermutation} for an example, with a random permutation associated with the permuton $\pi$ that is supported on the disc inscribed in the square $[0,1]^2$, and that has a density proportional to $\frac{1}{\sqrt{1-4r^2}}$, where $r$ is the distance to the center $(\frac{1}{2},\frac{1}{2})$ of this disc (this ensures that the marginal laws are uniform).

\begin{center}
\begin{figure}[ht]
\begin{tikzpicture}[scale=0.8]
\foreach \x in {(1,2),(2,4),(3,5),(4,3),(5,6),(6,1)}
\fill [shift={\x},NavyBlue!45!white] (0,0) rectangle (-1,-1);
\draw [<->] (0,7) -- (0,0) -- (7,0);
\foreach \x in {0,1,2,3,4,5,6}
\draw [dotted] (\x,0) -- (\x,6);
\foreach \x in {0,1,2,3,4,5,6}
\draw [dotted] (0,\x) -- (6,\x);
\foreach \x in {1,2,3,4,5,6}
{\draw (\x,-0.1) -- (\x,0.1) ; \draw (-0.1,\x) -- (0.1,\x) ;};
\draw (0,-0.4) node {$0$};
\draw (6,-0.4) node {$1$};
\draw (-0.4,0) node {$0$};
\draw (-0.4,6) node {$1$};
\fill [NavyBlue!45!white] (8,2) rectangle (8.5,2.5);
\draw [dotted] (8,2) rectangle (8.5,2.5);
\draw [dotted] (8,1) rectangle (8.5,1.5);
\draw (9.55,2.2) node {$=n=6$};
\draw (9,1.2) node {$=0$};
\end{tikzpicture}
\caption{The density of the permuton $\pi(\sigma)$ associated with the permutation $\sigma=245361$.\label{fig:permuton}}
\end{figure}
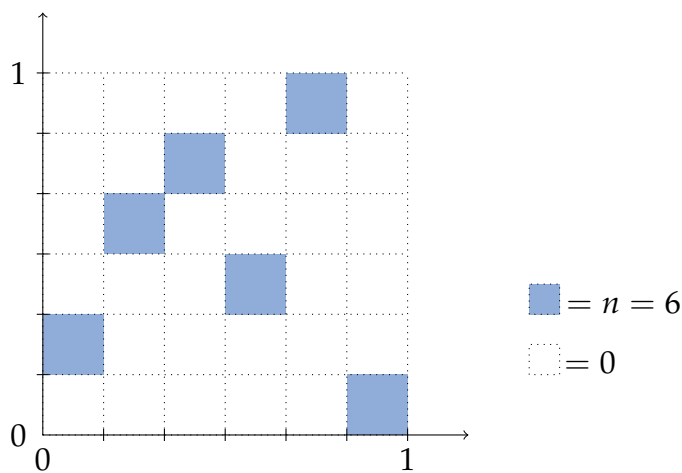
\end{center}

\begin{figure}[ht]
\begin{center}		
\begin{tikzpicture}[scale=0.05]
\foreach \x in {(1, 91), (2, 105), (3, 126), (4, 110), (5, 73), (6, 128), (7, 136), (8,
106), (9, 143), (10, 111), (11, 60), (12, 85), (13, 59), (14, 130), (15,
139), (16, 99), (17, 82), (18, 152), (19, 147), (20, 72), (21, 41), (22,
53), (23, 77), (24, 163), (25, 102), (26, 34), (27, 30), (28, 159), (29,
83), (30, 113), (31, 74), (32, 161), (33, 95), (34, 124), (35, 36), (36,
184), (37, 22), (38, 176), (39, 125), (40, 20), (41, 18), (42, 168),
(43, 165), (44, 93), (45, 178), (46, 24), (47, 49), (48, 162), (49,
177), (50, 149), (51, 11), (52, 101), (53, 80), (54, 55), (55, 194),
(56, 119), (57, 65), (58, 45), (59, 9), (60, 12), (61, 171), (62, 8),
(63, 67), (64, 181), (65, 26), (66, 28), (67, 68), (68, 5), (69, 172),
(70, 108), (71, 6), (72, 51), (73, 81), (74, 193), (75, 122), (76, 156),
(77, 154), (78, 19), (79, 42), (80, 27), (81, 4), (82, 141), (83, 103),
(84, 89), (85, 200), (86, 187), (87, 131), (88, 1), (89, 157), (90, 16),
(91, 189), (92, 44), (93, 71), (94, 3), (95, 86), (96, 66), (97, 52),
(98, 13), (99, 160), (100, 2), (101, 199), (102, 191), (103, 75), (104,
137), (105, 47), (106, 38), (107, 142), (108, 88), (109, 196), (110,
84), (111, 14), (112, 155), (113, 7), (114, 109), (115, 123), (116, 37),
(117, 197), (118, 10), (119, 198), (120, 150), (121, 76), (122, 188),
(123, 182), (124, 31), (125, 144), (126, 183), (127, 195), (128, 61),
(129, 46), (130, 35), (131, 185), (132, 135), (133, 63), (134, 118),
(135, 151), (136, 166), (137, 104), (138, 23), (139, 21), (140, 192),
(141, 29), (142, 15), (143, 33), (144, 179), (145, 190), (146, 170),
(147, 39), (148, 64), (149, 148), (150, 174), (151, 180), (152, 186),
(153, 169), (154, 17), (155, 54), (156, 96), (157, 164), (158, 145),
(159, 43), (160, 107), (161, 167), (162, 25), (163, 98), (164, 153),
(165, 70), (166, 173), (167, 175), (168, 87), (169, 140), (170, 116),
(171, 48), (172, 115), (173, 32), (174, 90), (175, 158), (176, 129),
(177, 133), (178, 40), (179, 112), (180, 134), (181, 121), (182, 127),
(183, 117), (184, 138), (185, 69), (186, 50), (187, 58), (188, 146),
(189, 132), (190, 62), (191, 92), (192, 79), (193, 56), (194, 57), (195,
97), (196, 94), (197, 114), (198, 120), (199, 100), (200, 78)}
{\fill[shift={\x},NavyBlue!75!white] (-1,-1) rectangle (0,0);}
\draw [<->] (210,0) -- (0,0) -- (0,210);
\draw [dotted] (200,0) -- (200,200) -- (0,200);
\draw [Orchid] (100,100) circle (100);
\draw (200,-2) -- (200,2);
\draw (-2,200) -- (2,200);
\draw (200,-5) node {$1$};
\draw (0,-5) node {$0$};
\draw (-5,200) node {$1$};
\draw (-5,0) node {$0$};
\end{tikzpicture}
\caption{Diagram of a random permutation with size $n=200$ associated with the permuton with density $1_{d((x,y),(\frac{1}{2},\frac{1}{2}))\leq \frac{1}{2}}\,\frac{2}{\pi\sqrt{1-4(x-\frac{1}{2})^2-4(y-\frac{1}{2})^2}}\DD{x}\DD{y}$.\label{fig:randompermutation}}
\end{center}
\end{figure}
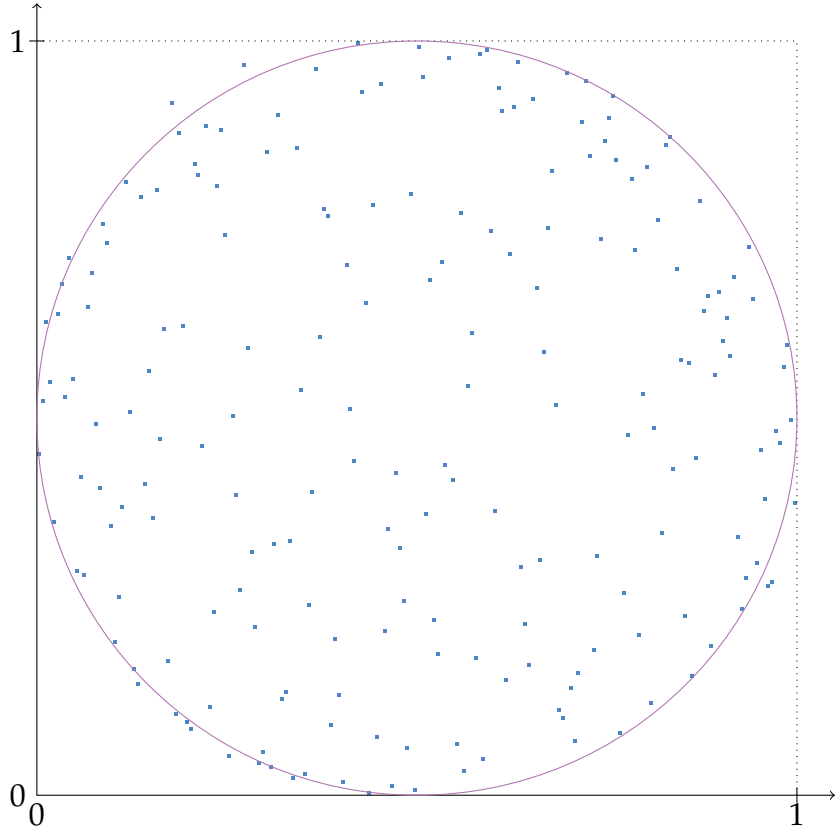

We denote $\Pi_n(\pi)$ the permuton in $\Scal$ associated with the random permutation $\sigma_n(\pi)$. In Figure \ref{fig:randompermutation}, it appears that for $n$ large, the random permuton $\Pi_n(\pi)$ looks a lot like the empirical measure of $n$ independent points under $\pi$. Indeed, the reparametrisation of $[0,1]^2$ associated with the order statistics $\psi_1$ and $\psi_2$ that were introduced in the definition of $\sigma_n(\pi) = \conf((X_1,Y_1),\ldots,(X_n,Y_n))$ can be shown to be very close to the identity map. This implies the following limiting result:
\begin{proposition}[Theorem 1.6.(ii) in \cite{HKMRS13}]\label{prop:concentrationpermutation}
For any permuton $\pi \in \Scal$, the sequence of random permutons $(\Pi_n(\pi))_{n \in \N}$ converges in probability towards $\pi$.
\end{proposition}

Thus, the space of permutons $\Scal$ parametrises certain models of random permutations 
whose diagrams have a property of asymptotic concentration. This property of concentration can be shown to be equivalent to the convergence of the densities of patterns (see Section \ref{subsec:permutationpattern}), and we shall prove later that the pattern densities of the models $(\sigma_n(\pi))_{n \in \N}$ are mod-Gaussian convergent.
\medskip

\subsection{The Thoma simplex}\label{subsec:thoma}
An analogous construction exists in the setting of random integer partitions, and the underlying theory goes back to the works of Kerov and Vershik in the 80's, see in particular \cite{KV77,KV81}. The space of parameters corresponding to these models is related to the classification of the totally positive sequences and of the positive specialisations of the Schur functions; see \cite{AESW51,Tho64}. We recall from the introduction that the \emph{Thoma simplex} is the set of pairs of sequences 
$$\Pcal = \left\{\omega=(\alpha,\beta) = ((\alpha_1\geq \alpha_2 \geq \cdots \geq 0),(\beta_1\geq \beta_2 \geq \cdots \geq 0))\,\,\big|\,\,\sum_{i=1}^\infty (\alpha_i+\beta_i) \leq 1\right\}.$$
In the following we shall denote $\gamma = 1-\sum_{i=1}^\infty (\alpha_i+\beta_i)$. As explained in Section \ref{subsec:informalpresentationmodels}, the system of Frobenius coordinates allows one to see any integer partition $\lambda \in \pym(n)$ as an element $\omega(\lambda) \in \Pcal$, with $\gamma=0$ if $n \geq 1$.
Conversely, fix $\omega \in \Pcal$, and an integer $n \geq 1$. It is known from \cite{Tho64,KV81} that the parameter $\omega$ corresponds to a function $\chi^\omega : \sym(\infty) \to \C$ which is
 \begin{itemize}
	\item tracial: $\chi^\omega(\sigma_1\sigma_2)=\chi^\omega(\sigma_2\sigma_1)$;
	\item normalised: $\chi^\omega(\mathrm{id}_{\N^*})=1$;
	\item non-negative definite: the matrix $(\chi^\omega(\sigma_i(\sigma_j)^{-1}))_{1\leq i,j\leq N}$ is Hermitian and non-nega\-tive for any finite family of permutations in $\sym(\infty)$.
\end{itemize}  
The formula for $\chi^\omega(\sigma)$ is
$$\chi^\omega(\sigma) = \prod_{\substack{c\text{ cycle of }\sigma \\ \text{with length }k\geq 2}} t(k,\omega) = \prod_{\substack{c\text{ cycle of }\sigma \\ \text{with length }k\geq 2}} \left(\sum_{i=1}^\infty (\alpha_i)^k + (-1)^{k-1} \sum_{i=1}^\infty (\beta_i)^k\right).$$
It is convenient to set $t(1,\omega)=1$ for any $\omega \in \Pcal$, so that one can take into account the fixed points of a permutation $\sigma$ in the previous formula. Thoma's theorem shows that $\omega \mapsto \chi^\omega$ is a bijection from $\Pcal$ to the set of extremal points in the compact convex set of non-negative definite normalised tracial functions on $\sym(\infty)$. Consider now the restriction of $\chi^\omega$ to a finite symmetric group $\sym(n)$: it is still tracial normalised non-negative, but it is not extremal anymore, and we have a decomposition
$$(\chi^\omega)_{|\sym(n)}  = \sum_{\lambda \in \pym(n)} \proba_{n,\omega}[\lambda]\,\chi^\lambda(\sigma),$$
where the $\chi^\lambda$'s are the characters of the irreducible Specht representations of $\sym(n)$, and $\proba_{n,\omega}[\cdot]$ is a probability measure on $\pym(n)$.
 We denote in the sequel $\lambda_n(\omega) \in \pym(n)$ a random integer partition chosen according to this probability measure $\proba_{n,\omega}$, and $\Omega_n(\omega) \in \Pcal$ the corresponding random parameter of the Thoma simplex. We refer to Figure \ref{fig:randompartition} for an example of random partition associated with a parameter $\omega \in \Pcal$.
\begin{figure}[ht]
 \begin{center}		
 \begin{tikzpicture}[scale=0.075]
 \draw [gray,thick] (-77.5,0) -- (-66.66,-10);
 \draw [gray,thick] (51.5,0) -- (50,-10);
 \draw [gray,thick] (27.5,0) -- (33.33,-10);
 \fill [gray!50!white] (6.5,0) -- (0,-10) -- (-11.5,0);
 \draw [gray,thick] (6.5,0) -- (0,-10) -- (-11.5,0);
 \draw (-78,78) -- (0,0) -- (52,52);
 \foreach \x in {0,...,51}
 \draw[shift={(\x,\x)}] (1,1) -- (0,2) -- (-1,1);
\foreach \x in {0,...,27}
 \draw[shift={(\x,\x)}] (1,3) -- (0,4) -- (-1,3);
\foreach \x in {0,...,6}
 \draw[shift={(\x,\x)}] (1,5) -- (0,6) -- (-1,5);
\foreach \x in {0,...,5}
 \draw[shift={(\x,\x)}] (1,7) -- (0,8) -- (-1,7);
\foreach \x in {0,...,3}
 \draw[shift={(\x,\x)}] (1,9) -- (0,10) -- (-1,9);
\draw (1,11) -- (0,12) -- (-1,11);
 \foreach \x in {1,...,77}
 \draw[shift={(-\x,\x)}] (1,1) -- (0,2) -- (-1,1);
\foreach \x in {1,...,11}
 \draw[shift={(-\x,\x)}] (1,3) -- (0,4) -- (-1,3);
\foreach \x in {1,...,9}
 \draw[shift={(-\x,\x)}] (1,5) -- (0,6) -- (-1,5);
\foreach \x in {1,...,3}
 \draw[shift={(-\x,\x)}] (1,7) -- (0,8) -- (-1,7);
\foreach \x in {1,...,2}
 \draw[shift={(-\x,\x)}] (1,9) -- (0,10) -- (-1,9);
\draw [->] (-95,0) -- (95,0);
\foreach \x in {-77.5,-11.5,-9.5,-3.5,-2.5,-0.5}
\fill[Red] (\x,0) circle (0.75);
\foreach \x in {0.5,3.5,5.5,6.5,27.5,51.5}
\fill[Orchid] (\x,0) circle (0.75);
\draw [dashed,Red] (-77.5,0) -- (-77.5,78.5);
\draw [dashed,Red] (-11.5,0) -- (-11.5,14.5);
\draw [dashed,Red] (-9.5,0) -- (-9.5,14.5);
\draw [dashed,Red] (-3.5,0) -- (-3.5,10.5);
\draw [dashed,Red] (-2.5,0) -- (-2.5,11.5);
\draw [dashed,Red] (-0.5,0) -- (-0.5,11.5);
\draw [dashed,Orchid] (51.5,0) -- (51.5,52.5);
\draw [dashed,Orchid] (27.5,0) -- (27.5,30.5);
\draw [dashed,Orchid] (6.5,0) -- (6.5,11.5);
\draw [dashed,Orchid] (5.5,0) -- (5.5,12.5);
\draw [dashed,Orchid] (3.5,0) -- (3.5,12.5);
\draw [dashed,Orchid] (0.5,0) -- (0.5,11.5);
\draw (-95,-10) -- (95,-10);
\foreach \x in {-66.66,-50,-33.33,-16.66,0,16.66,33.33,50,66.66}
\draw (\x,-11.5) -- (\x,-8.5);
\draw (0,-15.5) node {$0$};
\draw (50,-15.5) node {$\frac{1}{4}$};
\draw (33.33,-15.5) node {$\frac{1}{6}$};
\draw (-66.66,-15.5) node {$-\frac{1}{3}$};
\draw (105,0) node {$\Omega_n(\omega)$};
\draw (101,-10) node {$\omega$};
 \end{tikzpicture}
 \caption{Thoma parameter $\Omega_n(\omega)$ of a random partition $\lambda_n(\omega)$ with $n=200$ and $\omega = ((\frac{1}{4},\frac{1}{6},0,\ldots),(\frac{1}{3},0,\ldots))$.\label{fig:randompartition}}
 \end{center}
 \end{figure}
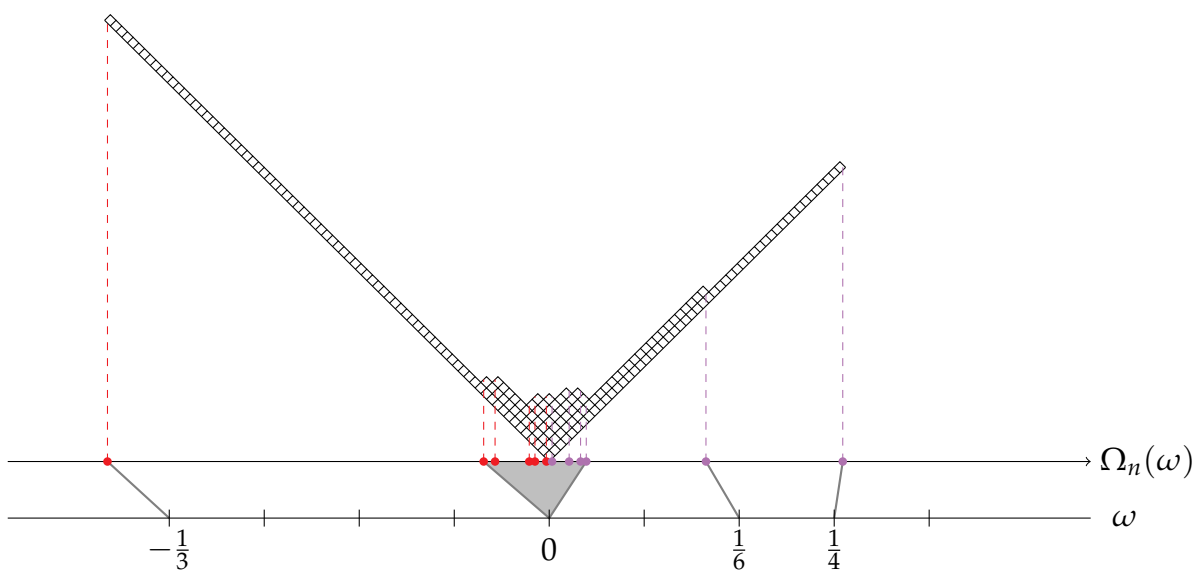 \medskip
 
 The law of the random integer partition $\lambda_n(\omega)$ can be computed by using the algebra of symmetric functions $\mathrm{Sym}$ (\emph{cf.}~\cite{Mac95} for the theory of this Hopf algebra, which we shall consider here over the field of real numbers). More precisely, denote $s_\lambda$ the Schur function with label $\lambda$, and if $k \geq 1$, denote $p_k$ the $k$-th power sum function; we recall that $\mathrm{Sym}=\R[p_1,p_2,\ldots]$, and that the family of Schur functions $(s_\lambda)_{\lambda \in \bigsqcup_{n \in \N} \pym(n)}$ form a linear basis of $\mathrm{Sym}$. We define for any $\omega \in \Pcal$ a morphism of algebras $\mathrm{Sym} \to \R$ by setting
 $ p_{k}(\omega) = t(k,\omega)$. Since $(p_k)_{k \geq 1}$ is a transcendence basis of $\mathrm{Sym}$ over $\R$, this rule allows one to define $s_\lambda(\omega)$ for any $\lambda$ integer partition. By using the Frobenius--Schur formula relatinf power sums to Schur functions, one can then show that
 $$\proba_{n,\omega}[\lambda] = (\dim \lambda)\,s_\lambda(\omega),$$
 where $\dim \lambda$ is the number of standard tableaux with shape $\lambda$.
These measures are called central measures in \cite{KV81}, and one has the following limiting result, which is quite clear in Figure \ref{fig:randompartition}.
 \begin{proposition}[Section 5 in \cite{KV81}] \label{prop:concentrationpartition}
 For any parameter $\omega \in \Pcal$ in the Thoma simplex, the sequence of random parameters $(\Omega_n(\omega))_{n \in \N}$ converges coordinate-wise and in probability towards $\omega$.
 \end{proposition}

\noindent This result implies that for any $\omega \in \Pcal$ and any $i \geq 1$,
 $$\frac{\lambda_{n,i}(\omega)}{n} \to_{\proba} \alpha_i\qquad;\qquad \frac{\lambda_{n,i}'(\omega)}{n} \to_{\proba} \beta_i.$$
The convergence coordinate by coordinate in $\Pcal$ can be reinterpreted as the convergence of moments of certain probability measures on $[-1,1]$ associated with the parameters of the Thoma simplex, and we shall prove in Section \ref{sec:dependency} that these moments are generically mod-Gaussian convergent after an appropriate renormalisation.

 \begin{remark}
 The reader might wonder how one can draw at random an integer partition $\lambda_n(\omega)$ as in Figure \ref{fig:randompartition}. The easiest way to do it is to use the Robinson--Schensted--Knuth algorithm in order the relate the central measures on integer partitions to certain models of random permutations obtained by shuffles; see \cite{Ful02} and  \cite[Section 12.2]{Mel17}.
 \end{remark}
\bigskip

\section{The algebras of observables}\label{sec:observables}
In the previous section, we introduced three compact spaces of parameters $\Gcal$, $\Scal$ and $\Pcal$ which parametrise models of random graphs, of random permutations and of random integer partitions; and all these models have an asymptotic concentration property (Propositions \ref{prop:concentrationgraph}, \ref{prop:concentrationpermutation} and \ref{prop:concentrationpartition}). The purpose of this section is to reinterpret the convergence in terms of observables of the combinatorial objects. The interest of these observables is that they will allow us to speak of fluctuations of models in Section \ref{sec:dependency}, and to prove an underlying mod-Gaussian convergence. We shall thus complete the following table:
\medskip

\renewcommand{\arraystretch}{1.2}
\begin{center}
 \begin{tabular}{|c|c|c|c|c|c|}
\hline \multirow{2}{*}{space} & \multirow{2}{*}{parameters} & random combina- & random & \multirow{2}{*}{observables} & algebra of \\
& & torial object & parameter & & observables\\
\hline \hline \multirow{2}{*}{$\Gcal$} & \multirow{2}{*}{$\gamma$} & \multirow{2}{*}{graph $G_n(\gamma)$} & \multirow{2}{*}{$\Gamma_n(\gamma)$} & subgraph & algebra of\\
& & & & densities  & graphs $\obsG$ \\
 \hline \multirow{2}{*}{$\Scal$} & \multirow{2}{*}{$\pi$} & \multirow{2}{*}{permutation $\sigma_n(\pi)$} & \multirow{2}{*}{$\Pi_n(\pi)$} & pattern & algebra of  \\
 & & & & densities & permutations $\obsS$\\
 \hline \multirow{2}{*}{$\Pcal$} & \multirow{2}{*}{$\omega=(\alpha,\beta)$} & \multirow{2}{*}{partition $\lambda_n(\omega)$} & \multirow{2}{*}{$\Omega_n(\omega)$} & Frobenius & algebra of \\
 & & & & moments &partitions $\obsP$\\
 \hline
\end{tabular}
\end{center}
\bigskip

For each space of parameters $\Mcal$, we shall exhibit a combinatorial algebra $\obs_{\mathfrak{M}}$,
endowed with a morphism of algebras $\Psi$ from $\obs_{\mathfrak{M}}$ to 
$\Ccal(\Mcal)$, the algebra of continuous functions on $\Mcal$.
These morphisms have the property that the convergence $m_n \to m$ 
in the compact metrisable space $\Mcal$ is equivalent to the convergence
of all the observables $\Psi(f)(m_n) \to \Psi(f)(m)$ for any $f \in \obs_{\mathfrak{M}}$. 
Equivalently (through the Stone--Weierstrass theorem), the image of $\Psi$ is a dense subalgebra of $\Ccal(\Mcal)$.
The combinatorics of the observables are easier to understand inside an abstract algebra $\obs_{\mathfrak{M}}$, instead of directly inside the algebra of continuous functions $\Ccal(\Mcal)$; this is one of the reasons why we use this point of view.
\medskip

\subsection{Subgraph counts}\label{subsec:graphdensity}
Let $\GG(n)$ denote the set of isomorphism classes of simple graphs with $n$ vertices. For instance, $\GG(4)$ is the set that consists in the $11$ following graphs:
\bigskip

\begin{center}
\begin{tikzpicture}[scale=0.8]
\foreach \x in {(0,0),(2.5,0),(5,0),(7.5,0),(10,0),(12.5,0),(1.25,-2.5),(3.75,-2.5),(6.25,-2.5),(8.75,-2.5),(11.25,-2.5)}
{\fill[shift={\x}] (0,0) circle (1.5pt);
\fill[shift={\x}] (0,1) circle (1.5pt);
\fill[shift={\x}] (1,1) circle (1.5pt);
\fill[shift={\x}] (1,0) circle (1.5pt);
\draw[shift={\x}] (-0.2,-0.2) node {\tiny $1$};
\draw[shift={\x}] (1.2,-0.2) node {\tiny $2$};
\draw[shift={\x}] (1.2,1.2) node {\tiny $3$};
\draw[shift={\x}] (-0.2,1.2) node {\tiny $4$};}
\draw[shift={(2.5,0)}] (1,0) -- (0,0);
\draw[shift={(5,0)}] (1,0) -- (0,0);
\draw[shift={(5,0)}] (1,1) -- (0,1);
\draw[shift={(7.5,0)}] (1,1) -- (0,0) -- (1,0);
\draw[shift={(10,0)}] (1,0) -- (0,0) -- (1,1) -- (1,0);
\draw[shift={(12.5,0)}] (0,1) -- (0,0) -- (1,0) -- (1,1) ;
\draw[shift={(1.25,-2.5)}] (1,0) -- (0,0) -- (0,1);
\draw[shift={(1.25,-2.5)}] (0,0) -- (1,1);
\draw[shift={(3.75,-2.5)}] (1,1) -- (0,1) -- (0,0) -- (1,0) -- (1,1) ;
\draw[shift={(6.25,-2.5)}] (1,1) -- (1,0) -- (0,0) -- (0,1);
\draw[shift={(6.25,-2.5)}] (0,0) -- (1,1);
\draw[shift={(8.75,-2.5)}] (1,0) -- (1,1) -- (0,0) -- (1,0);
 \draw[shift={(8.75,-2.5)}] (1,1) -- (0,1) -- (0,0);
\draw[shift={(11.25,-2.5)}] (0,0) -- (0,1) -- (1,0) -- (1,1) -- (0,1);
\draw[shift={(11.25,-2.5)}] (1,1) -- (0,0) -- (1,0);
\end{tikzpicture}
\end{center}

\noindent The \emph{algebra of graphs} $\obsG$ is the (commutative) algebra over the set of real numbers whose combinatorial basis consists in the elements of $\GG = \bigsqcup_{n \in \N} \GG(n)$, and whose product is defined by 
$$F_1 \times F_2 = F_1 \sqcup F_2,$$
where the right-hand side of the formula stands for the isomorphism class of the disjoint union of two graphs in the classes $F_1$ and $F_2$. The algebra $\obsG$ is graded by $\deg F = |V_F|=n$ if $F \in \GG(n)$.
\medskip

One can evaluate an element of $\obsG$ on a graph $G$ or on a graphon $\gamma$ by using the notion of subgraph count and of subgraph density. Let $F=(V_F,E_F)$ and $G=(V_G,E_G)$ be two finite graphs. A \emph{morphism} from $F$ to $G$ is a map $\phi : V_F \to V_G$ such that, if $\{v_1,v_2\} \in E_F$, then $\{\phi(v_1),\phi(v_2)\} \in E_G$. We denote $\hom(F,G)$ the set of morphisms from $F$ to $G$. The \emph{subgraph count} of $F$ in $G$ is the cardinality $|\hom(F,G)|$, and the \emph{subgraph density} of $F$ in $G$ is defined by
$$t(F,G)=\frac{|\hom(F,G)|}{|V_G|^{|V_F|}}.$$
This density is a real number between $0$ and $1$. 
\medskip

Given a finite graph $F$ and a graph function $g$, we can also define a density of $F$ in $g$ by the following formula:
$$t(F,g) = \int_{[0,1]^k} \left(\prod_{e=\{i,j\} \in E_F} g(x_i,x_j) \right)\DD{x_1}\DD{x_2}\cdots \DD{x_k},$$
where $V_F$ is identified with $\lle 1,k\rre$ if $k=|V_F|$. For instance, if $F$ is the graph of Figure \ref{fig:graphfunction}, then
$$t(F,g) = \int_{[0,1]^6} g(x_1,x_5)g(x_2,x_3)g(x_2,x_4) g(x_2,x_6) g(x_3,x_6)\DD{x_1}\DD{x_2}\cdots \DD{x_6}.$$
Let us describe some easy properties of this functions.
First, it is easily seen that $t(F,g)$ only depends on the equivalence class $\gamma$ of $g$ in $\Gcal$,
so $t(F,\gamma)=t(F,g)$ is well defined for any graphon $\gamma \in \Gcal$. 
Moreover, if $\gamma$ is the graphon associated with a finite graph $G$, then 
$$t(F,\gamma) = t(F,G)$$
for any finite graph $F$, see \cite[Equation (3.2)]{BCLSV08}.
Lastly, if $F$ and $F'$ are two finite graphs, then for any graphon $\gamma$,
we have $t(F \sqcup F', \gamma)=t(F,\gamma) \, t(F',\gamma)$.
In other words, the map $\Psi: \obsG \to \R^\Gcal$ defined by $\Psi(F) = t(F,\cdot)$
is a morphism of algebras.
The next statement connects this morphism to the topology of $\Gcal$ defined by the cut-metric.
\begin{proposition}[Theorem 5.1 in \cite{LS07} and Theorems 2.6 and 3.8 in \cite{BCLSV08}]
A sequence of graphons $(\gamma_n)_{n \in \N}$ converges in $(\Gcal,\delta_\oblong)$ to a graphon $\gamma$ if and only if, for any finite graph $F$, $t(F,\gamma_n) \to t(F,\gamma)$.
\end{proposition}
Equivalently, the range of the morphism of algebras $\Psi: \obsG \to \R^\Gcal$ defined by $\Psi(F) = t(F,\cdot)$ is included in $\Ccal(\Gcal)$, and it is a dense subalgebra of this algebra of continuous functions. 
This result can be used to prove Proposition \ref{prop:concentrationgraph}: 
indeed, one shows easily that for any finite graph $F$ and any graphon $\gamma \in \Gcal$,
$$|\esper[t(F,\Gamma_n(\gamma))] - t(F,\gamma)| \leq \frac{|V_F|^2}{2n}\qquad;\qquad \var(t(F,\Gamma_n(\gamma))) \leq \frac{3\,|V_F|^2}{n}$$
see \cite[Lemma 2.4]{LS06}; whence the asymptotic concentration by using the Bienaymé--Cheby\-shev inequality. In Section \ref{subsec:graphmodgauss}, we shall prove that $t(F,\Gamma_n(\gamma))$ is actually mod-Gaussian after an appropriate scaling.

\begin{remark}[Kernel of the morphism $\obsG \to \Ccal(\Gcal)$]
The morphism of algebras $\Psi: \obsG \to \Ccal(\Gcal)$ is not injective, since the one-point graph $\bullet$ has density $t(\bullet,\gamma)=1$ for any graphon $\gamma$. In \cite{Whit32,ELS79}, it is shown that the graph densities $t(F,\cdot)$ are algebraically independent over $\R$ when $F$ runs over the set of isomorphism classes of \emph{connected} finite graphs. Therefore, the kernel of $\Psi$ is actually the ideal of $\obsG$ generated by the difference of graphs $\bullet - \emptyset$. We refer to \cite{BCLSV06} for a general survey of the properties of enumeration of graph homomorphisms.
\end{remark}

\begin{remark}
  One could also work with \emph{embeddings} of $F$ into $G$, that is morphisms that are injective maps $V_F \to V_G$. Set
$$t_0(F,G) = \frac{|\mathrm{emb}(F,G)|}{|V_G|^{\downarrow |V_F|}},$$
where $\mathrm{emb}(F,G)$ is the set of embeddings of $F$ into $G$, and $n^{\downarrow k}$ denotes the falling factorial $n(n-1)(n-2)\cdots (n-k+1)$ --- thus, $|V_G|^{\downarrow |V_F|}$ is the number of injective maps from $V_F$ to $V_G$. The two quantities $t(F,G)$ and $t_0(F,G)$ are close when $G$ is sufficiently large:
$$|t(F,G)-t_0(F,G)| \leq \frac{1}{|V_G|}\,\binom{|V_F|}{2},$$
see \cite[Lemma 2.1]{LS06}. 
However, $t_0$ cannot be extended an algebra morphism from $\obsG$ to $\Ccal(\Gcal)$.
Therefore we prefer to work with the first definition of the density $t(F,G)$. 
\end{remark}
\medskip

\subsection{Permutation patterns}\label{subsec:permutationpattern}
As before, $\sym(n)$ is the symmetric group of order $n$, and we shall denote $\R\sym(n)$ the real vector space that is spanned by the permutations of size $n$.
We introduce a product $\R\sym(m) \times \R\sym(n) \to \R\sym(m+n)$ which we call the \emph{graphical shuffle product};
this operation was already considered, also in connection with pattern occurrences,
in \cite{vargas2014hopf} and \cite[Section 4]{BBFGP16}. 
\medskip

If $\sigma \in \sym(m)$ and $\tau \in \sym(n)$, 
consider two parts $A$ and $B$ of $\lle 1,m+n\rre$ with cardinality $|A|=|B|=m$,
The $(A,B)$-shuffle product of $\sigma$ and $\tau$, denoted $\sigma \,\,{}_A^B\!\!\times \tau$, is the unique permutation $\rho \in \sym(m+n)$ such that
\begin{itemize}
  \item $\rho$ maps the subset $A$ onto $B$ (and hence its complement $\overline{A}$ onto $\overline{B}$);
  \item the patterns induced by $\rho$ on the set $A$ and $\overline{A}$ are $\sigma$ and $\tau$ respectively.
\end{itemize}
Explicitly, if we denote $\psi_{m,A}$ and $\psi_{m,B}$ (respectively, $\psi_{n,\overline{A}}$ and $\psi_{n,\overline{B}}$) the two increasing bijections from $\lle 1,m\rre$ to $A$ and to $B$ (respectively, from $\lle 1,n\rre$ to the complement subsets $\overline{A}$ and $\overline{B}$), then $\rho$ is given by
$$\rho(k) = \begin{cases}
	\psi_{m,B} \circ \sigma \circ \psi_{m,A}^{-1}(k) &\text{if }k \in A,\\
	\psi_{n,\overline{B}} \circ \tau \circ \psi_{n,\overline{A}}^{-1}(k) &\text{if }k \in \overline{A}.
\end{cases}$$
An example is given on Figure~\ref{fig:ABshuffle}.
Informally we take the diagrams of the permutations $\sigma$ of $\tau$ and 
merge them by shuffling independently the $x$-coordinates and the $y$-coordinates of the dots,
which explain the name of graphical shuffle product.
\begin{figure}[ht]
\begin{tikzpicture}[scale=0.7,font=\scriptsize]
  \begin{scope}[color=red]
    \foreach \x in {(1,2),(2,3),(3,1)}
\fill \x circle (1mm);
\draw [<->] (0,4) -- (0,0) -- (4,0);
\foreach \x in {1,2,3}
{\draw (\x,-0.1) -- (\x,0.1) ; \draw (-0.1,\x) -- (0.1,\x) ;};
\draw (2,-0.6) node {diagram of $\sigma$};
  \end{scope}
  \begin{scope}[xshift=5.2cm, yshift=-5mm,color=blue]
        \foreach \x in {(1,3),(2,1),(3,4),(4,2)}
    \fill \x circle (1mm);
    \draw [<->] (0,5) -- (0,0) -- (5,0);
    \foreach \x in {1,2,3,4}
    {\draw (\x,-0.1) -- (\x,0.1) ; \draw (-0.1,\x) -- (0.1,\x) ;};
    \draw (2.5,-0.6) node {diagram of $\tau$};
  \end{scope}
  \begin{scope}[xshift=13cm, yshift=-3mm,scale=.6]
        \foreach \x in {(1,5),(4,7),(5,3)}
          \fill[red] \x circle (1.66mm);
        \foreach \x in {(2,4),(3,1),(6,6),(7,2)}
            \fill[blue] \x circle (1.66mm);
    \draw [<->] (0,8) -- (0,0) -- (8,0);
    \foreach \x in {1,2,3,4,5,6,7}
    {\draw (\x,-0.1) -- (\x,0.1) ; \draw (-0.1,\x) -- (0.1,\x) ;};
    \draw (4,-1.4) node {
    \begin{tabular}{c}
      diagram of $\rho=\sigma \,\,{_A^B\!\!\times \tau}$\\
      for $A=\{1,4,5\}$ and $B=\{3,5,6\}$.
    \end{tabular}};
  \end{scope}
\end{tikzpicture}\vspace{-4mm}
  \caption{Example of $(A,B)$-shuffle of permutations.}
  \label{fig:ABshuffle}
\end{figure}
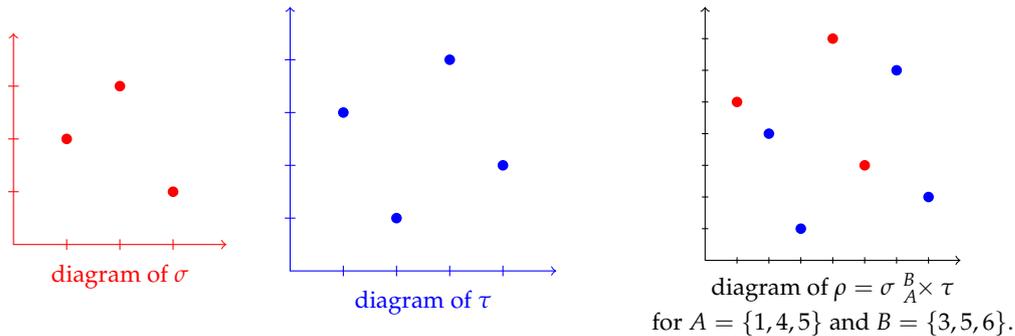

Now, the graphical shuffle product $\sigma \times \tau$ in $\R\sym(m+n)$ is defined as a linear combination of all the possible $(A,B)$-shuffle products, $A$ and $B$ being arbitrary subsets with cardinality $m$ in $\lle 1,m+n\rre$:
$$\sigma \times \tau = \frac{m!\,n!}{(m+n)!}\sum_{|A|=|B|=m} \sigma \,\,{}_A^B\!\!\times \tau.$$
Notice that the number of terms in the sum of the right-hand side is $(\frac{(m+n)!}{m!\,n!})^2$, and that some permutations $\rho \in \sym(m+n)$ may appear with multiplicity larger than $1$ in this sum. For instance, if $\sigma=[12]$ and $\tau=[21]$, then their graphical shuffle product is the linear combination
\begin{align*}
[12] \times [21] &= \frac{1}{6}([1243] +[1324] + [2134] + [2413] + [3142] + [3421] +[4231]+[4312]) \\
&\quad +\frac{1}{3} ([1342]+[1423]+[2314] + [2431]+ [3124]+ [3241]+[4132] +[4213])\\
&\quad + \frac{1}{2} ([1432] + [2341]  + [3214] + [4123] ).
\end{align*}
The \emph{algebra of permutations} is the real algebra $\obsS$ which as a vector space is equal to the direct sum $\bigoplus_{n=0}^\infty \R\sym(n)$, and which is endowed with the graphical shuffle product. Since the graphical shuffle product $\sigma \times \tau$ encodes all the ways of mixing graphically the two patterns $\sigma$ and $\tau$, this operation is clearly commutative, and one sees readily that it is also associative. 
Thus, $\obsS$ is a commutative algebra whose combinatorial basis consists in
all finite permutations $\sigma \in \bigsqcup_{n \in \N} \sym(n)$. 
It is graded by $\deg \sigma = |\sigma|=n$ if $\sigma \in \sym(n)$.
\bigskip

Recall from the introduction that if $\tau \in \sym(k)$ and $\sigma \in \sym(n)$, then $\tau$ is a pattern of $\sigma$ if there exists a subset $\{a_1<a_2<\cdots <a_k\}$ of $\lle 1,n\rre$ such that $\sigma(a_i)<\sigma(a_j)$ if and only if $\tau(i)<\tau(j)$. As for graphs, we can then define the \emph{pattern density} of $\tau$ in $\sigma$ by the ratio
$$t(\tau,\sigma) = \frac{\mathrm{occ}(\tau,\sigma)}{\binom{n}{k}},$$
where the numerator of this fraction is the number of \emph{occurrences} of $\tau$ in $\sigma$, that is the number of subsets $\{a_1<a_2<\cdots<a_k\} \subset \lle 1,n\rre$ that make appear $\tau$ as a pattern of $\sigma$. On the other hand, if $\pi$ is a permuton, we can also define the density of $\tau \in \sym(k)$ in $\pi$ by the following formula:
\begin{align*}
t(\tau,\pi) &= \int_{([0,1]^2)^k} 1_{\conf((x_1,y_1),\ldots,(x_k,y_k)) = \tau} \,\pi(\!\DD{x_1}\DD{y_1})\cdots\pi(\!\DD{x_k}\DD{y_k}) \\
&= \proba_{\pi^{\otimes k}}[\conf((X_1,Y_1),\ldots,(X_k,Y_k))=\tau].
\end{align*}
If $\sigma$ is a permutation and if $\pi(\sigma)$ is the associated permuton,
then we may have $t(\tau,\sigma) \ne t(\tau,\pi(\sigma))$.
The difference is however small for $\sigma$ sufficiently large.
More precisely, if $\tau \in \sym(k)$ and $\sigma \in \sym(n)$, then
$$|t(\tau,\sigma) - t(\tau,\pi(\sigma))| \leq \frac{1}{n}\binom{k}{2} = \frac{1}{|\sigma|}\binom{|\tau|}{2};$$
see \cite[Lemma 3.5]{HKMRS13}.
Therefore, in order to evaluate a pattern density in a permutation $\sigma$, we can use either $t(\tau,\sigma)$ or $t(\tau,\pi(\sigma))$. This is an important difference between the permutons and the two other classes of models, and it will lead us to state two distinct sets of limiting results for the models of random permutations. The second choice of observables $t(\tau,\pi(\sigma))$ is the better one when looking at random permutations and fluctuations thereof, for the following reason:
\begin{proposition}[Graphical shuffle products of permutations]
Consider the linear map $\Psi : \obsS \to \R^\Scal$ which is defined by $\Psi(\tau) = t(\tau,\cdot)$, the right-hand of this formula being a function on permutons. This map is a morphism of algebras, and its range is a dense subalgebra of the algebra of continuous functions $\Ccal(\Scal)$.
\end{proposition}

\begin{proof}
The second part of the proposition is contained in \cite[Theorem 1.8]{HKMRS13}, which proves that a sequence of permutons $(\pi_n)_{n \in \N}$ converges with respect to the weak topology on probability measures if and only if the sequences of observables $(t(\tau,\pi_n))_{n \in \N}$ converge for any finite permutation $\tau$. Therefore, each observable $t(\tau,\cdot)$ is continuous on $\Scal$, and these observables separate the permutons. Hence, by Stone--Weierstrass, $\Psi(\obsS)$ is a dense subalgebra of $\Ccal(\Scal)$. \medskip

Let us now prove that for any permutations $\tau_1$ and $\tau_2$ with sizes $k_1$ and $k_2$, and for any permuton $\pi$, 
\begin{equation}
t(\tau_1,\pi)\,t(\tau_2,\pi) = \frac{k_1!\,k_2!}{(k_1+k_2)!}\sum_{|A|=|B|=k_1} t(\tau_1\,\,{}_A^B\!\!\times \tau_2,\pi),\label{eq:alg_morph_perm}
\end{equation}
where the sum runs over pairs of subsets of $\lle 1,k_1+k_2\rre$ with cardinality $k_1$. This amounts to the property of morphism of algebras that remains to be proved. In \cite[Proposition 4.5]{BBFGP16}, an analogue result is given for the observables of permutations instead of permutons: for any $n \ge 0$ and any $\sigma \in \sym(n)$, 
\begin{equation}
t(\tau_1,\sigma)\,t(\tau_2,\sigma) = \frac{k_1!\,k_2!}{(k_1+k_2)!}\sum_{|A|=|B|=k_1} t(\tau_1\,\,{}_A^B\!\!\times \tau_2,\sigma) + O\!\left(\frac{1}{n}\right),\label{eq:alg_morph_perm2}
\end{equation}
where the constant in the $O(\cdot)$ depends on $\tau_1$ and $\tau_2$ but does not depend on $n$ and $\sigma$. Since this formula is true for any $\sigma$, it is also true for the random permutations $\sigma_n(\pi)$ attached to a permuton $\pi \in \Scal$. Since $t(\tau,\sigma_n(\pi)) = t(\tau,\Pi_n(\pi)) + O(\frac{1}{n})$ and since $\lim_{n \to \infty} \Pi_n(\pi) = \pi$ by Proposition \ref{prop:concentrationpermutation}, Equation \eqref{eq:alg_morph_perm} follows by applying Equation \eqref{eq:alg_morph_perm2} to $\sigma_n(\pi)$, and then by taking the limit of both sides when $n$ goes to infinity.
\end{proof}
\medskip
  
In Section \ref{subsec:permmodgauss}, we shall prove that the sequences of random variables $(t(\tau,\sigma_n(\pi)))_{n \in \N}$ and $(t(\tau,\Pi_n(\pi)))_{n \in \N}$ are mod-Gaussian convergent after an appropriate scaling.

\begin{remark}[Kernel of the morphism $\obsS \to \Ccal(\Scal)$]
As for graphons, the morphism of algebras $\Psi: \obsS \to \Ccal(\Scal)$ is not injective, since 
the image of the permutation $\tau=[1]$ of degree $1$ is the constant function $1$ on $\Scal$. This implies identities such as
\begin{align*}
t([21],\cdot) &= t([21],\cdot)\times t([1],\cdot) \\
&= t([321],\cdot)+\frac{2}{3}\,(t([231],\cdot)+t([312],\cdot))+\frac{1}{3}\,(t([132],\cdot)+t([213],\cdot)).
\end{align*}
Thus, the kernel of the morphism $\Psi$ contains the ideal of $\obsS$ generated by $[1]-[\hspace*{2mm}]$, where $[\hspace*{2mm}]$ denotes the empty permutation of size $0$. We do not know whether this ideal is the full kernel of $\Psi$.
\end{remark}
\medskip

\subsection{Observables of partitions}
We finally construct an algebra $\obsP$ of observables of partitions and of parameters of the Thoma simplex, which will play the same role as the role of $\obsG$ and $\obsS$ for graphons and permutons. This algebra is closely related to the Kerov--Olshanski algebra of polynomial functions on partitions which was introduced in \cite{KO94}, and whose combinatorics are detailed in \cite{IO02} and in \cite[Part III]{Mel17}. If $\lambda$ and $\mu$ are two integer partitions of size $m$ and $n$, we denote $\lambda \times \mu$ their disjoint union, that is the integer partition with size $m+n$ and whose parts are those of $\lambda$ and those of $\mu$, reordered in order to get a non-increasing sequence. For instance,
$$(4,4,1) \times (5,3,1) = (5,4,4,3,1,1).$$
The \emph{algebra of partitions} $\obsP$ is the real algebra spanned by the partitions in $\pym = \bigsqcup_{n \in \N} \pym(n)$, endowed with the product defined above. It is commutative and graded by $\deg \lambda = |\lambda| = \lambda_1+\lambda_2+\cdots + \lambda_r$.\medskip

Given an element $\omega \in \Pcal$ and an integer $k\geq 1$, we recall the definition
$$t(k,\omega) = p_k(\omega) = \begin{cases}
	1 &\text{if }k=1\\
	\sum_{i=1}^\infty (\alpha_i)^k + (-1)^{k-1} \sum_{i=1}^\infty (\beta_i)^k &\text{if }k\geq 2.
\end{cases}$$
We extend this definition to integer partitions $\rho$ by setting $$t(\rho,\omega) = t(\rho_1,\omega)\,t(\rho_2,\omega)\cdots t(\rho_r,\omega)$$ if $\rho = (\rho_1,\rho_2,\ldots,\rho_r)$. We thus obtain a map $\Psi : \obsP \to \R^{\Pcal}$ which is clearly a morphism of algebras. The following result is a bit less trivial. 
We endow $\Pcal$ with the topology of convergence of all coordinates:
a sequence $(\omega_n)_{n \in \N} = ((\alpha_{n,1},\alpha_{n,2},\ldots),(\beta_{n,1},\beta_{n,2},\ldots))_{n \in \N}$ converges towards a parameter $\omega = (\alpha,\beta)$ if and only if, for any $i \geq 1$, $\alpha_{n,i} \to \alpha_i$ and $\beta_{n,i} \to \beta_i$. This makes $\Pcal$ into a metrisable compact space.
\begin{proposition}[Convergence in the Thoma simplex]
A sequence of Thoma parameters $(\omega_n)_{n \in \N}$ converges towards a parameter $\omega \in \Pcal$ if and only if $t(\rho,\omega_n) \to t(\rho,\omega)$ for any integer partition $\rho \in \pym$. Equivalenty, the range of the morphism of algebras $\Psi : \obsP \to \R^{\Pcal}$ is a dense subalgebra of $\Ccal(\Pcal)$.
\label{prop:topology_obs_partitions}
\end{proposition}

\begin{proof}
To any parameter $\omega \in \Pcal$, we associate a probability measure $\theta_\omega$ on $[-1,1]$:
$$\theta_\omega = \sum_{i=1}^\infty \alpha_i\,\delta_{\alpha_i} + \sum_{i=1}^\infty \beta_i\,\delta_{-\beta_i} + \gamma\,\delta_0.$$
Notice then that $t(k,\omega)$ is the $(k-1)$-th moment of this measure:
$$t(k,\omega) = \int_{-1}^1 s^{k-1}\,\theta_\omega(\!\DD{s}).$$
Therefore, the convergence of all the observables $t(\rho,\omega_n)$ is equivalent to the convergence of all the moments of the measures $\theta_{\omega_n}$.
Since these measures are compactly supported, we have:
$$\big(\forall \rho\in \pym,\,\,\,t(\rho,\omega_n) \to t(\rho, \omega)\big) \quad \iff \quad \big(\theta_{\omega_n} \rightharpoonup \theta_\omega\big),$$
where the convergence on the right-hand side is with respect to the weak topology on probability measures on $[-1,1]$. It is then shown in \cite[proof of Theorem 12.19]{Mel17} that the restriction of the weak topology of $\Mcal^1([-1,1])$ to the set of measures $\{\theta_\omega,\,\,\omega \in \Pcal\}$ is the topology of convergence of all the coordinates of the sequences.
\end{proof}
\medskip

In Section~\ref{subsec:partmodgauss},
we establish the mod-Gaussian convergence of the sequences of observables $(t(\rho,\Omega_n(\omega)))_{n \in \N}$, for any $\rho \in \pym$. 

\begin{remark}[Kernel of the morphism $\obsP \to \Ccal(\Pcal)$]
As before, the morphism of algebras $\Psi : \obsP \to \Ccal(\Pcal)$ is not injective, since $t(1,\omega)=1$ for any parameter $\omega \in \Pcal$. The kernel of $\Psi$ is actually the ideal spanned by $1 - \emptyset$, which is equivalent to saying that the observables $t(k\geq 2,\cdot)$ are algebraically independent in $\R^{\Pcal}$. To prove this result, we set for $\mu \in \pym(n)$ and $k \geq 1$:
$$p_k(\mu) = \sum_{i=1}^d (a_i(\mu))^k + (-1)^{k-1} \sum_{i=1}^d (b_i(\mu))^k = n^k\,t(k,\omega(\mu)).$$
We extend this definition to partitions $\rho$ by setting $p_\rho(\mu) = \prod_{i=1}^r p_{\rho_i}(\mu) = n^{|\rho|}\,t(\rho,\omega(\mu))$. Suppose then that one has a relation
$$\sum_{\rho} c_\rho\, t(\rho,\cdot)=0$$
where the integer partitions $\rho$ appearing in the sum have all their parts larger than $2$ (we also allow the empty partition). Let $d$ be the maximal size of a partition $\rho$ appearing in the sum, and $\lambda$ be an arbitrary integer partition of size $n$. By evaluating the vanishing linear combination on $\lambda$ and by multiplying by $n^{d}$, we obtain
$$\sum_{\rho} c_\rho\,p_{\rho \times 1^{d-|\rho|}}(\lambda) = 0$$
since $p_1(\lambda) = n$. However, by \cite[Proposition 1.5]{IO02}, the functions $p_{k\geq 1}$ on $\pym$ are algebraically independent. Therefore, all the coefficients $c_\rho$ vanish, whence the result.
\end{remark}
\bigskip

\section{Dependency structures and mod-Gaussian convergence}\label{sec:dependency}
We can finally study the fluctuations of the random models $(\Gamma_n(\gamma))_{n \in \N}$, $(\Pi_n(\pi))_{n \in \N}$ and $(\Omega_n(\omega))_{n \in \N}$. The easiest case to deal with is by far the one of graphs, because of the identity of observables $t(F,G) = t(F,\gamma(G))$, and because of the clear dependency structure in the random variable $n^{|F|}\,t(F,G_n(\gamma))$. The two other models will require a bit more work, but the proofs all follow the same lines. On the other hand, we shall see that the computation of the limiting variances $\sigma^2$ and of the limiting third cumulants $L$ amounts to making certain operations in the algebras of observables $\obsG$, $\obsS$ and $\obsP$; this is the reason why we discussed these algebraic structures.
\medskip

\subsection{Fluctuations of random graphs and graphons}\label{subsec:graphmodgauss}
In this section, we fix a graphon $\gamma \in \Gcal$, a graph function $g$ in the equivalence class $\gamma$, and a graph $F$ with $k$ vertices.
We set $S_n(F,\gamma)=n^{|F|}\,t(F,G_n(\gamma))$, and we are going to prove that $(S_n(F,\gamma))_{n \in \N}$ satisfies the hypotheses \eqref{eq:cumulant1}, \eqref{eq:cumulant2} and \eqref{eq:cumulant3}. 
Given two graphs $F$ and $G$ with vertex-set  $\lle 1,k\rre$ and $\lle 1,n \rre$,
we can decompose the function $n^{k}\, t(F,G)=|\hom(F,G)|$ as a sum of $n^k$ functions:
\begin{equation}
  |\hom(F,G)| = \sum_{\psi : \lle 1,k\rre \to \lle 1,n\rre} A_{F,\psi}(G),\quad\text{with }A_{F,\psi}(G) = \begin{cases}
	1 &\text{if }\psi \text{ is a morphism from $F$ to $G$},\\
	0 &\text{otherwise}.
\end{cases}
\label{eq:tG_Sum}
\end{equation}
The functions $A_{F,\psi}$ in turns factorize as
\begin{equation}
  A_{F,\psi} = \prod_{\{i,j\} \in E_{F}} B_{\psi(i),\psi(j)},
  \label{eq:A_Prod_B}
\end{equation}
where $B_{i,j}(G)$ is $1$ if there is an edge between $i$ and $j$ in $G$, and $0$ otherwise.
\medskip

We are now interested to specialize this to the random graph $G=G_n(\gamma)$.
By construction, the law of the random variables $B_{i,j}=B_{i,j}(G_n(\gamma))$ is then characterized by
$$\proba[B_{i,j}=1|(X_1,\ldots,X_n)]=g(X_i,X_j),$$
where the $X_i$ are independent uniform random variables in $[0,1]$.
A way to construct these random variables $B_{i,j}$ so that they are independent conditionally to $(X_1,\ldots,X_n)$ is as follows: we introduce other independent uniform random variables $U_{i,j} \in [0,1]$ for any pair $\{i,j\}$, and then set
$$B_{i,j} = 1_{(U_{i,j} \leq g(X_i,X_j))}.$$
In this setting, $B_{i,j}$ is measurable with respect to the $\sigma$-field generated by $(U_{i,j},X_i,X_j)$, so if $\{i,j\} \cap \{i',j'\} = \emptyset$, then $B_{i,j}$ and $B_{i',j'}$ are independent. 
Summing up and using the short notation $A_{F,\psi}=A_{F,\psi}(G_n(\gamma))$, we have
\[S_n(F,\gamma)=\sum_{\psi : \lle 1,k\rre \to \lle 1,n\rre} A_{F,\psi},\]
with the following dependency structure.
\begin{lemma}\label{lem:graphondependencygraph}
The graph $G_n$
\begin{itemize}
	\item with vertex set $V_n=\{\psi : \lle 1,k\rre \to \lle 1,n\rre\}$,
	\item and with an edge between $\psi$ and $\phi$ if $\psi(i)=\phi(j)$ for some indices $i,j \in \lle 1,k\rre$
\end{itemize}
is a dependency graph for the family of random variables $(A_{F,\psi})_{\psi \in V_n}$ involved in the expansion of $S_n(F,\gamma)$.
\end{lemma} 
\begin{proof}
Consider two families $\{\psi_c\}_{c \in C}$ and $\{\psi_d\}_{d\in D}$ without any edge between $C$ and $D$. This means that the sets $\mathbb{C}=\bigcup_{c \in C} \psi_c(\lle 1,k\rre)$ and $\mathbb{D}=\bigcup_{d \in D} \psi_d(\lle 1,k\rre)$ do not intersect. However, the random vector $(A_{F,\psi_c})_{c \in C}$ (respectively, $(A_{F,\psi_d})_{d\in D}$) is measurable with respect to the $\sigma$-algebra generated by the variables 
$$\{X_i,\,U_{i,j}\}_{i,j \in \mathbb{C}}\quad(\text{respectively, }\{X_i,\,U_{i,j}\}_{i,j \in \mathbb{D}}).$$
Therefore, these two random vectors are independent.
\end{proof}\medskip

The parameters of the dependency graph $G_n$ can be chosen as follows. Obviously, $|A_{F,\psi}| \leq 1$ almost surely, so one can take $A=1$. The total number of vertices is $N_n=N_{n,k} = n^k$, and the number of maps $\phi$ connected or equal to another map $\psi$ is bounded by $k^2\,n^{k-1}$, so we can take $D_n=D_{n,k} = k^2\,n^{k-1}$. Hence, the previous lemma and Theorem \ref{thm:cumulantestimates} imply the hypothesis \eqref{eq:cumulant1} for $S_n(F,\gamma)$ with parameters $(D_{n,k},N_{n,k},A) = (k^2\,n^{k-1},n^k,1)$. To complete this analysis, we need to compute
$$\sigma^2(F,\gamma) = \lim_{n \to \infty} \frac{\var(S_n(F,\gamma))}{k^2\,n^{2k-1}} \qquad;\qquad L(F,\gamma) = \lim_{n \to \infty} \frac{\kappa^{(3)}(S_n(F,\gamma))}{k^4\,n^{3k-2}}.$$ 
It is a bit clearer to work with distinct graphs $F,G,H$ with $k$ vertices, and to evaluate the rescaled \emph{joint cumulants}
$$ \frac{\kappa(S_n(F,\gamma),S_n(G,\gamma))}{k^2\,n^{2k-1}}  \quad\text{and}\quad  \frac{\kappa(S_n(F,\gamma),S_n(G,\gamma),S_n(H,\gamma))}{k^4\,n^{3k-2}}.$$
We refer to \cite[Section 9.2]{FMN16} for details on the notion of joint cumulants, which generalise the cumulants introduced in our Definition \ref{def:cumulantmethod}. For our purpose, it is sufficient to know that 
$$\kappa^{(r)}(X)=\kappa(\underbrace{X,X,\ldots,X}_{r \text{ occurrences}})$$
and that 
\begin{align*}
\kappa(X,Y) &= \cov(X,Y) = \esper[XY] - \esper[X]\esper[Y];\\
\kappa(X,Y,Z) &= \esper[XYZ] - \esper[XY]\esper[Z]- \esper[XZ]\esper[Y]- \esper[YZ]\esper[X] + 2\,\esper[X]\esper[Y]\esper[Z].
\end{align*}

\begin{proposition}[Limiting first cumulants of subgraph counts]
Denote $\obs_{\GG,k}$ the vector space spanned by the graphs with $k$ vertices in $\obsG$. There exists two linear maps
\begin{align*}
\kappa_2 &: \obs_{\GG,k} \otimes \obs_{\GG,k} \to \obsG ;\\
\kappa_3 &: \obs_{\GG,k} \otimes \obs_{\GG,k} \otimes \obs_{\GG,k} \to \obsG 
\end{align*}
such that, for any finite graphs $F$, $G$ and $H$ with $k\geq 1$ vertices, and for any graphon $\gamma$,
\begin{align*}
\frac{\kappa(S_n(F,\gamma),S_n(G,\gamma))}{k^2\,n^{2k-1}} &= \kappa_2(F,G)(\gamma) +O(n^{-1});\\
\frac{\kappa(S_n(F,\gamma),S_n(G,\gamma),S_n(H,\gamma))}{k^4\,n^{3k-2}} &=\kappa_3(F,G,H)(\gamma) + O(n^{-1}),
\end{align*}
with constants in the $O(\cdot)$'s which depend only on $k$.
\end{proposition}
\noindent Because of this description, in order to compute the limiting first cumulants of the random variables $S_n(F,\gamma)$, it will suffice to evaluate two observables $\kappa_2(F,F)$ and $\kappa_3(F,F,F)$ on the graphon $\gamma$ (here we make a slight abuse of notation by writing $f(\gamma)$ instead of $\Psi(f)(\gamma)$ for $f \in \obsG$). This is the main reason why we introduced the graded algebra of observables $\obsG$, and moreover, we shall provide hereafter a combinatorial description of the maps $\kappa_2$ and $\kappa_3$.
\begin{proof}
We start with the covariances and we expand $\cov(S_n(F,\gamma),S_n(G,\gamma))$ by bilinearity:
$$\cov(S_n(F,\gamma),S_n(G,\gamma)) = \sum_{\psi,\phi} \cov(A_{F,\psi},A_{G,\phi}),$$
where the sum runs over pairs of maps $(\psi,\phi)$ from $\lle 1,k\rre$ to $\lle 1,n\rre$.
In this sum, if $\psi(\lle 1,k\rre) \cap \phi(\lle 1,k\rre)=\emptyset$, then the corresponding covariance vanishes, because of the dependency graph identified in Lemma \ref{lem:graphondependencygraph}. Thus, we can restrict the sum to the cases where $|\psi(\lle 1,k\rre) \cap \phi(\lle 1,k\rre) |\geq 1$, which implies that $|\psi(\lle 1,k\rre) \cup \phi(\lle 1,k\rre )|\leq 2k-1 $. If $|\psi(\lle 1,k\rre) \cup \phi(\lle 1,k\rre )|\leq 2k-2 $, then the maps $\psi$ and $\phi$ can be described:
\begin{itemize}
	\item by specifying all the identities $\psi(a)=\phi(b)$ or $\psi(a)=\psi(b)$ or $\phi(a)=\phi(b)$; there is a finite number of configurations of such identities, which depends only on $k$.
	\item and then by specifying less than $2k-2$ values in $\lle 1,n\rre$; there are less than $n^{2k-2}$ possible choices.
\end{itemize}
As a consequence, since $\cov(A_{F,\psi},A_{G,\phi}) \leq 1$ for any maps $\psi$ and $\phi$, the maps $\psi$ and $\phi$ such that
$$|\psi(\lle 1,k\rre) \cap \phi(\lle 1,k\rre) |\geq 1\qquad;\qquad |\psi(\lle 1,k\rre) \cup \phi(\lle 1,k\rre )|\leq 2k-2 $$
will have a contribution to $\cov(S_n(F,\gamma),S_n(G,\gamma)) $ that is a $O(n^{2k-2})$, with a constant in the $O(\cdot)$ that depends only on $k$. This contribution becomes a $O(n^{-1})$ when divided by $k^2\,n^{2k-1}$, so
$$ \frac{\kappa(S_n(F,\gamma),S_n(G,\gamma))}{k^2\,n^{2k-1}} =  \frac{1}{k^2\,n^{2k-1}}\sum_{\substack{|\psi(\lle 1,k\rre) \,\cap\, \phi(\lle 1,k\rre) |= 1 \\ |\psi(\lle 1,k\rre) \,\cup\, \phi(\lle 1,k\rre) |=2k-1}} \cov(A_{F,\psi},A_{G,\phi}) + O(n^{-1}).$$
The two identities $|\psi(\lle 1,k\rre) \cap \phi(\lle 1,k\rre) |= 1 $ and $ |\psi(\lle 1,k\rre) \cup \phi(\lle 1,k\rre) |=2k-1$ imply that $\psi$ and $\phi$ have images of size $k$, so they are injective maps. Therefore, there are exactly $k^2\,n^{\downarrow 2k-1}$ pairs of maps $(\psi,\phi)$ that satisfy these two conditions: 
indeed, to construct such a pair,
\begin{itemize}
 	\item one chooses two indices $a,b\in \lle 1,k\rre$ such that $\psi(a)=\phi(b)$ ($k^2$ possibilities);
 	\item then, there are $2k-1$ distinct values in $\lle 1,n\rre$ to choose, hence $n^{\downarrow 2k-1}$ choices.
 \end{itemize} 
If $a$ and $b$ are two indices in $\lle 1,k\rre$, denote $(F \join G)(a,b)$ the graph on $2k-1$ vertices obtained by identifying in $F \sqcup G$ the vertex $a$ in $V_F$ with the vertex $b$ in $V_G$. This is better understood with an example, see Figure \ref{fig:junctiongraph}.
\begin{center}		
\begin{figure}[ht]
\begin{tikzpicture}[scale=1]
\draw (0,0) -- (0,2) -- (-1.5,1) -- (0,0);
\draw (1,1) node {\Large $\join$};
\draw (4,0) -- (4,2);
\foreach \x in {(0,0),(0,2),(-1.5,1),(2,1),(3.5,0),(3.5,2),(8,1),(6.5,0),(6.5,2),(8,-1),(8,3)}
\fill \x circle (1mm);
\draw (-1.8,1) node {$1$};
\draw (0.25,2.25) node {$2$};
\draw (0.25,-0.25) node {$3$};
\draw (1.7,1) node {$1$};
\draw (3.75,2.25) node {$2$};
\draw (3.75,-0.25) node {$3$};
\draw [NavyBlue,thin] (-2,-0.5) rectangle (0.5,2.5);
\draw [Red,thin] (1.5,-0.5) rectangle (4,2.5);
\draw (4.85,1) node {(2,3)\ \ \large $=$};
\draw (8,1) -- (6.5,0) -- (8,-1) -- (8,3);
\draw [violet,thin] (6,-1.5) rectangle (8.5,3.5);
\end{tikzpicture}
\caption{The junction $(F\join G)(a,b)$ of two graphs along a pair of points $(a\in V_F,b \in V_G)$.\label{fig:junctiongraph}}
\end{figure}
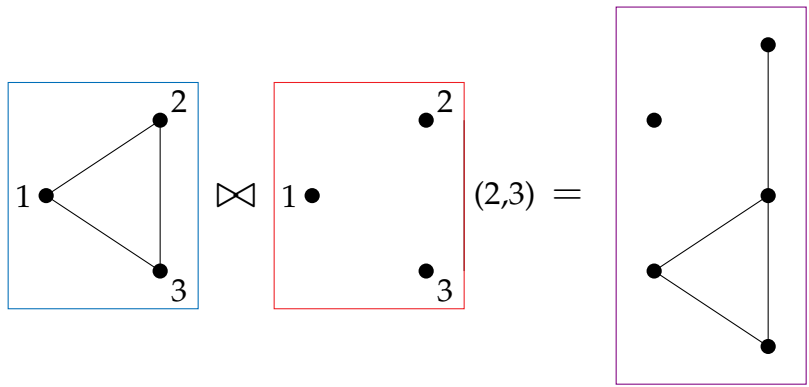
\end{center}
The graph $(F \join G)(a,b)$ is involved in the computation of the covariance of $A_{F,\psi}$ and $A_{G,\phi}$ if $\psi(a)=\phi(b)$ and $|\psi(\lle 1,k\rre)\cup\phi(\lle 1,k\rre)|=2k-1$. Indeed, if $\psi$ is an injective map, then 
\begin{align*}
\esper[A_{F,\psi}] &= \int_{[0,1]^n} \left(\prod_{\{i,j\} \in E_F} g(x_{\psi(i)},x_{\psi(j)})\right) \DD{x_1}\cdots \DD{x_n} \\
&= \int_{[0,1]^k} \left(\prod_{\{i,j\} \in E_F} g(x_i,x_j)\right) \DD{x_1}\cdots \DD{x_k} = t(F,g) = t(F,\gamma).
\end{align*}
When $\psi$ and $\psi$ are two injective maps whose images intersect at exactly one point $\psi(a)=\phi(b)$, a similar calculation yields
\begin{align*}
\esper[A_{F,\psi}\,A_{G,\phi}] &= \int_{[0,1]^n} \left(\prod_{\{i,j\} \in E_F} g(x_{\psi(i)},x_{\psi(j)})\right) \left(\prod_{\{i,j\} \in E_G} g(x_{\phi(i)},x_{\phi(j)}) \right)\DD{x_1}\cdots \DD{x_n} \\
&= \int_{[0,1]^n} \left(\prod_{\{i,j\} \in E_{(F\, \join \,G)(a,b)}} g(x_{\theta(i)},x_{\theta(j)})\right) \DD{x_1}\cdots \DD{x_n}  = t((F \join G)(a,b),\gamma)
\end{align*}
where $\theta$ is the unique injective map $\theta : V_{(F \,\join \,G)(a,b)} \to \lle 1,n\rre$ whose restriction to $V_F$ is $\psi$, and whose restriction to $V_G$ is $\phi$.
\bigskip

Let $F \join G$ be the multiset of all the $k^2$ graphs $(F \join G)(a,b)$, with $a,b \in \lle 1,k\rre$. We can finally compute the limiting covariance:
\begin{align*}
 \frac{\kappa(S_n(F,\gamma),S_n(G,\gamma))}{k^2\,n^{2k-1}} &=  \frac{n^{\downarrow 2k-1}}{k^2\,n^{2k-1}}\,\sum_{a,b=1}^k t((F\join G)(a,b),\gamma) - t(F,\gamma)\,t(G,\gamma)  + O(n^{-1})\\
&=\psi\left( \frac{1}{k^2}\,\sum_{H \in F \,\join\, G} (H -F\times G)\right)(\gamma) + O(n^{-1}).
\end{align*}
This proves the first part of our proposition, with the map $\kappa_2$ given by
$$\kappa_2(F,G) = \frac{1}{k^2}\,\sum_{H \in F \,\join\, G} (H -F\times G)=\frac{1}{k^2}\,\sum_{1\leq a,b\leq k} ((F\join F)(a,b) -F\times G).$$
A similar discussion allows one to compute the asymptotics of the third cumulants. Note first that the joint cumulants have the property that if one can split a family $(X_1,\ldots,X_k)$ into two non-empty families that are independent, then $\kappa(X_1,\ldots,X_k)=0$. Now, in order to compute $\kappa(S_n(F,\gamma),S_n(G,\gamma),S_n(H,\gamma))$, we expand this joint cumulant by multilinearity, and we are reduced to the computation of the third cumulant $\kappa(A_{F,\psi},A_{G,\phi},A_{H,\theta})$, where $\psi,\phi,\theta$ are maps $\lle 1,k\rre \to \lle 1,n\rre$. By using the vanishing property of joint cumulants on independent vectors, one can get rid of all the triples $(\psi,\phi,\theta)$ such that one of the three maps has an image disjoint from the other two.
On the other hand, if the union
$\psi(\lle 1,k\rre) \cup \phi(\lle 1,k\rre)\cup \theta(\lle 1,k\rre)$ has less than $3k-3$ elements, then the corresponding joint cumulant is a $O(1)$, and all these contributions form a $O(n^{3k-3})$. Hence,
$$\kappa(S_n(F,\gamma),S_n(G,\gamma),S_n(H,\gamma)) = \sum_{\psi,\phi,\theta} \kappa(A_{F,\psi},A_{G,\phi},A_{H,\theta})+O(n^{3k-3}),$$
where the constant hidden in the $O(\cdot)$ only depends on $k$, and where the sum is on triples of functions $\psi,\phi,\theta : \lle 1,k\rre \to \lle 1,n\rre$ which are injective, such that 
$$|\psi(\lle 1,k\rre) \cup \phi(\lle 1,k\rre)\cup \theta(\lle 1,k\rre)|=3k-2.$$
This implies that one of the following situations occur:
\begin{enumerate}[label=(\alph*)]
	\item there are indices $a,b,c \in \lle 1,k\rre$ such that $\psi(a)=\phi(b)=\theta(c)$, and there are no other identities between the images of $\psi,\phi,\theta$; 
	\item or, there are indices $a,b \neq c,d \in \lle 1,k\rre$ such that $\psi(a)=\phi(b)$ and $\phi(c)=\theta(d)$, and there are no other identities between the images of $\psi,\phi,\theta$; 
	\item or, one has the same configuration as (b), up to a cyclic permutation of the roles played by $\psi$, $\phi$ and $\theta$.
\end{enumerate}
In the first case, $a,b,c$ being fixed, there are exactly $n^{\downarrow 3k-2}$ possibilities for $\psi,\phi,\theta$, and the joint cumulant of $A_{F,\psi}$, $A_{G,\phi}$ and $A_{H,\theta}$ is in this case
\begin{align*}
&t((F\join G \join H)(a,b,c),\gamma ) - t((F\join G)(a,b)\sqcup H,\gamma) -t((F\join H)(a,c)\sqcup G,\gamma) \\
&- t((G\join H)(b,c)\sqcup F,\gamma)+ 2\,t(F \sqcup G \sqcup H,\gamma),
\end{align*}
where $(F\join G \join H)(a,b,c)$ denotes the graph on $3k-2$ vertices obtained by identifying in $F \sqcup G \sqcup H$ the vertex $a$ in $V_F$, the vertex $b$ in $V_G$ and the vertex $c$ in $V_H$. In the second case, $a,b\neq c,d$ being fixed, there are again $n^{\downarrow 3k-2}$ possibilities for $\psi,\phi,\theta$, and the joint cumulant of $A_{F,\psi}$, $A_{G,\phi}$ and $A_{H,\theta}$ is then
\begin{align*}
&t((F\join G \join H)(a,b;c,d),\gamma ) + t(F \sqcup G \sqcup H,\gamma) \\ 
&- t((F\join G)(a,b)\sqcup H,\gamma) -t((G\join H)(c,d)\sqcup F,\gamma) ,
\end{align*}
where $(F\join G \join H)(a,b;c,d)$ denotes the graph on $3k-2$ vertices obtained by identifying in $F \sqcup G \sqcup H$ the vertex $a$ in $V_F$ with the vertex $b$ in $V_G$, and the vertex $d$ in $V_H$ with the vertex $c$ in $V_G$.  These computations of "elementary" joint cumulants follow from the same argument as for the computation of the elementary covariance $\cov(A_{F,\psi},A_{G,\phi})$. We conclude that
\begin{align*}
&\lim_{n \to \infty} \frac{\kappa(S_n(F,\gamma),S_n(G,\gamma),S_n(H,\gamma))}{n^{3k-2}}  \\
&= \sum_{1\leq a,b,c \leq k} t((F\join G \join H)(a,b,c),\gamma ) + 2\,t(F \sqcup G \sqcup H,\gamma)  - t((F\join G)(a,b)\sqcup H,\gamma) \end{align*}
\vspace*{-0.7cm} $$- \,t((G\join H)(b,c)\sqcup F,\gamma) -t((F\join H)(a,c)\sqcup G,\gamma) $$
\vspace*{-0.3cm}
$$+\sum_{\Z/3\Z}\sum_{1\leq a,b\neq c,d\leq k} t((F\join G \join H)(a,b;c,d),\gamma ) + t(F \sqcup G \sqcup H,\gamma) \qquad\qquad\qquad~$$
\vspace*{-0.6cm}
$$\qquad\qquad- \,t((F\join G)(a,b)\sqcup H,\gamma) -t((G\join H)(c,d)\sqcup F,\gamma)$$
where the symbol $\Z/3\Z$ means that one permutes cyclically the roles played by $F$, $G$ and $H$. Moreover, this limit is attained at speed $O(n^{-1})$. This proves the second part of the proposition, with
\begin{align*}
\kappa_3(F,G,H) &= \frac{1}{k^4}\sum_{1\leq a,b,c \leq k} \left(\substack{(F\join G \join H)(a,b,c)+2\,F \times G \times H  - F \times (G\join H)(b,c)  \\ - G \times (F\join H)(a,c) - H \times (F\join G)(a,b) }\right) \\
&\quad +\frac{1}{k^4}\sum_{\Z/3\Z}\sum_{1\leq a,b\neq c,d\leq k} \left(\substack{(F\join G \join H)(a,b;c,d) + F \times G \times H \\ - H\times (F\join G)(a,b)- F \times (G\join H)(c,d)}\right).\qedhere
\end{align*}
\end{proof}
\bigskip

\begin{example}
Let us illustrate the computation of the maps $\kappa_2$ and $\kappa_3$ on an example. If $F=G=K_3=\tikz[baseline=1.5mm]{
\draw[scale=0.3] (-1.5,1) -- (0,0) -- (0,2) -- (-1.5,1);
\foreach \x in {(-1.5,1),(0,0),(0,2)}
\fill[scale=0.3] \x circle (1mm);}$\, is the triangle, then $H=\tikz[baseline=1.5mm]{
\draw[scale=0.3] (-3,0) -- (-3,2) -- (0,0) -- (0,2) -- (-3,0);
\foreach \x in {(-1.5,1),(0,0),(0,2),(-3,0),(-3,2)}
\fill[scale=0.3] \x circle (1mm);}$\, is the only isomorphism class in $F \join G$, so 
$$\kappa_2(K_3,K_3) = H - K_3\times K_3.$$
For the computation of the limiting third cumulant, the only possible junction of three triangles at a common point is
$$I = \tikz[baseline=0mm]{
	\draw (0:1) -- (180:1) -- (120:1) -- (300:1) -- (240:1) -- (60:1) -- (0:1);
	\foreach \x in {(0:1),(60:1),(120:1),(180:1),(240:1),(300:1),(0,0)}
	\fill \x circle (0.4mm);
}\,,$$
whereas a junction of three triangles $(K_3 \join K_3 \join K_3)(a,b;c,d)$ with $b \neq c$ is always isomorphic to
$$J = \tikz[baseline=0mm]{
	\draw (0:2) -- (180:1) -- (120:1) -- (300:1) -- (0:1);
	\foreach \x in {(0:1),(0:2),(120:1),(180:1),(300:1),(0,0)}
	\fill \x circle (0.4mm);
	\begin{scope}[shift={(1,0)}]
	\draw (0:1) -- (60:1) -- (0,0);
	\fill (60:1) circle (0.4mm);
	\end{scope}
}\,.$$
Therefore,
$$ \kappa_3(K_3,K_3,K_3) = \frac{1}{3} \,I + 2\,J +\frac{8}{3}\,K_3 \times K_3 \times K_3  -5\, K_3 \times H.$$
\end{example}
\medskip

\begin{theorem}[Mod-Gaussian convergence of graphons]\label{thm:modgraphon}
Let $\gamma \in \Gcal$ be a graphon, and $F$ be a finite graph with $k$ vertices. 
\begin{enumerate}
	\item The random variable $S_n(F,\gamma) = n^k\,t(F,\Gamma_n(\gamma))$ satisfies the hypotheses of the method of cumulants with parameters $D_{n}=k^2\,n^{k-1}$, $N_{n} = n^k$ and $A=1$, and with limits $\sigma^2 = \kappa_2(F,F)(\gamma)$ and $L = \kappa_3(F,F,F)(\gamma)$.
	\item If $\kappa_2(F,F)(\gamma) > 0$, then the random variables
$$Y_n(F,\gamma) = \frac{t(F,\Gamma_n(\gamma))-t(F,\gamma)}{\sqrt{\var(t(F,G_n(\gamma)))}}$$
satisfy all the limiting results from Theorem \ref{thm:cumulantestimates}.
\item Besides, we have the uniform concentration inequality 
$$\proba[|t(F,\Gamma_n(\gamma))-\esper[t(F,\Gamma_n(\gamma))]| \geq x] \leq 2\,\exp\left(-\frac{nx^2}{9k^2}\right).$$
\end{enumerate} 
\end{theorem}
Indeed, we have checked the three hypotheses of the method of cumulants. For the concentration inequality, a slightly better result was proven in \cite{BCLSV08}, namely,
$$\proba[|t(F,\Gamma_n(\gamma))-\esper[t(F,\Gamma_n(\gamma))]| \geq x] \leq 2\,\exp\left(-\frac{nx^2}{4k^2}\right).$$
The original proof of this result relied on martingale techniques; here we have shown that up to a constant, one gets similar inequalities from the mod-Gaussian structure. On the other hand, the mod-Gaussian convergence implies the speed of convergence estimate stated in Theorem \ref{thm:speedofconvergencegraphon}, noticing that
$$\lim_{n \to \infty} n\,\var(t(F,\Gamma_n(\gamma))) = k^2\,\kappa_2(F,F)(\gamma).$$
Similarly, we get immediately the moderate deviation estimates from Theorem \ref{thm:moderatedeviationgraphon}, with
$$l(F,\gamma) = \frac{k}{6}\,\frac{\kappa_3(F,F,F)(\gamma)}{(\kappa_2(F,F)(\gamma))^{3/2}}.$$
\begin{example}
Fix a graphon $\gamma$, and suppose that $t(H,\gamma) \neq (t(K_3,\gamma))^2$, where $H$ is a before the "bowtie"
graph. The densities of triangles satisfy the central limit theorem
$$ Y_n= \frac{t(K_3,G_n(\gamma))-t(K_3,\gamma) }{\sqrt{\var(t(K_3,G_n(\gamma)))}} \rightharpoonup \gauss,$$ 
and the left-hand side of this formula can be replaced by the more explicit random variable
$$ \widetilde{Y}_n = \sqrt{n}\, \frac{t(K_3,G_n(\gamma))-t(K_3,\gamma) }{3\,\sqrt{t(H,\gamma) -t(K_3\times K_3,\gamma) }} $$
since \begin{align*}
\esper[t(K_3,G_n(\gamma))] &= t(K_3,\gamma) + O\!\left(\frac{1}{n}\right);\\
\var(t(K_3,G_n(\gamma))) &= \frac{9}{n}(t(H,\gamma)-t(K_3 \times K_3,\gamma)) + O\!\left(\frac{1}{n^2}\right).
\end{align*}
For instance, if $\gamma$ is the graphon of the graph function $g(x,y)=xy$, then one can compute $t(K_3,\gamma)=\frac{1}{27}$, and one has the central limit theorem
$$\widetilde{Y}_n=\sqrt{\frac{5n}{4}}\left(\frac{27\,t(K_3,G_n((x,y)\mapsto xy))-1}{3}\right) \rightharpoonup \gauss.$$
The zone of normality of $(Y_n)_{n \in \N}$ or of $(\widetilde{Y}_n)_{n \in \N}$ is a $o(n^{1/6})$. At the edge of this zone, one enters the regime of moderate deviations, and
\begin{align*}
&\proba[Y_n \geq n^{1/6}\,x] = \proba\!\left[\widetilde{Y}_n \geq n^{1/6}\,x\right]\,(1+o(1)) \\
&=\frac{\E^{-\frac{n^{1/3}\,x^2}{2}}}{n^{1/6}\,x\,\sqrt{2\pi}}\,\exp\left(\frac{t(I,\gamma)+6\,t(J,\gamma)+8\,(t(K_3,\gamma))^3 - 15\,t(H\times K_3,\gamma) }{6\,(t(H,\gamma)-t(K_3\times K_3,\gamma))^{3/2}}\,x^3\right)\,(1+o(1))
\end{align*}
for $x > 0$ fixed. Last, for $n$ large enough,
$$\dkol(Y_n,\gauss) \leq \frac{230}{((t(H,\gamma)-t(K_3\times K_3,\gamma))^{3/2}}\,\frac{1}{\sqrt{n}}.$$
As far as we know, all these results for the fluctuations of densities of subgraphs in graphon models are new.
\end{example}
\medskip

\subsection{Fluctuations of random permutations and permutons}\label{subsec:permmodgauss}
We have the same kind of results for the fluctuations of the models $(\sigma_n(\pi))_{n \in \N}$ of random permutations, and the main difference between this theory and the theory for graphon models is that one can consider the observables $t(\tau,\sigma_n(\pi))$ or the observables $t(\tau,\Pi_n(\pi))$. The difference between these observables is bounded by $\frac{|\tau|^2}{2n}$, so a central limit theorem for one of these quantities will imply a central limit theorem for the other quantity. Unfortunately, if one wants to get more precise results (\emph{e.g.}~a speed of convergence estimate), then one needs to be more careful and to identify a dependency structure for each random variable $t(\tau,\sigma_n(\pi))$ or $t(\tau,\Pi_n(\pi))$. The first case is almost identical to the case of graphs, whereas the second case requires a bit more work (see Lemma \ref{lem:expansionthat}).\medskip

In the sequel, we fix a permuton $\pi \in \Scal$ and a permutation $\tau$ of size $k$. We set $S_n(\tau,\pi) = \binom{n}{k}\,t(\tau,\sigma_n(\pi))$ and $\widetilde{S}_n(\tau,\pi) = n^k\,t(\tau,\Pi_n(\pi))$; beware of the slightly different scalings involved in these formulas. We shall expand the variables $S_n(\tau,\pi)$ and $\widetilde{S}_n(\tau,\pi)$ as sums of bounded random variables with sparse dependency graphs. For $S_n(\tau,\pi)$, if we introduce independent random points $(X_1,Y_1),\ldots,(X_n,Y_n)$ with law $\pi$ and such that $\sigma_n(\pi)=\conf((X_1,Y_1),\ldots,(X_n,Y_n))$, then we have:
$$S_n(\tau,\pi) = \mathrm{occ}(\tau,\sigma_n(\pi)) = \sum_{\substack{L \subset \lle 1,n\rre\\|L|=k}} A_{\tau,L},$$
with 
$$ A_{\tau,L} = \begin{cases}
	1 &\text{if }\conf(\{(X_l,Y_l),\,\,l\in L\})=\tau,\\
	0 &\text{otherwise}.
\end{cases}$$
\begin{lemma}\label{lem:permutondependencygraph}
The graph $G_n$
\begin{itemize}
	\item with vertex set $V_n=\{L \subset \lle 1,n\rre,\,\,\,|L|=k\}$,
	\item and with an edge between $L$ and $M$ if $L \cap M \neq \emptyset$
\end{itemize}
is a dependency graph for the family of random variables $(A_{\tau,L})_{L \in V_n}$ involved in the expansion of $S_n(\tau,\pi)$.
\end{lemma} 

\begin{proof}
Denote $Z_i = (X_i,Y_i)$. Given two families $(A_{\tau,L})_{L \in W_1}$ and $(A_{\tau,M})_{M \in W_2}$ that are not connected in this graph, if $U_1 = \bigcup_{L \in W_1}L$ and $U_2 = \bigcup_{M \in W_2}M$, then we have
$$\left(\bigcup_{L \in W_1}L\right) \cap \left(\bigcup_{M \in W_2} M\right) = U_1 \cap U_2 = \emptyset.$$
Then, $(A_{\tau,L})_{L \in W_1}$ is measurable with respect $\sigma(\{Z_i,\,i \in U_1\})$, and $(A_{\tau,M})_{M \in W_2}$ is measurable with respect to $\sigma(\{Z_i,\,i \in U_2\})$. As the random points $Z_i$ are independent, we conclude that $(A_{\tau,L})_{L \in W_1}$ and $(A_{\tau,M})_{M \in W_2}$ are independent families
\end{proof}
\medskip

The parameters of the dependency graph $G_n$ in Lemma \ref{lem:permutondependencygraph} are $D_n=D_{n,k} = k\binom{n-1}{k-1}$, $N_n = N_{n,k} = \binom{n}{k}$ and $A=1$. The existence of this dependency graph implies the hypothesis \eqref{eq:cumulant1} for $S_n(\tau,\pi)$. For $\widetilde{S}_n(\tau,\pi)$, we have a similar decomposition
$$\widetilde{S}_n(\tau,\pi) = \sum_{1\leq l_1,\ldots,l_k\leq n} \widetilde{A}_{\tau,(l_1,\ldots,l_k)},$$
but the definition of the elementary variables $\widetilde{A}_{\tau,(l_1,\ldots,l_k)}$ is more complicated than before. It relies on the following:

\begin{lemma}\label{lem:expansionthat}
Let $\tau$ be a permutation of size $k\leq n$. Given a family of points $(z_1,\ldots,z_n)$ in $[0,1]^2$ in a general configuration, we set  
$$\widehat{t}(\tau,(z_1,\ldots,z_n)) = t(\tau,\pi(\conf(z_1,\ldots,z_n))).$$
There exists a measurable function $F_\tau : ([0,1]^2)^k \to [0,1]$ such that, for any sequence $(z_1,\ldots,z_n)$ in a general configuration, 
$$\widehat{t}(\tau,(z_1,\ldots,z_n)) = \sum_{1\leq l_1,\ldots,l_k \leq n} \frac{1}{n^k}\,F_\tau(z_{l_1},\ldots,z_{l_k}).$$
\end{lemma}

\begin{remark}[Values of $F_\tau$ on general configurations]
It will follow from the proof of the lemma that when $(t_1,\ldots,t_k)$ is a family in general configuration, one has simply $F_\tau(t_1,\ldots,t_k)=1_{\mathrm{conf}(t_1,\ldots,t_k)=\tau}$. However, in the expansion above, the families $(z_{l_1},\ldots,z_{l_k})$ are usually not in general configuration, since they might contain points $z_l$ with multiplicity larger than $2$ (for instance if $l_1=l_2$). Thus, our lemma extends in a measurable way the domain of the function $(t_1,\ldots,t_k) \mapsto 1_{\mathrm{conf}(t_1,\ldots,t_k)=\tau}$ to any $k$-tuple of points.
\end{remark}

\begin{proof}
By definition, if $\sigma = \conf(z_1,\ldots,z_n)$ and $W_1,\ldots,W_k$ are independent points with law $\pi(\sigma)$, then
$$\widehat{t}(\tau,(z_1,\ldots,z_n)) = \proba[\conf(W_1,\ldots,W_k)=\tau].$$ 
We construct the random variables $W_1,\ldots,W_k$ as follows. If $z_l=(x_l,y_l)$, we denote $\psi_1 : \{x_1,\ldots,x_n\} \to \lle 1,n\rre$ and $\psi_2: \{y_1,\ldots,y_n\} \to \lle 1,n\rre$ the two increasing bijections; then, the configuration $\sigma$ is defined by the identity $\sigma(\psi_1(x_l))=\psi_2(y_l)$. Let us introduce independent random variables $(S_{i})_{i \in \lle 1,k\rre}$ uniformly distributed over the square $[0,1]^2$, and also independent discrete random variables $L_1,\ldots,L_k$ uniformly distributed in $\lle 1,n\rre$. We then set:
$$W_i = \left(\frac{\psi_1(x_{L_i})}{n},\frac{\psi_2(y_{L_i})}{n}\right)-\frac{S_{i}}{n}.$$
Each $W_i$ has law $\pi(\sigma)$, and on the other hand, since the $L_i$'s and the $S_{i}$'s are all independent, the random variables $W_1,\ldots,W_k$ are independent. Now, let us construct other random variables $V_1,\ldots,V_k$ such that
$$\conf(V_1,\ldots,V_k)=\conf(W_1,\ldots,W_k)\quad\text{almost surely}.$$
The idea is that instead of choosing random points in squares of size $\frac{1}{n}$ attached to the points $(\frac{\psi_1(x_l)}{n},\frac{\psi_2(y_l)}{n})$, one can choose random points in small squares of size $\eps$ attached to the points $z_l$, and this without changing the configuration.  Since $(z_1,\ldots,z_n)$ is in general configuration, we can find an $\eps>0$ such that $|x_l-x_m|>\eps$ if $l\neq m$, and $|y_l-y_m|>\eps$ if $l \neq m$. We then set:
$$V_i = z_{L_i} - \eps \,S_{i}.$$
It is easy to convince oneself on a diagram that the points $V_1,\ldots,V_k$ can be obtained from the points $W_1,\ldots,W_k$ by a bijection that is increasing in both coordinates; see Figure \ref{fig:sameconfig}. Therefore, in particular, the configuration is unchanged.
\begin{center}		
\begin{figure}[ht]
\begin{tikzpicture}[scale=0.8]
\foreach \x in {(1,2),(2,4),(3,5),(4,3),(5,6),(6,1)}
\fill [shift={\x},black!15!white] (0,0) rectangle (-1,-1);
\draw [<->] (0,7) -- (0,0) -- (7,0);
\foreach \x in {0,1,2,3,4,5,6}
\draw [dotted] (\x,0) -- (\x,6);
\foreach \x in {0,1,2,3,4,5,6}
\draw [dotted] (0,\x) -- (6,\x);
\foreach \x in {1,2,3,4,5,6}
{\draw (\x,-0.1) -- (\x,0.1) ; \draw (-0.1,\x) -- (0.1,\x) ;};
\foreach \x in {(0.5,1.7),(2.2,4.8),(2.7,4.1),(5.5,0.6)}
\fill [Red] \x circle (3pt);
\draw [Red] (2.2,5.1) node {$W_1$};
\draw [Red] (1,1.7) node {$W_4$};
\draw [Red] (2.7,3.7) node {$W_3$};
\draw [Red] (5.5,0.95) node {$W_2$};
\begin{scope}[shift={(9,0)}]
\draw [<->] (0,7) -- (0,0) -- (7,0);
\foreach \x in {(1.3,2),(2.1,4.7),(2.8,5.4),(4.2,2.9),(5.1,6),(5.9,0.7)}
{\fill [shift={\x},black!15!white] (0,0) rectangle (-0.6,-0.6); \draw \x node {$\times$};}
\draw [dotted] (0,6) -- (6,6) -- (6,0);
\foreach \x in {(1,1.82),(2.32,5.28),(2.62,4.86),(5.6,0.42)}
\fill [NavyBlue] \x circle (3pt);
\draw [NavyBlue] (2.32,5.65) node {$V_1$};
\draw [NavyBlue] (0.55,1.82) node {$V_4$};
\draw [NavyBlue] (2.62,4.45) node {$V_3$};
\draw [NavyBlue] (5.5,0.85) node {$V_2$};
\end{scope}
\end{tikzpicture}
\caption{The random points $W_1,\ldots,W_k$ and $V_1,\ldots,V_k$ have the same configuration.\label{fig:sameconfig}}
\end{figure}
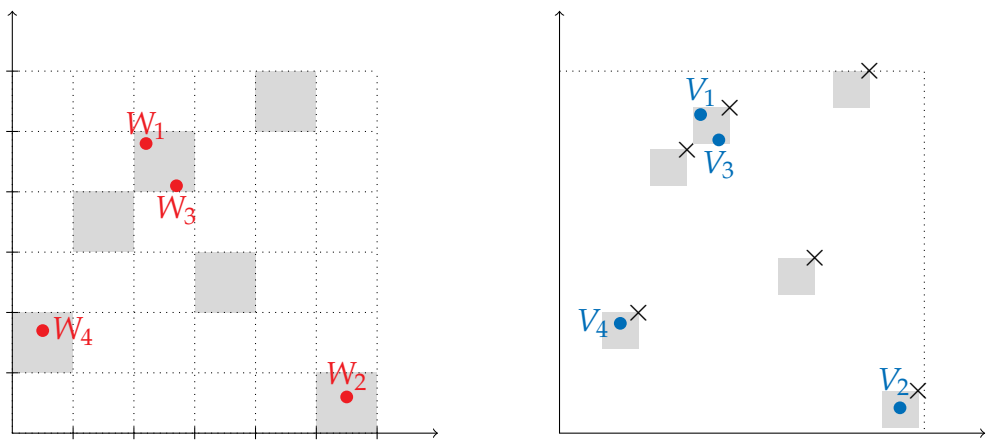
\end{center}
We now have:
\begin{align*}
\widehat{t}(\tau,(z_1,\ldots,z_n)) &= \proba[\conf(W_1,\ldots,W_k)=\tau] = \proba[\conf(V_1,\ldots,V_k)=\tau] \\
&=\sum_{1\leq l_1,\ldots,l_k \leq n}\frac{1}{n^k} \,\,\proba[\conf(z_{l_1}-\eps\,S_{1},\ldots,z_{l_k}-\eps \,S_{k}) =\tau]
\end{align*}
Given points $t_1,\ldots,t_k$ in the square $[0,1]^2$, we define 
$$F_\tau(t_1,\ldots,t_k) = \lim_{\eps \to 0} \left(\proba[\conf(t_1-\eps\,S_{1},\ldots,t_k-\eps \,S_{k}) =\tau]\right),$$
where $(S_1,\ldots,S_k)$ is a set of independent uniform random variables on $[0,1]^2$. This defines a measurable function of $(t_1,\ldots,t_k) \in ([0,1]^2)^k$, and the limit when $\eps$ goes to zero is stationnary for any $k$-tuple of points (this is the same argument of "rescaling of small squares" as before). Since $\widehat{t}(\tau,(z_1,\ldots,z_n)) = \sum_{1\leq l_1,\ldots,l_k \leq n}\frac{1}{n^k} \,\proba[\conf(z_{l_1}-\eps\,S_{1},\ldots,z_{l_k}-\eps \,S_{k}) =\tau]$ for $\eps$ small enough, by taking the limit, one obtains the identity claimed.
\end{proof}
\medskip

Given a family of independent random points $Z_1,\ldots,Z_n$ with law $\pi$, we can now write:
\begin{align*}
\widetilde{S}_n(\tau,\pi) &= n^k\,t(\tau,\pi(\conf(Z_1,\ldots,Z_n))) = n^k\,\widehat{t}(\tau,(Z_1,\ldots,Z_n)) \\
&= \sum_{1\leq l_1,\ldots,l_k\leq n} F_\tau(Z_{l_1},\ldots,Z_{l_n}),
\end{align*}
so if we set $\widetilde{A}_{\tau,(l_1,\ldots,l_k)} = F_\tau(Z_{l_1},\ldots,Z_{l_n})$, then we have an expansion of $\widetilde{S}_n(\tau,\pi)$ with the following property:
\begin{lemma}\label{lem:permutondependencygraph2}
The graph $\widetilde{G}_n$
\begin{itemize}
	\item with vertex set $\widetilde{V}_n=(\lle 1,n\rre)^k$,
	\item and with an edge between $(l_1,\ldots,l_k)$ and $(m_1,\ldots,m_k)$ if $l_a=m_b$ for some indices $a$ and $b$
\end{itemize}
is a dependency graph for the family of random variables $(\widetilde{A}_{\tau,(l_1,\ldots,l_k)})_{(l_1,\ldots,l_k) \in \widetilde{V}_n}$ involved in the expansion of $\widetilde{S}_n(\tau,\pi)$.
\end{lemma} 
\noindent One sees at once that the parameters of the dependency graph $\widetilde{G}_n$ in Lemma \ref{lem:permutondependencygraph2} can be taken equal to $\widetilde{D}_{n,k}=k^2\,n^{k-1}$, $\widetilde{N}_{n,k}=n^k$ and $A=1$. Thus, the sequence of random variables $(\widetilde{S}_n(\tau,\pi))_{n\in \N}$ satisfies the hypothesis \eqref{eq:cumulant1} with respect to these parameters.
\bigskip

We now turn to the computation of the limiting variances and third cumulants. It turns out that one obtains the same result in both cases:
\begin{proposition}[Limiting first cumulants of pattern occurrences]
Denote $\obs_{\sym,k}$ the vector space spanned by the permutations of size $k$ in $\obsS$. There exists two linear maps 
\begin{align*}
\kappa_2 &: \obs_{\sym,k} \otimes \obs_{\sym,k} \to \obsS ;\\
\kappa_3 &: \obs_{\sym,k} \otimes \obs_{\sym,k} \otimes \obs_{\sym,k} \to \obsS
\end{align*}
such that, for any permutations $\tau$, $\rho$ and $\mu$ in $\sym(k)$, and for any permuton $\pi$,
\begin{align*}
\lim_{n \to \infty}\frac{\kappa(S_n(\tau,\pi),S_n(\rho,\pi))}{k\binom{n}{k}\binom{n-1}{k-1}} &= \lim_{n \to \infty}\frac{\kappa(\widetilde{S}_n(\tau,\pi),\widetilde{S}_n(\rho,\pi))}{k^2\,n^{2k-1}} =  \kappa_2(\tau,\rho)(\pi);\\
\lim_{n \to \infty} \frac{\kappa(S_n(\tau,\pi),S_n(\rho,\pi),S_n(\mu,\pi))}{k^2\,\binom{n}{k}\binom{n-1}{k-1}^{\!2}} &= \lim_{n \to \infty} \frac{\kappa(\widetilde{S}_n(\tau,\pi),\widetilde{S}_n(\rho,\pi),\widetilde{S}_n(\mu,\pi))}{k^4\,n^{3k-2}} = \kappa_3(\tau,\rho,\mu)(\pi).
\end{align*}
Moreover, in each case, the limit is attained at speed $O(n^{-1})$, with a constant in the $O(\cdot)$ that depends only on $k$.
\end{proposition}

\begin{proof}
As before, $(Z_n)_{n \in \N}$ will be a sequence of independent random points in $[0,1]^2$ with law $\pi$. As in the case of graphs, we can expand the covariance $\cov(S_n(\tau,\pi),S_n(\rho,\pi))$ by bilinearity:
$$\cov(S_n(\tau,\pi),S_n(\rho,\pi))=\sum_{|L|=|M|=k} \cov(A_{\tau,L},A_{\rho,M}),$$
the sum running over pairs of subsets with size $k$ in $\lle 1,n\rre$. Each variable $A_{\tau,L}$ is equal to $1_{\conf(\{Z_l,\,\,l \in L\})=\tau}$; and similarly for the variables $A_{\rho,M}$. We can restrict the sum to pairs of subsets that intersect, since otherwise the covariance vanishes. On the other hand, the pairs $(L,M)$ where $|L\cap M| \geq 2$ yield a total contribution which is a $O(n^{2k-2})$, and which becomes a $O(n^{-1})$ when divided by $k\binom{n}{k}\binom{n-1}{k-1}$. Hence, we can restrict our attention to pairs $(L,M)$ where $|L \cap M| = 1$:
$$\frac{\kappa(S_n(\tau,\pi),S_n(\rho,\pi))}{k\binom{n}{k}\binom{n-1}{k-1}}  = \frac{1}{k\binom{n}{k}\binom{n-1}{k-1}} \sum_{\substack{|L|=|M|=k \\ |L \cap M|=1}} \cov(A_{\tau,L},A_{\rho,M})  + O(n^{-1}).$$
The remaining covariances can be computed by using the \emph{amalgamated graphical shuffle products} $(\tau \join \rho)(a,b)$ of the permutations $\tau$ and $\rho$. Fix two integers $a,b \in \lle 1,k\rre$. The set $(\tau \boxtimes \rho)(a,b)$ is the set of permutations $\sigma$ of size $2k-1$, such that there exists a partition of $\lle 1,2k-1\rre$ in three parts
\begin{align*}
I_- &= \{i_1<i_2<\cdots <i_k\} \setminus\{i_a\};\\
J_- &=\{j_1<j_2<\cdots <j_k\}\setminus\{j_b\};\\
K &=\{i_a=j_b\}
\end{align*}
with $I = I_- \sqcup K$ that makes appear $\tau$ as a pattern of $\sigma$, and $J = J_- \sqcup K$ that makes appear $\rho$ as a pattern of $\sigma$. Notice that there can exist several partitions $\lle 1,2k-1\rre = I_- \sqcup J_- \sqcup K$ with this property for $\sigma$; in this case, we count $\sigma$ several times. In other words, $(\tau \boxtimes \rho)(a,b)$ is defined as a multiset. For instance, if $k=3$, $\tau=312$, $\rho=132$, $a=2$ and $b=3$, then there are $9$ permutations in $(\tau \boxtimes \rho)(a,b)$, given by the diagrams of Figure \ref{fig:amalgamatedproduct}.

\begin{center}		
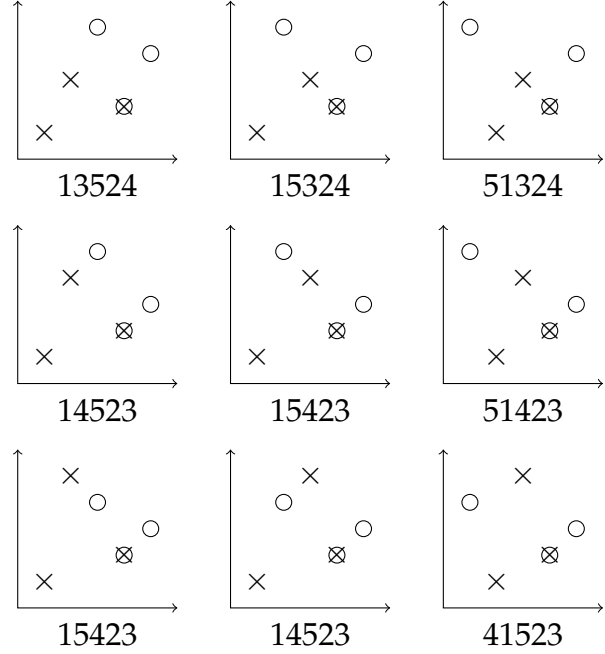
\begin{figure}[ht]
\begin{tikzpicture}[scale=0.35]
\draw [<->] (6,0) -- (0,0) -- (0,6);
\foreach \x in {(3,5),(5,4),(4,2)}
\draw \x circle (3mm);
\foreach \x in {(1,1),(2,3),(4,2)}
\draw \x node {$\times$};
\draw (3,-1) node {$13524$};
\begin{scope}[shift={(8,0)}]
\draw [<->] (6,0) -- (0,0) -- (0,6);
\foreach \x in {(2,5),(5,4),(4,2)}
\draw \x circle (3mm);
\foreach \x in {(1,1),(3,3),(4,2)}
\draw \x node {$\times$};
\draw (3,-1) node {$15324$};
\end{scope}
\begin{scope}[shift={(16,0)}]
\draw [<->] (6,0) -- (0,0) -- (0,6);
\foreach \x in {(1,5),(5,4),(4,2)}
\draw \x circle (3mm);
\foreach \x in {(2,1),(3,3),(4,2)}
\draw \x node {$\times$};
\draw (3,-1) node {$51324$};
\end{scope}
\begin{scope}[shift={(0,-8.5)}]
\draw [<->] (6,0) -- (0,0) -- (0,6);
\foreach \x in {(3,5),(5,3),(4,2)}
\draw \x circle (3mm);
\foreach \x in {(1,1),(2,4),(4,2)}
\draw \x node {$\times$};
\draw (3,-1) node {$14523$};
\end{scope}
\begin{scope}[shift={(8,-8.5)}]
\draw [<->] (6,0) -- (0,0) -- (0,6);
\foreach \x in {(2,5),(5,3),(4,2)}
\draw \x circle (3mm);
\foreach \x in {(1,1),(3,4),(4,2)}
\draw \x node {$\times$};
\draw (3,-1) node {$15423$};
\end{scope}
\begin{scope}[shift={(16,-8.5)}]
\draw [<->] (6,0) -- (0,0) -- (0,6);
\foreach \x in {(1,5),(5,3),(4,2)}
\draw \x circle (3mm);
\foreach \x in {(2,1),(3,4),(4,2)}
\draw \x node {$\times$};
\draw (3,-1) node {$51423$};
\end{scope}
\begin{scope}[shift={(0,-17)}]
\draw [<->] (6,0) -- (0,0) -- (0,6);
\foreach \x in {(3,4),(5,3),(4,2)}
\draw \x circle (3mm);
\foreach \x in {(1,1),(2,5),(4,2)}
\draw \x node {$\times$};
\draw (3,-1) node {$15423$};
\end{scope}
\begin{scope}[shift={(8,-17)}]
\draw [<->] (6,0) -- (0,0) -- (0,6);
\foreach \x in {(2,4),(5,3),(4,2)}
\draw \x circle (3mm);
\foreach \x in {(1,1),(3,5),(4,2)}
\draw \x node {$\times$};
\draw (3,-1) node {$14523$};
\end{scope}
\begin{scope}[shift={(16,-17)}]
\draw [<->] (6,0) -- (0,0) -- (0,6);
\foreach \x in {(1,4),(5,3),(4,2)}
\draw \x circle (3mm);
\foreach \x in {(2,1),(3,5),(4,2)}
\draw \x node {$\times$};
\draw (3,-1) node {$41523$};
\end{scope}
\end{tikzpicture}
\caption{The permutations involved in the amalgamated graphical shuffle product of $\tau=312$ and $\rho=132$ at $(a,b)=(2,3)$.}\label{fig:amalgamatedproduct}
\end{figure}
\end{center}
We now define the amalgamated graphical shuffle product of $\tau$ and $\rho$ at the points $a$ and $b$ by the following operation in the algebra of permutations $\obs$:
$$(\tau \join \rho)(a,b) = \frac{(k!)^2}{(2k-1)!} \sum_{\sigma \in (\rho \boxtimes \tau)(a,b)} \sigma.$$
The combinatorial factor $\frac{(k!)^2}{(2k-1)!}$ is akin to the combinatorial factor $\frac{k_1!\, k_2!}{k!}$ involved in the usual graphical shuffle product. \bigskip

Let $L$ and $M$ be two subsets of $\lle 1,n\rre$ which meet at one point and which have cardinality $k$. We have
\begin{align*}
\cov(A_{\tau,L},A_{\rho,M}) &= \proba[\conf(\{Z_l,\,\,l\in L\})=\tau \text{ and }\conf(\{Z_m,\,\,m\in M\})=\rho] \\
&\quad- \proba[\conf(\{Z_l,\,\,l\in L\})=\tau]\,\,\proba[\conf(\{Z_m,\,\,m\in M\})=\rho],
\end{align*}
and since the $Z_n$'s are independent and with same law, this quantity does not depend on the subsets $L$ and $M$. It is for instance equal to 
$$\proba[\conf(Z_1,Z_2,\ldots,Z_k)=\tau \text{ and }\conf(Z_k,Z_{k+1},\ldots,Z_{2k-1})=\rho] - t(\tau,\pi)\,t(\rho,\pi).$$
Let us then analyse the event $E=\{\conf(Z_{1},\ldots,Z_{k})=\tau \text{ and }\conf(Z_{k},\ldots,Z_{2k-1})=\rho\}$.
If one writes the multiset $(\tau\boxtimes \rho)(a,b)$ as a set of pairs $(\sigma,I_-\sqcup J_- \sqcup K)$, and if $\psi$ is the unique increasing bijection $\{X_1,\ldots,X_{2k-1}\} \to \lle 1,2k-1\rre$, then
$$E = \bigsqcup_{\substack{1\leq a,b\leq k \\ (\sigma,I_-\sqcup J_- \sqcup K) \in (\tau\boxtimes \rho)(a,b)}} \left\{\substack{\conf(Z_1,\ldots,Z_{2k-1})=\sigma \\ \psi(\{X_1,\ldots,X_{k-1}\}\sqcup \{X_k\} \sqcup \{X_{k+1},\ldots,X_{2k-1}\})=I_-\sqcup K \sqcup J_-}\right\}.$$
However, for any sequence $Z_1,\ldots,Z_{2k-1}$ of independent points with law $\pi$, 
the relabeling map $\psi$ is independent from the configuration $\sigma$ of the points $Z_1,\ldots,Z_{2k-1}$, therefore,
\begin{align*}
\proba[E]&= \sum_{\substack{1\leq a,b\leq k \\ (\sigma,I_-\sqcup J_- \sqcup K)}}  t(\sigma,\pi)\,\proba[\psi(\{X_1,\ldots,X_{k-1}\}\sqcup \{X_k\} \sqcup \{X_{k+1},\ldots,X_{2k-1}\}) = I_-\sqcup K \sqcup J_-] \\
&= \frac{((k-1)!)^2}{(2k-1)!}\sum_{\substack{1\leq a,b\leq k \\ (\sigma, I_-\sqcup J_- \sqcup K) }} t(\sigma,\pi) = \frac{1}{k^2} \sum_{1\leq a,b\leq k} ((\tau\join \rho)(a,b))(\pi).
\end{align*}
We conclude that
$$\frac{\kappa(S_n(\tau,\pi),S_n(\rho,\pi))}{k\binom{n}{k}\binom{n-1}{k-1}}  = \frac{1}{k^2} \sum_{1\leq a,b\leq k} (((\tau\join \rho)(a,b))(\pi) - t(\tau,\pi)\,t(\rho,\pi))+ O(n^{-1}).$$
This proves the existence of a map $\kappa_2$, as in the Lemma, where $\kappa_2$ is defined by
$$\kappa_2(\tau,\rho) = \frac{1}{k^2} \sum_{1\leq a,b\leq b} ((\tau \join \rho)(a,b) - \tau \times \rho).$$
For the estimation of the covariances of the quantities $\widetilde{S}_n$, we can write:
\begin{align*}
\cov(\widetilde{S}_n(\tau,\pi),\widetilde{S}_n(\rho,\pi)) &= \sum_{\substack{(l_1,\ldots,l_k) \in \lle 1,n\rre^k \\ (m_1,\ldots,m_k) \in \lle 1,n\rre^k}} \cov(\widetilde{A}_{\tau,(l_1,\ldots,l_k)},\widetilde{A}_{\rho,(m_1,\ldots,m_k)}) \\
&= \sum_{1\leq a,b\leq k^2}\sum_{\substack{(l_1\neq \cdots \neq l_k) \in \lle 1,n\rre^k \\ (m_1 \neq \cdots \neq m_k) \in \lle 1,n\rre^k \\ l_a=m_b}} \cov(\widetilde{A}_{\tau,(l_1,\ldots,l_k)},\widetilde{A}_{\rho,(m_1,\ldots,m_k)}) +O(n^{2k-2})
\end{align*}
where on the second line the sum is over pairs of $k$-arrangements with only one equality of indices $l_a=m_b$. However, by a previous remark, if $(l_1\neq \cdots \neq l_k)$ is an arrangement, then
$$\widetilde{A}_{\tau,(l_1,\ldots,l_k)} = F_{\tau}(Z_{l_1},\ldots,Z_{l_k}) = 1_{\mathrm{conf}(Z_{l_1},\ldots,Z_{l_k})=\tau} = A_{\tau,\{l_1,\ldots,l_k\}},$$
because in this case $(Z_{l_1},\ldots,Z_{l_k})$ is in a general configuration with probability $1$. This implies immediately that $\frac{\cov(\widetilde{S}_n(\tau,\pi),\widetilde{S}_n(\rho,\pi))}{k^2\,n^{2k-1}}$ has the same asymptotics as $\frac{\kappa(S_n(\tau,\pi),S_n(\rho,\pi))}{k\binom{n}{k}\binom{n-1}{k-1}}$.\bigskip

The calculation of the asymptotics of the third cumulants follows the same lines. In order to evaluate
$$ \frac{1}{k^2\,\binom{n}{k}\,\binom{n}{k-1}^2}\, \kappa\!\left(\widetilde{S}_n(\tau,\pi),\widetilde{S}_n(\rho,\pi),\widetilde{S}_n(\mu,\pi)\right) ,$$ we need to compute $\kappa(A_{\tau,L},A_{\rho,M},A_{\mu,N})$
when $L=\{l_1<l_2<\cdots<l_k\}$, $M=\{m_1<\cdots<m_k\}$ and $N=\{n_1<\cdots<n_k\}$ are three $k$-subsets of $\lle 1,n\rre$ that meet one of the following conditions:
\begin{enumerate}[label=(\alph*)]
	\item either $L \cap M = M \cap N = L \cap N = \{l_a=m_b=n_c\}$ and there are no other equality of indices;
	\item or, $L \cap M = \{l_a=m_b\}$ and $M\cap N = \{m_c=n_d\}$ with $b \neq c$, and there are no other equality of indices;
	\item or, one has the same configuration as (b), up to a cyclic permutation of the roles played by $L,M,N$.
\end{enumerate}
In the first case, the cumulant $\kappa(A_{\tau,L},A_{\rho,M},A_{\mu,N})$ does not depend on $L,M,N$, and it is equal to
\begin{align*}
&\proba[\conf(Z_0,Z_1,\ldots,Z_{k-1}) = \tau,\,\,\conf(Z_0,Z_k,\ldots,Z_{2k-2}) = \rho,\,\,
\conf(Z_0,Z_{2k-1},\ldots,Z_{3k-3}) = \mu]\\
&-\frac{1}{k^3}\,\sum_{1\leq a,b,c \leq k} \left((\tau \join \rho)(a,b)(\pi)\,\mu(\pi) + (\rho\join \mu)(b,c)(\pi)\,\tau(\pi) + (\tau \join \mu)(a,c)(\pi)\,\rho(\pi)\right)\\
&+ 2\,\tau(\pi)\,\rho(\pi)\,\mu(\pi).
\end{align*}
In this formula, we have written $\sigma(\pi)$ for the evaluation of a pattern density, instead of $t(\sigma,\pi)$. To compute the first term $\proba[E]$, we introduce the multiset $(\tau \boxtimes \rho \boxtimes \mu)(a,b,c)$, which is the set of permutations $\sigma \in \sym(3k-2)$ such that there exists a partition of $\lle 1,3k-2\rre$ in four parts
\begin{align*}
H_- &= \{h_1<h_2<\cdots <h_k\}\setminus\{h_a\};\\
I_- &=\{i_1<i_2<\cdots <i_k\}\setminus\{i_b\};\\
J_- &=\{j_1<j_2<\cdots <j_k\}\setminus\{j_c\};\\
K &=\{h_a=i_b=j_c\}
\end{align*}
with $H=H_- \sqcup K$ that makes appear $\tau$ as a pattern of $\sigma$, and similarly for $(I_-\sqcup K,\rho)$ and $(J_-\sqcup K,\mu)$. Then, if
$$(\tau \join \rho \join \mu)(a,b,c) = \frac{(k!)^3}{(3k-2)!}\sum_{\sigma \in (\tau \boxtimes \rho \boxtimes \mu)(a,b,c)} \sigma, $$
the same arguments as for the computation of the limiting covariance yields
$$\proba[E] = \frac{1}{k^3}\sum_{1\leq a,b,c\leq k} ((\tau \join \rho \join \mu)(a,b,c))(\pi).$$
On the other hand, there are $$\binom{n}{3k-2}\,\frac{(3k-2)!}{((k-1)!)^3}$$ parts $A$, $B$, $C$ that meet the condition (a). Hence, the contribution to the limit of $\widetilde{d}_n$ of the parts $A,B,C$ satisfying (a) is given by evaluating on the permuton $\pi$ the observable
$$\frac{1}{k^4} \sum_{1\leq a,b,c\leq k} \left(\substack{(\tau \join \rho \join \mu)(a,b,c) + 2\,\tau\times\rho\times\mu - (\tau \join \rho)(a,b)\times\mu \\ - (\rho\join \mu)(b,c)\times \tau - (\tau \join \mu)(a,c)\times \rho}\right) .$$
To treat the other cases (b) or (c), we introduce the multiset $(\tau \boxtimes \rho \boxtimes \mu)(a,b;c,d)$, which is the set of permutations $\sigma \in \sym(3k-2)$  such that there exists a partition of $\lle 1,3k-2\rre$ in five parts
\begin{align*}
H_- &= \{h_1<h_2<\cdots <h_k\}\setminus \{h_a\};\\
I_- &=\{i_1<i_2<\cdots <i_k\}\setminus \{i_b,i_c\};\\
J_- &=\{j_1<j_2<\cdots <j_k\} \setminus \{j_d\};\\
K &=\{h_a=i_b\};\\
L &= \{i_c=j_d\}
\end{align*}
with $H=H_- \sqcup K$ that makes appear $\tau$ as a pattern of $\sigma$, and similarly for the two pairs $(I_-\sqcup K \sqcup L,\rho)$ and $(J_-\sqcup L,\mu)$. As usual, we count a permutation $\sigma$ several times if several partitions satisfy these conditions. If we define 
$$(\tau \join \rho \join \mu)(a,b;c,d) = \frac{(k!)^3}{(3k-2)!}\sum_{\sigma \in (\tau \boxtimes \rho \boxtimes \mu)(a,b;c,d)} \sigma, $$
then for any subsets $L,M,N$ that satisfy (b), the cumulant $\kappa(A_{\tau,L},A_{\rho,M},E_{\mu,N})$ is obtained by evaluating the observable
$$
\frac{1}{k^3(k-1)}\sum_{1\leq a,b\neq c,d\leq k}  \!\!\!(\tau \join \rho \join \mu)(a,b;c,d) + \tau\times\rho\times\mu - (\tau \join \rho)(a,b)\times\mu - (\rho \join \mu)(c,d)\times \tau
$$
on $\pi$. Since the number of parts $L,M,N$ which satisfy (b) is $$\binom{n}{3k-2}\,\frac{(3k-2)!}{((k-1)!)^2\,(k-2)!},$$
 we conclude that the contribution of the case (b) to the limit of the rescaled joint cumulant of $S_n(\tau,\pi)$, $S_n(\rho,\pi)$ and $S_n(\mu,\pi)$ is given by the observable
$$
\frac{1}{k^4}\sum_{1\leq a,b\neq c,d\leq k}  \!\!\!(\tau \join \rho \join \mu)(a,b;c,d) + \tau\times\rho\times\mu - (\tau \join \rho)(a,b)\times \mu - (\rho \join \mu)(c,d)\times\tau.
$$
Hence, we conclude that if $\kappa_3$ is the linear map $(\obs_{\sym,k})^{\otimes 3} \to \obs_{\sym,3k-2}$ defined by
\begin{align*}
\kappa_3(\tau,\rho,\mu) &= \frac{1}{k^4}\sum_{1\leq a,b,c\leq k} \left(\substack{(\tau \join \rho\join \mu)(a,b,c) + 2\,\tau \times \rho \times \mu - (\tau \join \rho)(a,b) \times \mu \\ - (\rho \join \mu)(b,c) \times \tau - (\tau \join \mu)(a,c) \times \rho}\right) \\
&\quad +\frac{1}{k^4}\sum_{\Z/3\Z}\sum_{1 \leq a,b\neq c,d\leq k} \left(\substack{(\tau \join \rho\join \mu)(a,b;c,d) + \tau \times \rho \times \mu \\ - (\tau \join \rho)(a,b) \times \mu - (\rho \join \mu)(c,d) \times \tau}\right)
\end{align*}
with the sum over $\Z/3\Z$ meaning that we permute cyclically the roles played by $\tau$, $\rho$ and $\mu$, then
$$
\frac{\kappa(S_n(\tau,\pi),S_n(\rho,\pi),S_n(\mu,\pi))}{k^2\binom{n}{k}\binom{n-1}{k-1}^2} = \kappa_3(\tau,\rho,\mu)(\gamma)+ O(n^{-1}).
$$
Finally, the asymptotics of the rescaled joint cumulant $\frac{\kappa(\widetilde{S}_n(\tau,\pi),\widetilde{S}_n(\rho,\pi),\widetilde{S}_n(\mu,\pi))}{k^4\,n^{3k-2}}$ are the same, for the same reason as in the case of covariances: when summing over $k$-tuples $(l_1,\ldots,l_k)$, a random variable $\widetilde{A}_{\tau,(l_1,\ldots,l_k)}$ is equal to $A_{\tau,\{l_1,\ldots,l_k\}}$ as soon as $(l_1,\ldots,l_k)$ is an arrangement, and this case is the main contribution when computing the joint cumulant of order $3$.
\end{proof}

\begin{remark}[Cardinality of an amalgamated graphical shuffle product]
Let us give a general formula for the cardinality of the multiset $(\tau \boxtimes \rho)(a,b)$ involved in the definition of $(\tau\join \rho)(a,b)$ and in the computation of the limiting covariances. Knowing $a$ and $b$, in order to construct $\sigma$ in $(\tau \boxtimes \rho)(a,b)$, we first need to decide which indices among those smaller than $i_a=j_b$ will fall in $I_-$, and which indices among those larger than $i_a=j_b$ will fall in $I_-$. If one draws the diagram of $\tau$ with circles and the diagram of $\rho$ with crosses as in Figure \ref{fig:amalgamatedproduct}, then this amounts to choose the \emph{horizontal} positions of the circles and crosses. These choices determine the partition $I_- \sqcup J_- \sqcup K$ of $\lle 1,2k-1\rre$, and there are
$$\binom{a+b-2}{a-1}\binom{2k-a-b}{k-a}$$
possibilities. We then also have to choose the \emph{vertical} positions of the circles and crosses, and this enumeration is the same as before, but with $a$ and $b$ replaced by $\tau(a)$ and $\rho(b)$. Hence,
$$\big|(\tau \boxtimes \rho)(a,b)\big| = \binom{a+b-2}{a-1}\binom{2k-a-b}{k-a}\binom{\tau(a)+\rho(b)-2}{\tau(a)-1}\binom{2k-\tau(a)-\rho(b)}{k-\tau(a)}.$$
\end{remark}
\bigskip

We have thus checked all the hypotheses for the following theorem:
\begin{theorem}[Mod-Gaussian convergence of permutons]\label{thm:modpermuton}
 Let $\pi \in \Scal$ be a permuton, and $\tau$ be a finite permutation in $\sym(k)$.
\begin{enumerate}
	\item The random variable $S_n(\tau,\pi) = \mathrm{occ}(\tau,\sigma_n(\pi))=\binom{n}{k}\,t(\tau,\sigma_n(\pi))$ satisfies the hypotheses of the method of cumulants with parameters $D_n=k\,\binom{n-1}{k-1}$, $N_n=\binom{n}{k}$ and $A=1$, and with limits $\sigma^2 = \kappa_2(\tau,\tau)(\pi)$ and $L = \kappa_3(\tau,\tau,\tau)(\pi)$. The random variable $\widetilde{S}_n(\tau,\pi)=n^k\,t(\tau,\Pi_n(\pi))$ satisfies the hypotheses of the method of cumulants with parameters $\widetilde{D}_n=n^2\,n^{k-1}$, $\widetilde{N}_n=n^k$ and $A=1$, and with the same limits as for $S_n(\tau,\pi)$.

	\item If $\kappa_2(\tau,\tau)(\pi)>0$, then the random variables
	\begin{align*}
	Y_n(\tau,\pi) &= \frac{t(\tau,\sigma_n(\pi))-t(\tau,\pi)}{\sqrt{\var(t(\tau,\sigma_n(\pi)))}};\\
	\widetilde{Y}_n(\tau,\pi) &= \frac{t(\tau,\Pi_n(\pi))-\esper[t(\tau,\Pi_n(\pi))]}{\sqrt{\var(t(\tau,\Pi_n(\pi)))}}
	\end{align*}
	satisfy all the limiting results from Theorem \ref{thm:cumulantestimates}.

	\item We have the concentration inequalities
	\begin{align*}
	\proba[|t(\tau,\sigma_n(\pi))-t(\tau,\pi)| \geq x] &\leq 2\,\exp\left(-\frac{nx^2}{9k^2}\right);\\
	\proba[|t(\tau,\Pi_n(\pi))-\esper[t(\tau,\Pi_n(\pi))]| \geq x] &\leq 2\,\exp\left(-\frac{nx^2}{9k^2}\right).
	\end{align*}
\end{enumerate}
\end{theorem}
\noindent Regarding the concentration inequalities, the first one appears in \cite[Theorem 4.2]{HKMS11} with a constant $\frac{1}{2}$ instead of $\frac{1}{9}$; the second one seems new. On the other hand, the second part of the theorem implies immediately the estimates of Theorems \ref{thm:speedofconvergencepermuton} and \ref{thm:moderatedeviationpermuton}, with
\begin{align*}
\lim_{n \to \infty} n\,\var(t(\tau,\sigma_n(\pi)))&= \lim_{n \to \infty} n\,\var(t(\tau,\Pi_n(\pi))) = k^2\,\kappa_2(\tau,\tau)(\pi) ;\\
l(\tau,\pi) &= \frac{k}{6}\,\frac{\kappa_3(\tau,\tau,\tau)(\pi)}{(\kappa_2(\tau,\tau)(\pi))^{3/2}}.
\end{align*}
Besides, we have similar estimates for the random variables $\widetilde{Y}_n(\tau,\pi)$ instead of the variables $Y_n(\tau,\pi)$.

\begin{example}
Let $\tau=21$. The quantity $S_n(\tau,\pi)$ is the number of inversions of the random permutation $\sigma_n(\pi)$. One computes
\begin{align*}
\kappa_2(\tau,\tau) &= 321 + \frac{1}{3}(312+231) - 4321 - \frac{2}{3}(3412+3421+4231+4312) \\
&\quad-\frac{1}{3}(2143+2413+2431+3142+3241+4132+4213).
\end{align*}
To fix the ideas, let us consider the case when $\pi$ is the uniform measure on $[0,1]^2$. Then, $t(\tau,\pi)=\frac{1}{k!}$ for any permutation $\tau$ of size $k$, therefore, $\kappa_2(\tau,\tau)(\pi) = \frac{1}{36}$. Theorem \ref{thm:speedofconvergencepermuton} yields then
$$\dkol\!\left(3\sqrt{n}\left(\frac{\mathrm{inv}(\sigma_n(\pi))}{\binom{n}{2}}-\frac{1}{2}\right),\,\gauss\right)  \leq \frac{33000}{\sqrt{n}}$$
for $n$ large enough, where $\mathrm{inv}(\sigma)$ is the number of inversions of a permutation $\sigma$.
This recovers, up to the value of the constant,
a result of Fulman \cite{Ful04}.
Similar estimates of the speed of convergence
in the central limit theorem of the number of occurences of other patterns
(and even vincular patterns) in uniform random permutations
have recently been given in \cite{hofer2017central}.
In the general case of a random permutation $\sigma_n(\pi)$, 
such estimates seem however to be new.
\end{example}
\medskip

\begin{remark}
  [on hypothesis $\kappa_2(\tau,\tau) >0$]
  The most studied case is of course the case of the uniform permuton: {\em i.e.}
  $\pi$ is the Lebesgue measure on $[0,1]^2$,
  which implies that $\sigma_n(\pi)$ is a uniform random permutation of size $n$.
  Then one can prove that the hypothesis $\kappa_2(\tau,\tau) >0$
  is satisfied for all patterns $\tau$:
  see \cite{janson2015asymptotic,hofer2017central}.
\end{remark}

\subsection{Fluctuations of random integer partitions}\label{subsec:partmodgauss}
The case of the models of random partitions $(\lambda_n(\omega))_{n \in \N}$ is very similar to the two other cases, and most of the arguments that we shall use come from \cite[Chapter 11]{FMN16}. The main new ingredient will be a change of basis argument in the Kerov--Olshanski algebra of polynomial functions on Young diagrams, which is used in the proof of:
\begin{proposition}[Uniform bounds for the cumulants of the Frobenius moments]\label{prop:boundcumulantprho}
Let $\omega \in \Pcal$ be a parameter of the Thoma simplex, and $\rho$ an integer partition with size $k$. The random variable $S_n(\rho,\omega)=p_\rho(\lambda_n(\omega)) = n^k\,t(\rho,\Omega_n(\omega))$ satisfies the uniform bound on cumulants \eqref{eq:cumulant1} with parameters $D_n=k^2\,n^{k-1}$, $N_n = n^k$ and $A=1$.
\end{proposition}
\begin{proof}
Given two integer partitions $\lambda$ and $\rho$ with arbitrary sizes, we define the renormalised character value $\varSigma_\rho(\mu)$ by the following formula:
$$\varSigma_\rho(\mu) = \begin{cases}
	|\mu|(|\mu|-1)\cdots (|\mu|-|\rho|+1)\,\chi^\mu(\sigma_\rho)&\text{if }|\mu|\geq |\rho|,\\
	0 &\text{if }|\mu|<|\rho|,
\end{cases}$$
where $\sigma_\rho$ is a permutation with cycle-type $\rho$, and $\chi^\mu(\sigma)= \frac{\mathrm{tr}\,\rho^\mu(\sigma)}{\mathrm{tr}\,\rho^\mu(1)}$ is the normalised irreducible character of the symmetric group $\sym(|\mu|)$ 
that is associated with the integer partition $\mu$. We refer to \cite{IO02} and \cite[Section 7.3]{Mel17} for a presentation of these functions, and to \emph{loc.~cit.}~or to \cite{Sag01} for generalities on the representations of the symmetric groups. The functions $\varSigma_\rho$ with $\rho \in \pym$ span an algebra called the Kerov--Olshanski algebra of polynomial functions on Young diagrams \cite{KO94}. The use of the renormalised character values in the study of the central measures is natural since these measures $\proba_{n,\omega}$ were given in Section \ref{subsec:thoma} a representation theoretic definition.\medskip

In \cite[Section 11]{FMN16}, we proved the mod-Gaussian convergence of the random character values
$$\varSigma_\rho(\lambda_n(\omega)) = n^{\downarrow k}\,\chi^{\lambda_n(\omega)}(\sigma_\rho)\quad \text{for }n\geq k=|\rho|$$
after an appropriate scaling. Remark 11.4.2 in \emph{loc.~cit.}~ensures that for any integer partitions $\rho^{(1)},\ldots,\rho^{(r)}$, 
$$
\left|\kappa\!\left(\varSigma_{\rho^{(1)}}(\lambda_n(\omega)),\varSigma_{\rho^{(2)}}(\lambda_n(\omega)),\ldots,\varSigma_{\rho^{(r)}}(\lambda_n(\omega))\right)\right|\leq \left(\frac{2k^2}{n}\right)^{r-1}\,r^{r-2}\,n^{\downarrow k_1} n^{\downarrow k_2}\cdots n^{\downarrow k_r}$$
where $k_i = |\rho^{(i)}|$, and $k = \max(k_1,\ldots,k_r)$. 
On the other hand, the functions $\varSigma_\rho$ are related to the observables of integer partitions and Thoma parameters by a change of basis formula
$$p_\rho = \sum_{|\mu|\leq |\rho|} c_{\rho,\mu}\,\varSigma_\mu, $$
where the coefficients $c_{\rho,\mu}$ are positive rational numbers, and where the only non-zero coefficient $c_{\rho,\mu}$ with $|\mu|=|\rho|$ is $c_{\rho,\rho}=1$. For instance,
$$p_6 = \varSigma_6 + 6\,\varSigma_{(3,2)} + 6\,\varSigma_{(4,1)} + \frac{95}{4}\,\varSigma_4 + 15\,\varSigma_{(2,1,1)} + 35\,\varSigma_{(2,1)} + \frac{91}{16}\,\varSigma_2.$$
We refer to \cite[\S3.6]{Wass81} and \cite[Section 3]{IO02} for a description of the relations between the observables $p_\rho$ and the observables $\varSigma_\rho$; see also \cite[Section 7.3]{Mel17}. It is not needed to know exactly what are the coefficients $c_{\rho,\mu}$; we shall only use the fact that 
$$n^{|\rho|} \geq \prod_{i=1}^r \left(\left(n-\frac{1}{2}\right)^{\rho_i} - \left(-\frac{1}{2}\right)^{\rho_i}\right) = p_\rho(1^n) =  \sum_{|\mu|\leq |\rho|} c_{\rho,\mu}\,\varSigma_\mu(1^n) = \sum_{|\mu| \leq |\rho|} c_{\rho,\mu}\,n^{\downarrow |\mu|}.$$
We can expand the $r$-th cumulant of $p_\rho(\lambda_n(\omega))$ by multilinearity:
\begin{align*}
|\kappa^{(r)}(p_\rho(\lambda_n(\omega)))| &\leq \sum_{\mu^{(1)},\ldots,\mu^{(r)}} c_{\rho,\mu^{(1)}}\cdots c_{\rho,\mu^{(r)}} \left|\kappa\!\left(\varSigma_{\mu^{(1)}}(\lambda_n(\omega)),\ldots,\varSigma_{\mu^{(r)}}(\lambda_n(\omega))\right)\right| \\
&\leq \left(\frac{2k^2}{n}\right)^{r-1}\,r^{r-2}\,\sum_{\mu^{(1)},\ldots,\mu^{(r)}} c_{\rho,\mu^{(1)}}\cdots c_{\rho,\mu^{(r)}}\,n^{\downarrow |\mu^{(1)}|}\cdots n^{\downarrow |\mu^{(r)}|}\\
&\leq \left(\frac{2k^2}{n}\right)^{r-1}\,r^{r-2}\,n^{kr} = n^{k}\,(2k^2n^{k-1})^{r-1}\,r^{r-2}.\qedhere
\end{align*}
\end{proof}
\begin{remark}[Dependency graphs for Frobenius moments (absence of)]
Proposition \ref{prop:boundcumulantprho} is the only case in this paper where we obtain a uniform bound on cumulants without specifying a dependency graph for the underlying random variables. 
On the other hand, the bound on the joint cumulants of the variables $\varSigma_\rho(\lambda_n(\omega))$ follows indeed from the theory of dependency graphs, but in a setting of non-commutative probability theory: see \cite[Section 11.3]{FMN16}.
\end{remark}
\medskip

We now introduce the maps $\kappa_2$ and $\kappa_3$ on the algebra of partitions. Given two integer partitions $\rho=(\rho_1,\ldots,\rho_{\ell(\rho)})$ and $\mu=(\mu_1,\ldots,\mu_{\ell(\mu)})$ with size $k$, for $a \in \lle 1,\ell(\rho)\rre$ and $b \in \lle 1,\ell(\mu)\rre$, we set 
$$(\rho \join \mu)(a,b) = (\rho \setminus \{\rho_a\})\sqcup (\mu \setminus \{\mu_b\}) \sqcup \{\rho_a+\mu_b-1\},$$
the right-hand side being reordered in order to obtain an integer partition. For instance, $((3,2,2) \join (4,1))(2,2) = (4,3,2,2)$. The map $\kappa_2 : \obs_{\pym,k} \otimes \obs_{\pym,k} \to \obsP$ is then defined by
$$\kappa_2(\rho,\mu) = \frac{1}{k^2}\sum_{a=1}^{\ell(\rho)} \sum_{b=1}^{\ell(\mu)} \rho_a\mu_b\, ((\rho\join \mu)(a,b) -\rho \times \mu).$$
Similarly, given three integer partitions $\rho$, $\mu$ and $\nu$ in $\pym(k)$, we define
\begin{align*}
(\rho\join \mu\join \nu)(a,b,c) &= (\rho \setminus \{\rho_a\})\sqcup (\mu \setminus \{\mu_b\}) \sqcup (\nu \setminus \{\nu_c\}) \sqcup \{\rho_a+\mu_b+\nu_c-2\} ;\\
(\rho \join \mu \join \nu)(a,b;c,d) &=(\rho \setminus \{\rho_a\})\sqcup (\mu \setminus \{\mu_b,\mu_c\}) \sqcup (\nu \setminus \{\nu_d\}) \sqcup \{\rho_a+\mu_b-1,\mu_c+\nu_d-1\}
\end{align*}
where in the first case we have $a \in \lle 1,\ell(\rho)\rre$, $b \in \lle 1,\ell(\mu)\rre$ and $c \in \lle 1,\ell(\nu)\rre$; and in the second case we have $a \in \lle 1,\ell(\rho)\rre$, $b\neq c \in \lle 1,\ell(\mu)\rre$ and $d \in \lle 1,\ell(\nu)\rre$. We then set
\begin{align*}
\kappa_3(\rho,\mu,\nu) &= \frac{1}{k^4}\sum_{a=1}^{\ell(\rho)} \sum_{b=1}^{\ell(\mu)} \sum_{c=1}^{\ell(\nu)} \rho_a\mu_b\nu_c \left(\substack{(\rho \join \mu\join \nu)(a,b,c) + 2\,\rho \times \mu \times \nu - (\rho \join \mu)(a,b) \times \nu \\ - (\mu \join \nu)(b,c) \times \rho - (\rho \join \nu)(a,c) \times \mu}\right)\\
&\quad + \frac{1}{k^4}\sum_{\Z/3\Z} \sum_{a=1}^{\ell(\rho)} \sum_{b=1}^{\ell(\mu)} \sum_{c=1}^{\ell(\nu)} \rho_a\mu_b(\mu_b-1)\nu_c \left(\substack{\rho\times \mu\times \nu + (\rho\join \mu \join \nu)(a,b,c) \\ - (\rho \join \mu)(a,b) \times \nu - (\mu \join \nu)(b,c)\times \rho}\right) \\
&\quad + \frac{1}{k^4}\sum_{\Z/3\Z} \sum_{a=1}^{\ell(\rho)} \sum_{(b\neq c)=1}^{\ell(\mu)} \sum_{d=1}^{\ell(\nu)} \rho_a\mu_b\mu_c\nu_d \left(\substack{\rho \times \mu \times \nu + (\rho \join \mu \join \mu)(a,b;c,d) \\ -(\rho \join \mu)(a,b) \times \nu - (\mu \join \nu)(c,d)\times \rho}\right)
\end{align*}
with the sums over $\Z/3\Z$ meaning that we permute cyclically the roles played by $\rho$, $\mu$ and $\nu$. We then have:
\begin{proposition}[Limiting first cumulants of Frobenius moments]
For any partitions $\rho$, $\mu$ and $\nu$ with size $k$ and any Thoma parameter $\omega \in \Pcal$,
\begin{align*}
\lim_{n \to \infty} \frac{\kappa(\varSigma_\rho(\lambda_n(\omega)), \varSigma_\mu(\lambda_n(\omega)))}{k^2\,n^{2k-1}} &= \lim_{n \to \infty} \frac{\kappa(S_n(\rho,\omega)), S_n(\mu,\omega))}{k^2\,n^{2k-1}} = \kappa_2(\rho,\mu)(\omega);
\end{align*}
and 
\begin{align*}
\lim_{n \to \infty} \frac{\kappa(\varSigma_\rho(\lambda_n(\omega)), \varSigma_\mu(\lambda_n(\omega)),\varSigma_\nu(\lambda_n(\omega)))}{k^4\,n^{3k-2}} &= \lim_{n \to \infty} \frac{\kappa(S_n(\rho,\omega)), S_n(\mu,\omega),S_n(\nu,\omega))}{k^4\,n^{3k-2}}\\
& = \kappa_3(\rho,\mu,\nu)(\omega),
\end{align*}
In each case, the limit is attained at speed $O(n^{-1})$, with a constant in the $O(\cdot)$ that depends only on $k$.
\end{proposition}
\begin{proof}
The case of the random character values is the content of \cite[Propositions 11.4.3 and 11.4.4]{FMN16}, and the case of the observables $S_n(\rho,\omega)=p_\rho(\lambda_n(\omega))$ follows by using the expansion $p_\rho = \varSigma_\rho + \text{terms with lower degree}$ in the algebra of polynomial functions on Young diagrams.
\end{proof}
\medskip

\begin{theorem}[Mod-Gaussian convergence of random integer partitions] \label{thm:modpartition}
Fix $\omega \in \Pcal$ and $\rho \in \pym(k)$.
\begin{enumerate}
	\item The random variable $S_n(\rho,\omega) = n^k\,t(\rho,\Omega_n(\omega))$ satisfies the hypotheses of the method of cumulants with parameters $D_n=k^2\,n^{k-1}$, $N_n=n^k$ and $A=1$, and with limits $\sigma^2 = \kappa_2(\rho,\rho)(\omega)$ and $L = \kappa_3(\rho,\rho,\rho)(\omega)$.

	\item If $\kappa_2(\rho,\rho)(\omega)>0$, then the random variables
	\begin{align*}
	Y_n(\rho,\omega) &= \frac{t(\rho,\Omega_n(\omega))-\esper[t(\rho,\Omega_n(\omega))]}{\sqrt{\var(t(\rho,\Omega_n(\omega)))}}
	\end{align*}
	satisfy all the limiting results from Theorem \ref{thm:cumulantestimates}. Moreover, $|\esper[t(\rho,\Omega_n(\omega))] - t(\rho,\omega)| \leq \frac{2k^2}{n}$.

	\item We have the concentration inequalities
	\begin{align*}
	\proba[|t(\rho,\Omega_n(\omega))-\esper[t(\rho,\Omega_n(\omega))]| \geq x] &\leq 2\,\exp\left(-\frac{nx^2}{9k^2}\right);\\ 
	\proba[|t(\rho,\Omega_n(\omega))-t(\rho,\omega)| \geq x] &\leq 4\,\exp\left(-\frac{nx^2}{9k^2}\right).
	\end{align*}
\end{enumerate}
\end{theorem}
\begin{proof}
The previous discussion proves the first part of the theorem, and the asymptotic results on $Y_n(\rho,\omega)$ follow immediately by Theorem \ref{thm:cumulantestimates}. For the bound on $\esper[t(\rho,\Omega_n(\omega))] - t(\rho,\omega)$, we use the identity $\esper[\varSigma_{\rho}(\lambda_n(\omega))]  = n^{\downarrow k}\,t(\rho,\omega)$, and the expansion $$p_\rho = \sum_{|\mu|\leq |\rho|} c_{\rho,\mu}\,\varSigma_\mu = \varSigma_\rho + \sum_{|\mu|< |\rho|} c_{\rho,\mu}\,\varSigma_\mu.$$
Thus,
\begin{align*}
|\esper[t(\rho,\Omega_n(\omega))] - t(\rho,\omega)| &\leq \frac{1}{n^k} \left( \sum_{|\mu|<|\rho|} c_{\rho,\mu}\,|t(\mu,\omega)| + (n^k-n^{\downarrow k}) |t(\rho,\omega)| \right) \\
&\leq \frac{2}{n^k} \left( n^k-n^{\downarrow k}\right) \leq \frac{2k^2}{n}.
\end{align*}
The first concentration inequality follows then from Proposition \ref{prop:concentration1}, and the second concentration inequality is obtained by the following computation. One can assume without loss of generality $x \leq 1$, and then
\begin{align*}
\proba[|t(\rho,\Omega_n(\omega))-t(\rho,\omega)| \geq x] &\leq \proba\!\left[|t(\rho,\Omega_n(\omega))-\esper[t(\rho,\Omega_n(\omega))]| \geq \max\!\left(0,x-\frac{2k^2}{n}\right)\right] \\
&\leq 2\,\exp\left(-\frac{n}{9k^2}\,\max\!\left(0,x-\frac{2k^2}{n}\right)^{\!2}\right)\\
&\leq 2\,\exp\left(-\frac{n}{9k^2}\left(x^2-\frac{4k^2x}{n}\right)\right)\\
&\leq 2\,\E^{\frac{4}{9}}\,\exp\left(-\frac{nx^2}{9k^2}\right) \leq 4\,\exp\left(-\frac{nx^2}{9k^2}\right).\qedhere
\end{align*}
\end{proof}

\noindent We recover in particular the exponential concentration of the central measures stated in Theorem \ref{thm:concentrationpartitions}; this result completes the law of large numbers obtained by Kerov and Vershik in \cite{KV81}, and the central limit theorem obtained independently by Bufetov and Méliot in \cite{Buf12,Mel12} (see also \cite[Section 12.3]{Mel17}).
\bigskip

\section{Mod-Gaussian moduli spaces}\label{sec:moduli}
The similarities that we encountered when studying subgraphs in random graphs,
patterns in random permutations and Frobenius moments of random partitions makes it tempting to try to formalise this common structure by the following definition.

\begin{definition}[Mod-Gaussian moduli space]\label{def:modgaussianMS}
 Let $\Mcal$ be a compact metrisable space, and $\obs_{\mathfrak{M}}=\bigoplus_{k \in \N} \obs_{\mathfrak{M},k}$ be a graded algebra. We say that the pair $(\Mcal,\obs_{\mathfrak{M}})$ can be endowed with a structure of \emph{mod-Gaussian moduli space} if the following conditions are satisfied:
\begin{enumerate}[label=(MS\arabic*)]
\item\label{item:MS1} The pair is endowed with a morphism of algebras $\Psi : \obs_{\mathfrak{M}} \to \Ccal(\Mcal)$ whose image is a dense subalgebra of $\Ccal(\Mcal)$. In particular, by Stone--Weierstrass theorem, 
$$ \big(m_n \to_{n \to \infty} m \in \Mcal\big) \iff \big(\forall f \in \obs,\,\,\,\Psi(f)(m_n) \to_{n\to \infty} \Psi(f)(m)\big).$$ 
In the sequel we simply denote $\Psi(f)(m) = f(m)$.
\item\label{item:MS2} The space $\Mcal$ is also endowed with a construction of random objects which have a property of asymptotic concentration. 
  Namely, for each parameter $m \in \Mcal$, there is a family of random variables $(M_n(m))_{n \in \N}$ in $\Mcal$, such that $M_n(m) \to_{n \to \infty} m$ in probability.
\item\label{item:MS3} For any $k \geq 0$, there exist two linear maps
\begin{align*}
\kappa_2 &: \obs_{\mathfrak{M},k} \otimes \obs_{\mathfrak{M},k} \to \obs_{\mathfrak{M}}, \\
\kappa_3 &: \obs_{\mathfrak{M},k} \otimes \obs_{\mathfrak{M},k} \otimes \obs_{\mathfrak{M},k} \to \obs_{\mathfrak{M}};
\end{align*}
two sequences $N_{n,k} \to \infty$ and $D_{n,k} = o(N_{n,k})$; and a distinguished linear basis $\mathfrak{M}(k)$ of $\obs_k$ such that for any $f \in \mathfrak{M}(k)$, the sequence of random variables 
$$S_n(f,m)=N_{n,k}\,f(M_n(m))$$
satisfies the hypotheses of the method of cumulants with parameters $(D_{n,k},N_{n,k},1)$, and with limiting variance $\sigma^2(m)=\kappa_2(f,f)(m)$ and third cumulant $L(m)=\kappa_3(f,f,f)(m)$.
 \end{enumerate} 
\end{definition}
\noindent Rigorously, a structure of mod-Gaussian moduli space is given by a family 
$$(\Mcal,\obs_{\mathfrak{M}};\Psi;(M_n(m))_{m\in \Mcal,\,\,n\in \N};\kappa_2,\kappa_3;(\mathfrak{M}(k))_{k \in \N};(N_{n,k})_{n,k\in \N},(D_{n,k})_{n,k\in \N}).$$
However, in the sequel, we shall just speak of the pair $(\Mcal,\obs_{\mathfrak{M}})$, being understood that this means that there is a canonical (and interesting!) associated family 
$$(\Psi;(M_n(m))_{m\in \Mcal,\,\,n\in \N};\kappa_2,\kappa_3;(\mathfrak{M}(k))_{k \in \N};(N_{n,k})_{n,k\in \N},(D_{n,k})_{n,k\in \N}).$$
All the results of this article can be summarised by the following statement: the pairs $(\Gcal,\obsG)$, $(\Scal,\obsS)$ and $(\Pcal,\obsP)$ are three mod-Gaussian moduli spaces, each time with $N_{n,k}=n^k$ and $D_{n,k}=k^2n^{k-1}$. In addition to the concision of this statement, the interest of the notion of mod-Gaussian moduli spaces lies in the following remarks, by which we conclude this paper.

\begin{remark}[Moduli spaces and convergence in probability]
Given a mod-Gaussian moduli space $(\Mcal,\obs_{\mathfrak{M}})$, since $\Mcal$ is metrisable, there is a distance $d$ on $\Mcal$ which corresponds to the topology:
$$ \big(m_n \to_{n \to \infty} m \in \Mcal\big) \iff \big(d(m_n,m) \to 0\big).$$ 
However, in many cases, $\Mcal$ is a part of an infinite-dimensional vector space, and the distances corresponding to its topology are not very practical to deal with. For instance, in the space of graphons, the cut-metric indeed corresponds to the topology of $\Gcal$, but given two graphs $G$ and $H$ with arbitrary sizes, their distance $\delta_\oblong(G,H)$ is not easy to manipulate or even compute, in particular if $|G|\neq |H|$ (see for instance \cite[Section 2.3]{BCLSV08}). The same remark applies to the two other cases of moduli spaces studied in the article, or to the space of probability measures evoked hereafter. This is the reason why it is better to control the topology by an algebra of observables $\obs_{\mathfrak{M}}$, in particular for the study of random models that are convergent in probability. Indeed, the convergence in probability of a sequence of random variables $(M_n)_{n \in \N}$ towards $m$, which is defined by
$$\forall \eps>0,\,\,\,\proba[d(M_n,m)\geq \eps]   \to_{n \to \infty} 0,$$
can then be characterised by
$$\forall \eps>0,\,\,\,\forall f \in \obs_{\mathfrak{M}},\,\,\,\proba[|f(M_n)-f(m)|\geq \eps]   \to_{n \to \infty} 0,$$
and this second condition is usually much easier to check (assuming that one has chosen an adequate algebra of observables). In the mod-Gaussian framework, we saw that these probabilities $\proba[|f(M_n)-f(m)|\geq \eps]$ are exponentially small.
\end{remark}
\medskip

\begin{remark}[Classical theorems of probability theory]
The framework of mod-Gaussian and mod-$\phi$ convergent sequences that was developed in particular in \cite{JKN11,DKN15,FMN16,FMN17} can be seen as a far reaching generalisation of the classical theorems of probability theory: law of large numbers, central limit theorem, Cramér's large deviation estimates, Berry--Esseen estimates of the speed of convergence, \emph{etc.} for the scaled sums of i.i.d. random variables. Let us explain how to summarise all these results by constructing an adequate mod-Gaussian moduli space. Let $X$ be a compact metrisable space, and $\Mcal^1(X)$ be the space of Borel probability measures on $X$, endowed with the topology of convergence in law. It is well known that $\Mcal^1(X)$ is again compact metrisable, see \cite[Chapter 1, Section 5]{Bil69}. Let $\mathscr{A}$ be a dense subalgebra of the separable algebra $\Ccal(X)$, endowed with a distinguished countable basis $\mathfrak{A}$ (thus, $\mathscr{A} = \mathrm{Span}_{\R}(\mathfrak{A})$). We can assume without loss of generality that $\|f\|_{\infty}=1$ for any $f \in \mathfrak{A}$. A graded algebra of observables for $\Mcal^1(X)$ is the symmetric algebra 
$$\obs_{\mathfrak{A}} = S(\mathscr{A}) = \bigoplus_{k=0}^\infty S^k(\mathscr{A}),$$
where $S^k(\mathscr{A})$ is the quotient of $\mathscr{A}^{\otimes k}$ by the relations $t^\sigma=t$ for any $k$-tensor $t$ and any permutation $\sigma \in \sym(k)$. We evaluate a tensor $t=f_1\otimes f_2\otimes \cdots \otimes f_k$ on a probability measure $\pi \in \Mcal^1(X)$ by
$$t(\pi) = f_1(\pi)\,f_2(\pi)\,\cdots\,f_k(\pi) = \int_{X^k} f_1(x_1)\cdots f_k(x_k)\,\pi^{\otimes k}(\!\DD{x_1}\cdots \!\DD{x_k}).$$
It is easy to see that the convergence in law in $\Mcal^1(X)$ amounts to the convergence of all the observables $t \in \obs_{\mathfrak{A}}$. On the other hand, for any $\pi \in \Mcal^1(X)$, a way to approximate $\pi$ by a discrete "combinatorial" object is by using the empirical measures
$$\Pi_n(\pi) = \frac{1}{n}\sum_{i=1}^n \delta_{x_i},$$
where the $x_i$'s are independent identically distributed variables with law $\pi$. By the classical law of large numbers, $\Pi_n(\pi) \to \pi$ in probability and for any probability measure $\pi \in \Mcal^1(X)$. It is then easy to prove by using the theory of dependency graphs that one has in fact a structure of mod-Gaussian moduli space on $(\Mcal^1(X),\obs_{\mathfrak{A}})$. This structure is associated with the parameters $N_{n,k}=n^k$ and $D_{n,k}=k^2\,n^{k-1}$ and to the maps
\begin{align*}
&\kappa_2(f_1\otimes \cdots \otimes f_k,\,g_1 \otimes \cdots \otimes g_k) \\
&=\frac{1}{k^2}\sum_{1\leq a,b \leq k} (f_ag_b - f_a \otimes g_b)\otimes \left(\bigotimes_{a' \neq a}f_{a'}\right) \otimes \left(\bigotimes_{b' \neq b} g_{b'}\right) ;\\
&\kappa_3(f_1 \otimes \cdots \otimes f_k,\,g_1 \otimes \cdots \otimes g_k,h_1 \otimes \cdots \otimes h_k) \\
&= \frac{1}{k^4} \sum_{1 \leq a,b,c \leq k} \kappa_3(f_a,g_b,h_c) \otimes \left(\bigotimes_{a' \neq a}f_{a'}\right) \otimes \left(\bigotimes_{b' \neq b} g_{b'}\right) \otimes \left(\bigotimes_{c' \neq c} h_{c'}\right),
\end{align*}
where for $\kappa_3$, 
$$\kappa_3(f_a,g_b,h_c) = f_ag_bh_c-(f_ag_b)\otimes h_c - (f_ah_c)\otimes g_b - (g_bh_c)\otimes f_a + 2\,f_a\otimes g_b \otimes h_c.$$
Thus, this basic setting in classical probability theory
can be integrated in our Definition \ref{def:modgaussianMS}. 
We think that many other classes of random models give rise to a mod-Gaussian moduli space, in particular when one tries to approximate an object $m$ by a random sequence $(M_n(m))_{n \in \N}$ of "discrete" objects. For instance,  De Catelan has proven recently that the space of metric measured spaces introduced in \cite[Section 3$\tfrac{1}{2}$]{Gro07} and in \cite{GPW09} gives rise to such a structure, and that the singular points of this mod-Gaussian moduli space (see the next remarks) are exactly the compact homogeneous spaces $X=G/K$. 
\end{remark}
\medskip

\begin{remark}[Moduli spaces as fields of Gaussian fluctuations]
Let us explain the intuition behind our definition of mod-Gaussian moduli space. It is convenient to represent a mod-Gaussian moduli space by a surface as in Figure \ref{fig:field}.
\begin{center}
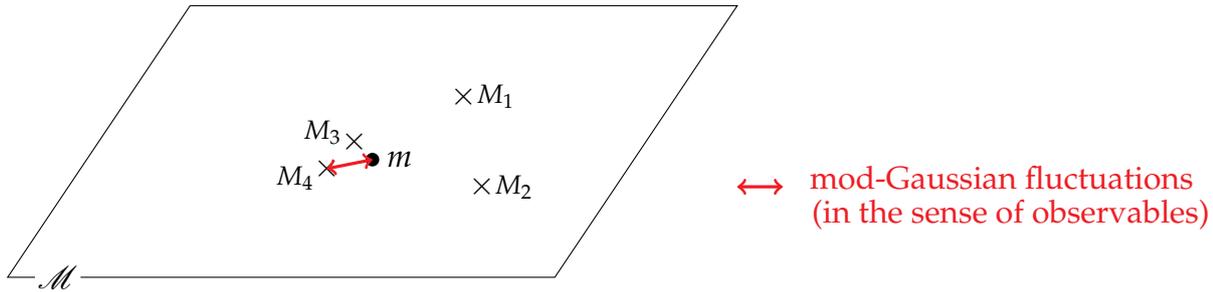
\begin{figure}[ht]
\begin{tikzpicture}[scale=1.2]
\draw (0.3,0) -- (0,0) -- (2,3) -- (8,3) -- (6,0) -- (0.8,0);
\draw (0.55,0) node {$\Mcal$};
\fill (4,1.3) circle (0.75mm);
\draw (4.3,1.3) node {$m$};
\foreach \x in {(5,2),(3.8,1.5),(5.2,1),(3.5,1.2)}
\draw \x node {$\times$};
\draw (5.35,2) node {\footnotesize $M_1$};
\draw (5.55,1) node {\footnotesize $M_2$};
\draw (3.15,1.1) node {\footnotesize $M_4$};
\draw (3.45,1.6) node {\footnotesize $M_3$};
\draw [Red,very thick,<->] (4,1.3) -- (3.5,1.2);
\draw [Red,very thick,<->] (8.5,1) -- (8,1);
\draw [Red] (10.9,1.1) node {mod-Gaussian fluctuations};
\draw [Red] (11,0.7) node {(in the sense of observables)};
\end{tikzpicture}
\caption{Mod-Gaussian moduli spaces as fields of fluctuations.\label{fig:field}}
\end{figure}
\end{center}

\noindent For any parameter $m \in \Mcal$, the random model $M_n(m)$ can be seen as a canonical way to construct random perturbations $M_n(m)$ of the parameter $m$. These random objects are perturbations, because by Hypothesis \ref{item:MS2}, as $n$ goes to infinity, $M_n(m)$ converges back to the parameter $m$. In all the examples that we looked at, each random variable $M_n(m)$ corresponds to a random combinatorial object of size $n$: graph with $n$ vertices, permutation on $n$ points, integer partition with size $n$. The parameter $m$ drives then the asymptotics of the random model $(M_n(m))_{n\in \N}$. 
The \emph{mod-Gaussian} part of the definition amounts to ask that the fluctuations $f(M_n(m))-\esper[f(M_n(m))]$ are mod-Gaussian convergent after an appropriate renormalisation, for any function $f$ in a dense subalgebra $\Psi(\obs_{\mathfrak{M}})$ of $\Ccal(\Mcal)$. The exact renormalisation required may depend on the observable $f$, but only through a gradation of the algebra of observables $\obs_{\mathfrak{M}}$. On the other hand, the asymptotics of the second and third cumulants of the observables can be encoded by simple linear functions $(\obs_{\mathfrak{M},k})^{\otimes 2} \to \obs$ and $(\obs_{\mathfrak{M},k})^{\otimes 3} \to \obs$. \medskip

So, we can think of a mod-Gaussian moduli space as a space whose elements $m$ have a way to generate random perturbations of themselves, which correspond to random combinatorial models, and which are asymptotically Gaussian, and even mod-Gaussian. The interest of the definition is that if one can show that some space $\Mcal$ is a mod-Gaussian moduli space, then one gets at once numerous precise probabilistic estimates (moderate deviations, bounds of Berry--Esseen type, \emph{etc.}), and this for an extremely large family of random variables. 
\end{remark}
\medskip

\begin{remark}[Singularities of a mod-Gaussian moduli space]\label{rem:singularity}
Given a mod-Gaussian moduli space $(\Mcal,\obs_{\mathfrak{M}})$, if $f$ is an element of $\mathfrak{M}(k)$ and if $m \in \Mcal$, assuming that $\kappa_2(f,f)(m) > 0$, all the results from Theorem \ref{thm:cumulantestimates} apply to the random variable
$$Y_n(f,m) = \frac{f(M_n(m))-\esper[f(M_n(m))]}{\sqrt{\var(f(M_n(m)))}}.$$
We did not discuss yet the case when the limiting variance $\kappa_2(f,f)(m)$ vanishes. In this situation, we cannot apply Theorem \ref{thm:cumulantestimates}, and the generic renormalisation of $f(M_n(m))$ does not yield Gaussian fluctuations (or to be more precise, it yields random variables that converge in probability to $0$). We then say that $m$ is a \emph{singular point} or \emph{singularity} of the space $\Mcal$ (with respect to the observable $f$). The idea is that certain parameters $m \in \Mcal$ correspond to models with smaller variances than usual (again, with respect to some observables $f$). A typical example is when $\Mcal=\Gcal$ and when $\gamma$ is the graphon corresponding to a constant graph function $g=p$, with $p \in [0,1]$. The corresponding random graphs $G_n(p)$ are the Erdös--Rényi random graphs with parameter $p$. It is easy to see that for any finite graph $F$,
$$t(F,p) = p^{|E_F|}.$$
Since the number of edges in a product $F \times G$ and in a joined graph $(F \join G)(a,b)$ is the same, we have $\kappa_2(F,F)(p)=0$ for any graph $F$. Thus, the graphon $p$ is a \emph{global singularity} of the moduli space of graphons $\Gcal$ (\emph{i.e.}, it is a singularity with respect to any graph observable $F$). In this particular case, one can show that the variance of $t(F,G_n(p))$ is generically of order $\frac{1}{n^2}$ instead of $\frac{1}{n}$; and one can even show the mod-Gaussian convergence of an appropriate (non-generic) renormalisation of $t(F,G_n(p))$, see \cite[Chapter 10]{FMN16}.\medskip

Another example of a globally singular point is with $\Mcal = \Pcal$ and $\omega = \omega_0=(0,0)$. This point corresponds to the celebrated Plancherel measures on integer partitions, and for any integer partition $\rho$, one has 
$$t(\rho,\omega_0) = \begin{cases}
	1 &\text{if }\rho=1^k,\\
	0 &\text{otherwise}.
\end{cases}$$
This implies the vanishing of the limiting variance $\kappa_2(\rho,\rho)(\omega_0)$ for any integer partition $\rho$. Again, in this particular case, the random variables $\varSigma_\rho(\lambda_n(\omega_0))$ can be shown to satisfy a central limit theorem after an appropriate non-generic renormalisation; see \emph{e.g.}~\cite{IO02} for the details. It is also conjectured that this singular point corresponds to mod-Gaussian fluctuations, but with different parameters $N_{n,k}$ and $D_{n,k}$ than in the generic case.
\end{remark}
\medskip

\begin{remark}[Link with the geometric notion of moduli space]
In algebraic geometry, the \emph{moduli spaces} are geometric spaces whose points parametrise (isomorphism classes of) geometric objects of some fixed kind. For instance, the smooth projective complex curves of genus $g$ and with $n$ marked points are parametrised by the moduli space $\mathcal{M}_{g,n}$, which is itself a geometric space of dimension $3g-3+n$ (to be precise, $\mathcal{M}_{g,n}$ is an algebraic stack). In particular, the elliptic curves (smooth projective curves of genus $1$ with $1$ marked point) are classified by the moduli space $\mathcal{M}_{1,1}$, which is also the one-dimensional complex orbifold $\mathrm{SL}(2,\Z)\backslash\!\!\backslash \mathbb{H}$ (see \emph{e.g.} \cite{Hain08}). This space allows one to consider smooth families of elliptic curves (it is the universal solution of a classification problem), and on the other hand, the singularities of $\mathcal{M}_{1,1}$ correspond to elliptic curves with special symmetries, namely, those associated with the complex lattices with fundamental domains \vspace{1mm}
\begin{center}
\begin{tikzpicture}[scale=1.2]
\draw (0,0) rectangle (1,1);
\fill (0,0) circle (1pt);
\fill (3,0) circle (1pt);
\draw (3,0) -- (4,0) -- (4.5,0.866) -- (3.5,0.866) -- (3,0);
\draw [->] (0.3,0) arc (0:90:0.3);
\draw [->] (3.3,0) arc (0:60:0.3);
\draw (0.35,0.35) node {$\frac{\pi}{2}$};
\draw (3.4,0.3) node {$\frac{\pi}{3}$};
\end{tikzpicture}
\end{center}\vspace{1mm}

\noindent On the other hand, our \emph{mod-Gaussian moduli spaces} are compact topological spaces whose points $m$ parametrise certain random models $(M_n(m))_{n \in \N}$: random graphs, random permutations, \emph{etc.} All these models have asymptotic properties (central limit theorem, local limit theorem, \emph{etc.}) that vary continuously with the parameter $m$. Thus, the introduction of the mod-Gaussian moduli space $(\Mcal,\obs_{\mathfrak{M}})$ allows one to understand generic properties of these models, and to consider smooth families of such models. Moreover, the singularities of $\Mcal$ correspond to random models with special symmetries that force the variances to be smaller than for the other random models. As evoked in the previous remark, the study of these singularities can be extremely interesting: indeed, when $m$ is a singular parameter, one can try to establish the mod-Gaussian behavior of an adequate non-generic renormalisation of the observables of the models $M_n(m)$. The fact that a parameter $m \in \Mcal$ with additional symmetries can correspond to vanishing variances and to singularities in the fluctuations of the model $(M_n(m))_{n \in \N}$ is to be compared with the fact that a singular point of a moduli space in algebraic geometry usually parametrises a manifold with additional (non-generic) symmetries. This analogy explains why we borrowed the terminology of moduli space from algebraic geometry.
\end{remark}

\bigskip

\bibliographystyle{alpha}
\bibliography{graphon.bib}


\end{document}